\documentclass[12pt,oneside,openany]{report}

\usepackage[a4paper,width=150mm,headheight=110pt,top=25mm,bottom=25mm]{geometry}
\usepackage[utf8]{inputenc}
\usepackage{listings}
\usepackage{graphicx}
\usepackage{pgffor} 
\usepackage{float}
\usepackage{graphics} 
\usepackage{fancyhdr}

\usepackage[square,numbers,authoryear]{natbib}
\usepackage{color}
\usepackage[dvipsnames]{xcolor}

\definecolor{mycolor}{HTML}{163680}

\usepackage[hyphens]{url}
\usepackage{hyperref}
\hypersetup{colorlinks = true,
            linkcolor = mycolor,
            urlcolor  = mycolor,
            citecolor = mycolor,
            anchorcolor = mycolor}

\usepackage[hypcap=true,font={small,it}]{caption}
\usepackage{microtype, babel}
\usepackage{amsfonts, amssymb, amsmath}
\usepackage{amsthm}
\usepackage{mathtools}
\usepackage{booktabs}
\usepackage{colortbl}
\usepackage{outlines}
\usepackage{subcaption}
\usepackage{acronym}
\usepackage{optidef}
\usepackage[ruled]{algorithm2e}
\usepackage{tikz}
\usepackage{pdfpages}
\usetikzlibrary{positioning}
\usepackage{eth-template__ETHlogo}
\usepackage[capitalise, noabbrev]{cleveref}
\newcommand{\obar}[1]{\mkern 1.5mu\overline{\mkern-1.5mu#1\mkern-1.5mu}\mkern 1.5mu}
\newcommand{\otild}[1]{\mkern 1.5mu\tilde{\mkern-1.5mu#1\mkern-1.5mu}\mkern 1.5mu}
\newcommand{\sources}{S} 
\newcommand{\sinks}{D} 
\newcommand{\sink}{d}
\newcommand{\dur}{T} 
\newcommand{\csupplies}{\varsigma} 
\newcommand{\dfuns}{\eta} 
\newcommand{\demands}{\kappa} 
\newcommand{\TIntervals}{\varrho} 
\newcommand{\prodcpus}{\psi} 
\newcommand{\supplies}{\sigma} 
\newcommand{\caps}{u} 
\newcommand{\outfuns}{\gamma} 
\newcommand{\InArcs}[1]{\delta_{#1}^{-}} 
\newcommand{\OutArcs}[1]{\delta_{#1}^{+}} 
\newcommand{\vin}[1]{y_{\InArcs{#1}}} 
\newcommand{\vout}[1]{x_{\OutArcs{#1}}} 
\newcommand{\val}[2][]{\nu_{#1}(#2)} 
\newcommand{\OPT}[1][]{p^*_{#1}} 
\newcommand{\bits}[1]{\text{bits}[#1]} 
\newcommand{\InsMap}{\mathcal{A}} 
\newcommand{\SolMap}{\mathcal{B}} 
\newcommand{\ssource}{{s^*}} 
\newcommand{\costs}{c} 
\newcommand{\clin}{\alpha} 
\newcommand{\cquad}{\beta} 
\newcommand{\outscaler}{\mu} 
\newcommand{\MaxGradNorm}{\Gamma}
\newcommand{\MaxPointNorm}{R}
\newcommand{\InitialGuarantor}{Q}
\newcommand{\IneqPerturb}{\omega}
\newcommand{\EqPerturb}{\rho}
\newcommand{\given}{\ \big| \ }
\newcommand{\FeasSetHardCP}{\mathcal{F}_3}
\newcommand{\FeasSetCP}{\mathcal{F}_4}
\newcommand{\OptSetHardCP}{\mathcal{O}_3}
\newcommand{\OptSetCP}{\mathcal{O}_4}
\newcommand{\CumPurger}{H}
\newcommand{\MinTIntervalLen}{\abs{\TIntervals_{min}}}
\newcommand{\vini}[1]{y_{\InArcs{#1}i}}
\newcommand{\vouti}[1]{x_{\OutArcs{#1}i}}
\newcommand{\MinGrad}{\iota}
\newcommand{\MaxDemandsSlope}{\Omega}
\newcommand{\PieceLengths}{l}
\newcommand{\slack}{^*s}
\newcommand{\AuxSpace}{B_\varepsilon^+}
\newcommand{\AuxFunI}{f_\varepsilon}
\newcommand{\AuxFunII}{f_\varepsilon^+}
\newcommand{\AuxFunIII}{F_\varepsilon^+}
\newcommand{\AuxConst}{\Omega(\varepsilon)}
\newcommand{\FunctionsBound}{W}
\newcommand{\FunOPT}{\tilde{p}^*} 
\newcommand{\DeltaOPT}[1]{\Delta_{#1}} 
\newcommand\at[2]{\left.#1\right|_{#2}}
\DeclarePairedDelimiter\ceil{\lceil}{\rceil}

\DeclareMathOperator*{\argmax}{arg\,max}
\DeclareMathOperator*{\argmin}{arg\,min}
\newcommand{\BigOh}{\mathcal{O}}

\renewcommand{\a}{\text{ and }}
\newcommand{\Id}{\text{Id}} 
\newcommand{\domain}[1]{\text{domain}(#1)} 
\newcommand{\image}[1]{\text{Im}(#1)} 
\newcommand{\interior}[1]{\text{int}(#1)} 
\newcommand{\Bern}[1]{\text{Bern}(#1)} 
\newcommand{\Bin}[1]{\text{Bin}(#1)} 
\newcommand{\Unif}[1]{\text{Unif}(#1)} 
\newcommand{\Beta}[1]{\text{Beta}(#1)} 
\newcommand{\BetaConstant}[1]{\text{B}(#1)} 
\newcommand{\N}{\mathbb{N}} 
\newcommand{\R}{\mathbb{R}} 
\newcommand{\extR}{\obar{\R}} 
\DeclarePairedDelimiter\abs{\lvert}{\rvert} 
\DeclarePairedDelimiter\norm{\lVert}{\rVert} 
\newtheorem{theorem}{Theorem}
\newtheorem{lemma}{Lemma}
\newtheorem{proposition}{Proposition}
\newtheorem{corollary}{Corollary}
\newtheorem{definition}{Definition}
\newtheorem{problem}{Problem}
\newtheorem{program}{Program}
\newtheorem{example}{Example}
\crefname{problem}{Problem}{Problems}
\crefname{program}{Program}{Programs}

\tikzset{
    source/.style={circle, draw=green!60, fill=green!5, very thick, minimum size=7mm},
    sink/.style={circle, draw=red!60, fill=red!5, very thick, minimum size=7mm},
    trans/.style={rectangle, draw=gray!60, fill=gray!5, very thick, minimum size=5mm},
    bus/.style={circle, draw=blue!60, fill=blue!5, very thick, minimum size=7mm},
}

\begin{document}

\pagenumbering{roman}

\begin{titlepage}
    \noindent\ETHlogo[2in]\\
    \vspace{0pt plus 0.7fill}
        \begin{center}
            {\LARGE \textbf{Solving Min-Cost Concave Generalized Dynamic Flows and Approximating Dynamic Optimal Power Flows}}
            
            \vspace{2.5cm}
            Master's Thesis\\
            \vspace{0.2cm}
            \href{https://www.linkedin.com/in/jacob-h-rothschild/}{Jacob H. Rothschild}\\
            \vspace{0.2cm}
            September 2023
                    
        \end{center}
        \vfill\par
        \begin{flushright}
            Spring Semester 2023\\
            \vspace{0.2cm}
            Advisor: Dr. Laura Vargas Koch\\
            \vspace{0.2cm}
            Professor: Prof. Dr. Rico Zenklusen\\
            \vspace{0.2cm}
            Institute for Operations Research\\
            \vspace{0.2cm}
            Department of Mathematics, ETH Z\"urich
        \end{flushright}
\end{titlepage}

\acresetall
\begin{abstract}        
    Assuming power travels instantaneously, can be steered by us, and is lost quadratically in each power line, the \acl{DOPF} problem simplifies to a \ac{MCDGFWQL} problem. As this is a special case of the \ac{MCCDGF} problem, we derive both general results for the \ac{MCCDGF} problem and stronger results for the \ac{MCDGFWQL} problem.
    
    Our main contribution is a \acl{FPTAS} for a mild restriction of the \ac{MCDGFWQL} problem. We also implement and benchmark a slight modification of this algorithm, finding it to be efficient in practice, with a slightly superlinear and possibly subquadratic dependence of execution time on the edge count of the input graph.
    
    Our secondary contributions are the derivation of a reduction from the \ac{MCCDGF} problem to a \acl{CP} and a corresponding sensitivity analysis of the \ac{MCCDGF} problem. We also provide a brief comparison of our \acl{DOPF} simplification and the classic direct current simplification and point out some interesting directions for future research.
   \end{abstract}
\acresetall

\chapter*{Acknowledgements}
I would like to express my sincerest gratitude to Dr.\ Laura Vargas Koch for excellent supervision throughout this thesis and thank her for all of our interesting discussions.

\hypersetup{linkcolor=black}
\tableofcontents
\listoftables
\begingroup
    \let\clearpage\relax
    \listofalgorithms
\endgroup
\pagebreak
\hypersetup{linkcolor=mycolor}

\pagenumbering{arabic}

\acresetall
\chapter{Introduction}\label{Ch:Introduction}
    \pagenumbering{arabic}
    To properly function, substantial parts of modern society require electricity to be instantaneously available at all times \citep{meier2006electric}. It is hence critical that we never fail to satisfy all electricity demands. As our total electricity production leads to significant direct and indirect costs \citep{nrc2010hidden, iea2023co2}, we however want to not just satisfy these demands but to also minimize the cost with which we do so. This is the problem of \ac{DOPF} \citep{gill2014dynamic}. As electrical power must abide by the laws of physics, the \ac{DOPF} problem is a non-convex optimization problem, and it is in general strongly NP-hard \citep{bienstock2019strong}. It is therefore worthwhile to look at ways to simplify it.

One such simplification, as used in \citep{hernandez2023pollution}, is to assume power travels instantaneously, incurs a quadratic loss in each power line, and can be freely steered through the power grid. The \ac{DOPF} problem then becomes a min-cost dynamic generalized flow problem with quadratic losses. To the best of our knowledge, such flow problems have not been studied.

Hence, the main contribution of this thesis is to, under a set of mild assumptions, develop a practically efficient \acf{FPTAS} for such flow problems. More specifically, \cref{Ch:Solution} develops this \ac{FPTAS}, \cref{Ch:Sensitivity} studies the sensitivity of the flow problem, \cref{Ch:Implementation} implements and benchmarks the \ac{FPTAS} in practice, while \cref{Ch:Extensions} looks at some generalizations.
    \section{Background}\label{S:Background} We will in this section give a quick summary of the relevant definitions leading up to dynamic generalized flows. For a more thorough survey, the interested reader is referred to \cite{aronson1989survey}.

In the context of flows, the most basic construct is the static $\sources$-$\sinks$ flow, where flow is produced in $\sources$, consumed in $\sinks$, and preserved everywhere else.

\begin{definition}[Static flow]
    Given a directed graph $G=(V,A)$ with non-negative arc capacities $\caps \in \R_{\geq 0}^{A}$ and $\sources,\sinks \subseteq V$, a static $\sources$-$\sinks$ flow is a set $x \in \R_{\geq 0}^{A}$ such that
    \begin{equation*}
    \begin{array}{r@{\;}ll}
        x_a &\leq \caps_a, \quad &\forall{a \in A}, \\
        \vout{s} - x_{\InArcs{s}} &\geq 0, \quad &\forall{s \in \sources}, \\
        \vout{\sink} - x_{\InArcs{\sink}} &\leq 0, \quad &\forall{\sink \in \sinks}, \\
        \vout{v} - x_{\InArcs{v}} &= 0, \quad &\forall{v \in V \setminus (\sources \cup \sinks)}.
    \end{array}
    \end{equation*}
\end{definition}

If $\sources=\{s\}$ and $\sinks=\{\sink\}$, then this is also called a static $s$-$d$ flow, while if $\sources$ and $\sinks$ are clear from context it is simply called a static flow. If, moreover, every source $s \in \sources$ has some non-negative capacity $\supplies_s$, every sink $\sink \in \sinks$ has some non-negative demand $\demands_s$, and we are given some cost function $\prodcpus$, then we can define the min-cost static $\sources$-$\sinks$ problem.

\begin{problem}[Min-cost static flow]
    Given: A directed graph $G=(V,A)$ with arc capacities $\caps \in \R_{\geq 0}^{A}$, cost function $\prodcpus: \R_{\geq 0}^{A} \to \R$, source supplies $\supplies \in \R_{\geq 0}^{\sources}$, and sink demands $\demands \in \R_{\geq 0}^{\sinks}$.

    Task: Find a static $\sources$-$\sinks$ flow $x \in \R_{\geq 0}^{A}$ of minimum $\prodcpus(x)$ that satisfies
    \begin{equation*}
    \begin{array}{r@{\;}ll}
        \vout{s} - x_{\InArcs{s}} &\leq \supplies_s, \quad &\forall{s \in \sources}, \\
        \vout{\sink} - x_{\InArcs{\sink}} &= \demands_\sink, \quad &\forall{\sink \in \sinks}.
    \end{array}
    \end{equation*}
\end{problem}

In general, if $f_0$ is the objective function of some problem instance $(f)$ and $x$ is some point in $\domain{f_0}$, then we denote
\begin{equation}
\begin{split}
    \val[(f)]{x} &:= f_0(x), \\
    \OPT[(f)] &:= \min\{\val[(f)]{x'} \given x' \text{ satisfies all the constraints of } (f)\}.
\end{split}
\end{equation}
Alternatively, if the instance is clear from context, we will sometimes drop the subscript.

In a generalized flow, the flow coming out of an arc $a$ need not equal the flow going into $a$, but is instead determined by a function $\outfuns_a$.

\begin{definition}[Generalized flow]
    Given a directed graph $G=(V,A)$ with non-negative arc capacities $\caps \in \R_{\geq 0}^{A}$, out-flow functions $\outfuns: (\R_{\geq 0} \to \R_{\geq 0})^{A}$ and $\sources,\sinks \subseteq V$, a (static) generalized $\sources$-$\sinks$ flow is a tuple $z=(x,y)$ with $x,y \in \R_{\geq 0}^{A}$ such that
    \begin{equation*}
    \begin{array}{r@{\;}ll}
        y_a &= \outfuns_a(x_a), \quad &\forall{a \in A}, \\
        x_a,y_a &\leq \caps_a, \quad &\forall{a \in A}, \\
        \vout{s} - \vin{s} &\geq 0, \quad &\forall{s \in \sources}, \\
        \vout{\sink} - \vin{\sink} &\leq 0, \quad &\forall{\sink \in \sinks}, \\
        \vout{v} - \vin{v} &= 0, \quad &\forall{v \in V \setminus (\sources \cup \sinks)}.
    \end{array}
    \end{equation*}
\end{definition}

This allows us to trivially generalize the min-cost static flow problem to the min-cost (static) generalized flow problem.

\begin{problem}[Min-cost generalized flow]
    Given: A directed graph $G=(V,A)$ with arc capacities $\caps \in \R_{\geq 0}^{A}$, out-flow functions $\outfuns: (\R_{\geq 0} \to \R_{\geq 0})^{A}$, cost function $\prodcpus: \R_{\geq 0}^{A} \to \R$, source supplies $\supplies \in \R_{\geq 0}^{\sources}$, and sink demands $\demands \in \R_{\geq 0}^{\sinks}$.

    Task: Find a (static) generalized $\sources$-$\sinks$ flow $z=(x,y)$ with $x,y \in \R_{\geq 0}^{A}$ of minimum $\prodcpus(z)$ that satisfies
    \begin{equation*}
    \begin{array}{r@{\;}ll}
        \vout{s} - \vin{s} &\leq \supplies_s, \quad &\forall{s \in \sources}, \\
        \vout{\sink} - \vin{\sink} &= \demands_\sink, \quad &\forall{\sink \in \sinks}.
    \end{array}
    \end{equation*}
\end{problem}

So far we have looked at static flows, which can be thought of as either capturing an instant in time or some infinite-horizon stable state. If we have cumulative capacities and/or time-dependent network parameters, then the temporal component becomes non-trivial. Hence, it then makes sense to make our flows functions of time. To avoid measure-theoretic details, we demand that these functions are piecewise uniformly continuous with non-zero piece lengths. To avoid unnecessary complexity, we will here limit ourselves to only allowing the demand to depend on time and to the only cumulative capacities being on flow generation at sources. Generalizing generalized flows in this way then gives us dynamic generalized flows.

\begin{definition}[Dynamic generalized flow]
    Given a duration $\dur$ and a directed graph $G=(V,A)$ with non-negative arc capacities $\caps \in \R_{\geq 0}^{A}$, out-flow functions $\outfuns: (\R_{\geq 0} \to \R_{\geq 0})^{A}$, $\sources,\sinks \subseteq V$, and cumulative source capacities $\csupplies \in \R_{\geq 0})^{\sources}$, a dynamic generalized $\sources$-$\sinks$ flow is a tuple $z=(x,y)$ where $x,y \in ([0,\dur] \to \R_{\geq 0})^{A}$ are piecewise uniformly continuous with non-zero piece lengths and satisfy
    \begin{equation*}
    \begin{array}{r@{\;}ll}
        y_a(t) &= \outfuns_a(x_a(t)), \quad &\forall{(a,t) \in A \times [0,\dur]}, \\
        x_a(t),y_a(t) &\leq \caps_a, \quad &\forall{(a,t) \in A \times [0,\dur]}, \\
        \int_0^{\dur}\vout{s}(t) - \vin{s}(t) dt &\leq \csupplies_s, \quad &\forall{s \in \sources}, \\
        \vout{s}(t) - \vin{s}(t) &\geq 0, \quad &\forall{(s,t) \in \sources \times [0,\dur]}, \\
        \vout{\sink}(t) - \vin{\sink}(t) &\leq 0, \quad &\forall{(\sink,t) \in \sinks \times [0,\dur]}, \\
        \vout{v}(t) - \vin{v}(t) &= 0, \quad &\forall{(v,t) \in (V \setminus (\sources \cup \sinks)) \times [0,\dur]}.
    \end{array}
    \end{equation*}
\end{definition}

This then allows further generalizing the min-cost generalized flow problem to the min-cost dynamic generalized flow problem.

\begin{problem}[Min-cost dynamic generalized flow]
    Given: A duration $\dur$ and a directed graph $G=(V,A)$ with arc capacities $\caps \in \R_{\geq 0}^{A}$, out-flow functions $\outfuns: (\R_{\geq 0} \to \R_{\geq 0})^{A}$, cost functions $\costs: ([0,\dur] \to \R_{\geq 0})^{A} \to \R$, production rate capacities $\supplies \in \R_{\geq 0}^{\sources}$, cumulative production capacities $\csupplies \in \R_{\geq 0})^{\sources}$, and sink demand functions $\dfuns \in ([0,\dur] \to \R_{\geq 0})^{\sinks}$.

    Task: Find a dynamic generalized $\sources$-$\sinks$ flow $z=(x,y)$ of minimum $\costs(z)$ that satisfies
    \begin{equation*}
    \begin{array}{r@{\;}ll}
        \vout{s}(t) - \vin{s}(t) &\leq \supplies_s, \quad &\forall{(s,t) \in \sources \times [0,\dur]}, \\
        \vout{\sink}(t) - \vin{\sink}(t) &= \demands_\sink, \quad &\forall{(\sink,t) \in \sinks \times [0,\dur]}.
    \end{array}
    \end{equation*}
\end{problem}

Given a similar graph and temporal aspect, we can also define a production schedule.

\begin{definition}[Production schedule]
    Given a duration $\dur$ and a directed graph $G=(V,A)$ with $\sources,\sinks \subseteq V$, production rate capacities $\supplies \in \R_{\geq 0}^{\sources}$, and cumulative production capacities $\csupplies \in \R_{\geq 0})^{\sources}$, a production schedule is a set of functions $p \in ([0,\dur] \to \R_{\geq 0})^{\sources}$ that satisfies
    \begin{equation*}
    \begin{array}{r@{\;}ll}
        p_s(t) &\leq \supplies_s, \quad &\forall{(s,t) \in \sources \times [0,\dur]}, \\
        \int_0^{\dur}p_s(t)dt &\leq \csupplies_s, \quad &\forall{s \in \sources}.
    \end{array}
    \end{equation*}
\end{definition}

Finally, if $G$ is instead an undirected graph $(V,E)$, then a flow should be free to choose its direction on each edge, but not to flow in both directions at once. To capture this, we hence define a flow on an undirected graph $G=(V,E)$ as a flow on its directed version $G'=(V,A)$ that does not utilize any two antiparallel arcs of $G'$ at once. By $G'$ being the direction version of $G$, we mean that $A$ contains two antiparallel arcs for every edge in $E$.

    \section{Problem statements}\label{S:Problem} Real-world power grids are highly complex systems, but can for our purposes be modelled as an undirected simple graph $G=(V,E)$, where every node has some supply and demand and an edge $\{u,v\} \in E$ represents all the power lines directly connecting nodes $u$ and $v$ \citep{meier2006electric}. As previously mentioned, we make the simplifying assumptions that power travels instantaneously and with quadratic losses and that we can control how the outflow of a node is distributed among its incident edges. As $\sources=\sinks=V$, we can then model the \ac{DOPF} problem as a min-cost dynamic generalized $V$-$V$ flow problem with quadratic losses. We capture these quadratic losses by the set of parameters $r$, assumed to be constant over time:
\begin{equation}
    y_e(t) = \outfuns_e(x_e(t)) := x_e(t) - r_e x_e(t)^2, \ \forall t.
\end{equation}

In real power grids, every transmission line has a thermal rating upper bounding how much current can safely pass through it for an extended period of time without the line overheating. To capture this with our model, every edge $e \in E$ is given some capacity $\caps_e$, which we assume to be constant over time and small enough that $\outfuns_e$ is strictly increasing on $[0,\caps_e]$. To capture production costs and capacities, we assume that every node $v \in V$ either has a limit $\supplies_v$ on its production rate or a limit $\csupplies_v$ on its cumulative production, but not both. Such a node $v$ then has a marginal production cost function $\prodcpus_v$, which is positive, non-decreasing, constant over time, and piecewise constant with finitely many pieces. If $v$ has a limit on its production rate, then we assume that its marginal production cost also depends on its production rate, with $\domain{\prodcpus_v}=[0,\supplies_v]$. If $v$ instead has a limit on its cumulative production, then we correspondingly assume its marginal production cost at some time $t$ depends on its cumulative production in $[0,t]$ and that $\domain{\prodcpus_v}=[0,\csupplies_v]$. Finally, every node $v \in V$ is assumed to have some demand function $\dfuns_v$. $\dfuns$ is a function of time, and is assumed to be piecewise constant with non-zero piece lengths. This then allows us to define what we will call an \ac{UPGG} as a
\begin{equation}
    G = (V,E,r,\caps,\supplies,\csupplies,\dfuns,\prodcpus)
\end{equation}
where all parameters follow the assumptions just described. By relying on this definition, \cref{P:ConstantOP} captures our problem statement.

\begin{problem}[Simplified \ac{DOPF} with piecewise constant demands]\label{P:ConstantOP}
    Given: a duration $\dur\in\R_+$, and an \ac{UPGG}
    \begin{equation*}
        G=(V,E,r,\caps,\supplies,\csupplies,\dfuns,\prodcpus),
    \end{equation*}
    with
    \begin{equation*}
        \dfuns_\sink: [0,\dur] \ni t \mapsto
        \begin{cases}
            \demands_{\sink 1}, &\text{if } t \in \TIntervals_1 \subseteq [0,\dur]\\
            \vdots &\\
            \demands_{\sink k}, &\text{if } t \in \TIntervals_k \subseteq [0,\dur]\\
        \end{cases}
        \qquad \forall{\sink \in \sinks}.
    \end{equation*}
    
    Task: find a production schedule and a corresponding generalized dynamic $V$-$V$ flow over $[0,\dur]$ such that all demands are satisfied and the total production cost is minimized.
\end{problem}

\subsection{Simplifying this problem}

While \cref{P:ConstantOP} captures our problem and assumptions well, it can be cumbersome to analyze. Hence, we will instead define and analyze \acp{SPGG}. Compared to an \ac{UPGG}, an \ac{SPGG}
\begin{equation}
    G = (V,A,r,\caps,\csupplies,\dfuns,\prodcpus)
\end{equation}
has the undirected edges $E$ replaced by directed arcs $A$, all supplies and marginal production costs in terms of cumulative production, and $\sources,\sinks \subsetneq V$ with $\sources \cap \sinks = \emptyset$. Moreover, without loss of generality, we for the \ac{SPGG} assume that every node is reachable from $\sources$ and that every source has no in-arcs and every sink has no out-arcs. For our analysis we then essentially redefine \cref{P:ConstantOP} on \acp{SPGG}, arriving at \cref{P:SimpleConstant}.

\begin{problem}[Min-cost dynamic generalized $\sources$-$\sinks$ flow with quadratic losses]\label{P:SimpleConstant}
    Given: a duration $\dur\in\R_+$, and an \ac{SPGG}
    \begin{equation*}
        G=(V,A,r,\caps,\csupplies,\dfuns,\prodcpus),
    \end{equation*}
    with
    \begin{equation*}
        \dfuns_\sink: [0,\dur] \ni t \mapsto
        \begin{cases}
            \demands_{\sink 1}, &\text{if } t \in \TIntervals_1 \subseteq [0,\dur]\\
            \vdots &\\
            \demands_{\sink k}, &\text{if } t \in \TIntervals_k \subseteq [0,\dur]\\
        \end{cases}
        \qquad \forall{\sink \in \sinks}.
    \end{equation*}
    
    Task: find a production schedule and a corresponding generalized dynamic $\sources$-$\sinks$ flow over $[0,\dur]$ such that all demands are satisfied and the total production cost is minimized.
\end{problem}

As \cref{P:ConstantOP} and \cref{P:SimpleConstant} are on fundamentally different graphs, analyzing \cref{P:SimpleConstant} instead of \cref{P:ConstantOP} may at first seem like a strange choice. As we will see in \cref{SS:SolvingOriginal}, the results we obtain for \cref{P:SimpleConstant} however turn out to be highly applicable to \cref{P:ConstantOP}.

\subsection{Generalizing this simplification}

To achieve a higher generality in our analysis, much of it will be performed on a more general graph which we will call a \ac{GPGG}. Compared to an \ac{SPGG}, a \ac{GPGG}
\begin{equation}
    G = (V,A,\outfuns,\caps,\csupplies,\dfuns,\prodcpus)
\end{equation}
only requires that $\prodcpus$ is non-decreasing and piecewise twice differentiable with finitely many such pieces and that $\outfuns_a$ is concave, non-negative, strictly increasing and upper bounded by $\Id$. \cref{P:GeneralizedConstant} captures the naive generalization of \cref{P:SimpleConstant} to \acp{GPGG}.

\begin{problem}[Min-cost concave dynamic generalized flow] \label{P:GeneralizedConstant}
    Given: a duration $\dur\in\R_+$, and a \ac{GPGG}
    \begin{equation*}
        G=(V,A,\outfuns,\caps,\csupplies,\dfuns,\prodcpus),
    \end{equation*}
    with
    \begin{equation*}
        \dfuns_\sink: [0,\dur] \ni t \mapsto
        \begin{cases}
            \demands_{\sink 1}, &\text{if } t \in \TIntervals_1 \subseteq [0,\dur]\\
            \vdots &\\
            \demands_{\sink k}, &\text{if } t \in \TIntervals_k \subseteq [0,\dur]\\
        \end{cases}
        \qquad \forall{\sink \in \sinks}.
    \end{equation*}
    
    Task: find a production schedule and a corresponding generalized dynamic $\sources$-$\sinks$ flow over $[0,\dur]$ such that all demands are satisfied and the total production cost is minimized.
\end{problem}

Please note that the assumptions used everywhere that only $\dfuns$ is time-dependent is for notational convenience rather than necessity. All the results, reductions and arguments of this thesis trivially extend to the case where also $\outfuns$ and $\caps$ are piecewise constant over time with non-zero piece lengths. However, such a generalization would make our notation more cumbersome, which is why we make the stronger assumption of constancy over time.

\subsection{Formulating Problem \ref{P:GeneralizedConstant} as a mathematical program}

We denote by $x: [0,\dur] \to \R_{\geq0}^A$ and $y: [0,\dur] \to \R_{\geq0}^A$ the in-flow and out-flow rates of all arcs respectively, and seek to formalize \cref{P:GeneralizedConstant} in terms of a mathematical program on $x$ and $y$. From the definition of generalized flows, we directly get the constraint
\begin{equation}\label{E:InOutRelationship}
    y_a(t)=\outfuns_a(x_a(t)),\ \forall{(a,t) \in A \times [0,\dur]}.
\end{equation}

Similarly, strong flow conservation is equivalent to
\begin{equation}
    \vin{v}=\vout{v},\ \forall{v \in V \setminus (S \cup \sinks)},
\end{equation}
where
\begin{align}
    \vin{v}&:=\sum_{a\in\InArcs{v}}{y_a} \\
    \vout{v}&:=\sum_{a\in\OutArcs{v}}{x_a}.
\end{align}

Thus, if we additionally assume that
\begin{equation}\label{E:xyDomain}
    \begin{cases}
        0 \leq x_a(t) \leq \caps_a\\
        0 \leq y_a(t) \leq \caps_a
    \end{cases} \quad
    \forall{(a,t) \in A \times [0,\dur]},
\end{equation}
and that every $x_a$ and $y_a$ is piecewise uniformly continuous on $[0,\dur]$ with non-zero piece lengths, then $x,y$ is a generalized flow on $G$. As
\begin{equation*}
    \outfuns_a \leq \Id,\ \forall{a \in A},
\end{equation*}
we know by \eqref{E:InOutRelationship} that
\begin{equation*}
    y_a \leq x_a,\ \forall{a \in A},
\end{equation*}
thus that \eqref{E:xyDomain} is equivalent to
\begin{equation}
    \begin{cases}
        y_a(t) \geq 0 \\
        x_a(t) \leq \caps_a
    \end{cases} \quad
    \forall{(a,t) \in A \times [0,\dur]}.
\end{equation}

Finally, satisfying the demands is equivalent to
\begin{equation}
    \vin{\sink}=\dfuns_\sink,\ \forall{\sink\in\sinks},
\end{equation}
while the objective function is the sum of cumulative costs incurred at all sources:
\begin{equation}
    \sum_{s \in \sources}\int_0^{\int_0^{\dur}{\vout{s}(t)dt}}{\prodcpus_s(\alpha)d\alpha}.
\end{equation}
For this to be well-defined, we need that
\begin{equation}
    \int_0^{\dur}{\vout{s}(t)dt} \leq \csupplies_s,\ \forall{s \in \sources}.
\end{equation}

We have thus arrived at our first mathematical formulation of \cref{P:GeneralizedConstant}, in the form of \cref{M:OP}.

\begin{program}[Min-cost concave dynamic generalized flow]\label{M:OP}
Given: $G$ as in \cref{P:GeneralizedConstant}.
\begin{mini*}{x,y}{\sum_{s \in \sources}\int_0^{\int_0^{\dur}{\vout{s}(t)dt}}{\prodcpus_s(\alpha)d\alpha}}{}{}
    \addConstraint{y_a}{=\outfuns_a(x_a),}{a \in A}
    \addConstraint{x_a}{\leq \caps_a,}{a \in A}
    \addConstraint{y_a}{\geq 0,}{a \in A}
    \addConstraint{\vin{\sink}}{=\dfuns_\sink,}{\sink\in\sinks}
    \addConstraint{\vin{v}}{=\vout{v},}{v \in V \setminus (\sources \cup \sinks)}
    \addConstraint{\int_0^{\dur}{\vout{s}(t)dt} }{\leq \csupplies_s,}{s \in \sources}
    \addConstraint{x_a,y_a\text{ p.w. unif. }}{\text{cont. with non-null pieces}, \quad}{a \in A}
\end{mini*}
\end{program}

Note how all but the last set of inequalities are between functions over $[0,\dur]$, and thus have to hold for every $t\in[0,\dur]$ individually.

\subsection{Outlook}

In the next two sections we will under a mild assumption derive a \acf{FPTAS} for \cref{P:ConstantOP}. We will do this in three steps:
\begin{outline}[enumerate]
    \1 First, we will show that \cref{M:OP}, hence \cref{P:GeneralizedConstant}, reduces to a \ac{CP}, and that this is a \ac{QCQP} in the special case of \cref{P:SimpleConstant}.
    \1 Then, we will show that this implies that \cref{P:ConstantOP} also reduces to a convex \ac{QCQP}.
    \1 Finally, we will finish the proof by combining \ac{QCQP} theory and flow-based sensitivity analysis to construct an \ac{FPTAS} for our \ac{QCQP}.
\end{outline}

To be exact, whenever we in this thesis write
\begin{center}
    Problem P reduces to Problem P',
\end{center}
this should be interpreted as
\begin{center}
    There exists a polynomial-time strict approximation-preserving\\
    reduction from Problem P to Problem P'.
\end{center}

\begin{definition}[Polynomial-time strict approximation-preserving reduction]\label{D:Reduction}
    Given minimization Problems P and P', a polynomial-time strict approximation-preserving reduction $(\InsMap, \SolMap)$ from Problem P to Problem P' is a pair of functions satisfying the following conditions for any feasible instance $(f)$ of Problem P:
    \begin{outline}
        \1 $(g):=\InsMap[(f)]$ is an instance of Problem P'.
        \1 The execution time of $\InsMap$ on $(f)$ and the bit count of the output $(g)$, denoted by $\bits{(g)}$, are polynomial in $\bits{(f)}$.
        \1 If $x'$ is $(g)$-feasible, then $\SolMap(x')$ is $(f)$-feasible and computed in time polynomial in $\bits{{(g)}}$, and
        \begin{equation*}
            \frac{\val[(f)]{\SolMap(x')}}{\OPT[(f)]} \leq \frac{\val[{(g)}]{x'}}{\OPT[{(g)}]}.
        \end{equation*}
    \end{outline}
\end{definition}

\chapter{Algorithm development}\label{Ch:Solution}
    \section{Reducing to a convex program}\label{S:Reducing} \Cref{M:OP} is a non-convex program with functions on $[0,\dur]$ as variables. Alternatively, one can view it as a program with uncountably infinitely many variables, by having a tuple $z_t=(x_t,y_t)$ for every $t\in[0,\dur]$. In this section we will prove that \cref{M:OP} reduces to a \ac{CP} with scalar variables and that \cref{P:ConstantOP,P:SimpleConstant} reduce to a convex \ac{QCQP}.

\subsection{Permitting waste}\label{SS:OPToFunCP}
We first relax \cref{M:OP} by allowing flow to be wasted in both arcs and nodes. More specifically,
\begin{equation}\label{E:RelaxOP}
\begin{array}{@{}r@{\;}l@{}l@{}l@{\qquad}l}
    y_a &= \outfuns_a(x_a) \quad &\longrightarrow \quad y_a &\leq \outfuns_a(x_a), &\forall{a \in A},\\
    \vin{v} &= \vout{v} \quad &\longrightarrow \quad \vin{v} &\geq \vout{v}, &\forall{v \in V \setminus (\sources \cup \sinks)},\\
    \vin{\sink} &= \dfuns_\sink \quad &\longrightarrow \quad \vin{\sink} &\geq \dfuns_\sink, &\forall{\sink \in \sinks}.
\end{array}
\end{equation}

We are then left with only convex constraints, as captured by \cref{M:FunCP}.

\begin{program}[Functional convex program]\label{M:FunCP}
Given: $G$ as in \cref{P:GeneralizedConstant}.
\begin{mini*}{x,y}{\sum_{s \in \sources}\int_0^{\int_0^{\dur}{\vout{s}(t)dt}}{\prodcpus_s(\alpha)d\alpha}}{}{}
    \addConstraint{y_a}{\leq \outfuns_a(x_a),}{a \in A}
    \addConstraint{x_a}{\leq \caps_a,}{a \in A}
    \addConstraint{y_a}{\geq 0,}{a \in A}
    \addConstraint{\vin{\sink}}{\geq \dfuns_\sink,}{\sink\in\sinks}
    \addConstraint{\vin{v}}{\geq \vout{v},}{v \in V \setminus (\sources \cup \sinks)}
    \addConstraint{\int_0^{\dur}{\vout{s}(t)dt} }{\leq \csupplies_s,}{s \in \sources}
    \addConstraint{x_a,y_a\text{ p.w. unif. }}{\text{cont. with non-null pieces}, \quad}{a \in A}
\end{mini*}
\end{program}

We will now prove some strong relationships between \cref{M:OP} and \cref{M:FunCP}.

\begin{definition}[Flow-carrying path]
    We call a path $P \subseteq A$ on a graph $G=(V,A)$ flow-carrying for the flow $z=(x,y)$ if
    \begin{equation*}
        x_a,y_a > 0, \ \forall{a \in P}.
    \end{equation*}

    If $z$ is clear from context, we simply call $P$ \textit{flow-carrying}.
\end{definition}

\begin{lemma}\label{L:WasteIsReachableFromSource}
    If x is feasible for some \cref{M:FunCP} instance with graph $G=(V,A)$, and there exists a $u \in V$ and a $t \in [0, \dur]$ such that
    \begin{equation*}
        \vin{u}(t) > \vout{t}
    \end{equation*}
    and/or
    \begin{equation*}
        \exists{a \in \OutArcs{u}}: y_a(t) < \outfuns_a(x_a(t)),
    \end{equation*}
    then there exists a flow-carrying $\sources$-$u$ path $P \subseteq A$.
\end{lemma}
\begin{proof}
    For any $V' \subseteq V$, we know by definition that
    \begin{equation}\label{E:GlobalBalance}
        \vin{V'}(t) + \sum_{v \in V'}(\vout{v}(t)-\vin{v}(t)) = \vout{V'}(t) + \sum_{\mathclap{\substack{a \in A \\ u,v \in V'}}}(x_a(t)-y_a(t)) \geq 0, \ \forall{t \in [0,T]}.
    \end{equation}

    Assume that an arc or node has waste at some fixed time $t$. Denote by $R \subseteq V$ the set of all nodes that can reach $u$ along a flow-carrying path at time $t$. Then, by construction, $\vin{R}(t)=0$.

    Assume first that there is waste in $u$. Then $\vout{u}-\vin{u}<0$, so by \eqref{E:GlobalBalance} there must exist some other $s \in R$ such that $\vout{s}-\vin{s}>0$. By the constraints of \cref{M:FunCP} and the definition of $R$, we then know that $s$ is a source that can reach $u$ along a flow-carrying path at time $t$.

    Assume now instead that $u$ has an out-arc $a=(u,v)$ with waste. If $a$ is part of a cycle, then $v \in R$, hence
    \begin{equation*}
        \sum_{\mathclap{\substack{a \in A \\ u,v \in R}}}(x_a(t)-y_a(t)) > 0.
    \end{equation*}

    If $a$ is not part of a cycle, then $v \notin R$, hence $a \in \OutArcs{R}$ and
    \begin{equation*}
        \vout{R}(t) \geq x_a(t) > y_a(t) \geq 0.
    \end{equation*}
    
    Thus, in both cases the right hand side of \eqref{E:GlobalBalance} is positive, and so we must again have an $s \in R$ for which $\vout{s}-\vin{s}>0$, hence a flow-carrying $\sources$-$u$ path at time $t$.
\end{proof}

\begin{lemma}\label{L:OptIsFeasible}
    If $(f)$ is an instance of \cref{M:OP}, $(g)$ is the corresponding \cref{M:FunCP} instance, and $x$ is an optimal solution of $(g)$, then $x$ is a feasible point of $(f)$.
\end{lemma}
\begin{proof}
    Assume $x$ is feasible for $(g)$ but not for $(f)$. Assume first that this is due to waste in some arc $a=(u,v)$ at some time $t \in [0,T]$:
    \begin{equation*}
        y_a(t) < \outfuns_a(x_a(t)).
    \end{equation*}
    
    By \cref{L:WasteIsReachableFromSource} we can for some $s \in \sources$ then find a flow-carrying $s$-$u$ path
    \begin{equation*}
        P=(a_1,a_2,...,a_k)=((s,w_1),(w_1,w_2),...,(w_{k-1},u).
    \end{equation*}
    Hence, we can backtrack along
    \begin{equation*}
        P':=(a_1,...,a_k,a)
    \end{equation*}
    by first reducing $x_a(t)$, then $y_{a_k}(t)$ accordingly, then $x_{a_k}(t)$ accordingly, then $y_{a_{k-1}}(t)$ accordingly, and so on, finally reducing $x_{a_1}$, thus the total production costs at that moment in time. As $t$ is part of a positive-length uniformly continuous component of $(x,y)$, we can, by being conservative enough in these reductions, make this modification for every $t'$ in some positive-length interval $I \subseteq [0,\dur]$ containing $t$. The result is then a solution $x'$ that is feasible for $(g)$ and has strictly lower cumulative production cost, thus objective cost, than $x$. Hence, $x$ is suboptimal for $(g)$. By similar arguments, $x$ is suboptimal if it has waste on some node $v \in V$ at some time $t \in [0,T]$. Hence, we conclude that any optimal solution of $(g)$ must be feasible for $(f)$.
\end{proof}

\begin{corollary}\label{C:OPAndFunCPShareOptima}
    If $(f)$ is an instance of \cref{M:OP} and $(g)$ is the corresponding \cref{M:FunCP} instance, then $(f)$ and $(g)$ have the same set of optimal solutions and $\OPT[(f)]=\OPT[(g)]$.
\end{corollary}
\begin{proof}
    \Cref{M:FunCP} is a relaxation of \cref{M:OP}, so the statement follows from \cref{L:OptIsFeasible}.
\end{proof}

\subsection{Constraining to piecewise constant flows}\label{SS:FunCPToHardCP}
\begin{lemma}\label{L:OptIsConstant}
    If $(g)$ is an instance of \cref{M:FunCP}, then $(g)$ has an optimal solution that is piecewise constant over $[0,\dur]$ and shares its breakpoints with the demands function $\dfuns$ of $(g)$.
\end{lemma}
\begin{proof}
    Assume $z:=(x,y)$ is an optimal solution of $(g)$, hence by \cref{L:OptIsFeasible} a feasible point of $(f)$.

    Recall that $\TIntervals$ denote the maximal constant pieces of $\dfuns$, and let $\tilde{z}$ be a function that on each such $\TIntervals_i$ maps every $t$ to the average of $z$ over that interval:
    \begin{equation}
        \tilde{z}: [0,\dur] \ni t \mapsto \left\{
        \begin{array}{@{}r@{\;}ll}
            \obar{z_1} &:= \frac{1}{\abs{\TIntervals_1}}\int_{\TIntervals_1}{z(\tau)d\tau}, &\text{if } t \in \TIntervals_{1}\\
            \vdots &&\\
            \obar{z_k} &:= \frac{1}{\abs{\TIntervals_k}}\int_{\TIntervals_k}{z(\tau)d\tau}, &\text{if } t \in \TIntervals_{k}
        \end{array} \right.
    \end{equation}
    
    By definition,
    \begin{equation}
        \int_0^{\int_0^{\dur}{\vout{s}(t)dt}}{\prodcpus_s(\alpha)d\alpha} = \int_0^{\int_0^{\dur}{\tilde{x}_{\OutArcs{s}}(t)dt}}{\prodcpus_s(\alpha)d\alpha} = \int_0^{\sum_{i=1}^{k}\abs{\TIntervals_{i}}\obar{\vout{s}}}{\prodcpus_s(\alpha)d\alpha}, \ \forall{s \in \sources},
    \end{equation}
    hence
    \begin{equation*}
        \val[(g)]{z} = \val[(g)]{\tilde{z}} = \OPT[(g)].
    \end{equation*}
    Moreover, since $z$ is feasible and the integral operator is linear, $\tilde{z}$ fulfills all the linear constraints of \cref{M:FunCP}. Finally, by $y=\outfuns(x)$, concavity of $\outfuns$, and Jensen's inequality:
    \begin{equation}
        \forall{(a,i) \in A \times [k]}: \obar{y}_{ai} = \frac{1}{\abs{\TIntervals_{i}}}\int_{\TIntervals_{i}}{\outfuns_a(x_{ai}(\tau))} d\tau \leq \outfuns_a(\obar{x}_{ai}), \quad \obar{z}=(\obar{x},\obar{y}).
    \end{equation}

    We hence conclude that $\tilde{z}$ is a feasible point of $(g)$ of optimal value, hence an optimal solution of $(g)$.
\end{proof}

As such a piecewise constant dynamic flow $\tilde{z}$ is entirely determined by the vector of static flows
\begin{equation}
    \obar{z}=(z_1,z_2,...,z_k),
\end{equation}
\cref{M:FunCP} reduces to finding an optimal such $\obar{z}$, captured by \cref{M:HardCP}.

\begin{program}[Hard convex program]\label{M:HardCP}
Given: $G$ as in \cref{P:GeneralizedConstant}.
\begin{mini*}{x,y}{\sum_{s \in \sources}{\int_0^{\sum_{i=1}^{k}\abs{\TIntervals_{i}}\vouti{s}}{\prodcpus_s(\alpha)d\alpha}}}{}{}
    \addConstraint{y_{ai}}{\leq\outfuns_a(x_{ai}), \quad}{(a,i) \in A \times [k]}
    \addConstraint{x_{ai}}{\leq \caps_a, \quad}{(a,i) \in A \times [k]}
    \addConstraint{y_{ai}}{\geq 0, \quad}{(a,i) \in A \times [k]}
    \addConstraint{\vini{v}}{\geq \vouti{v}, \quad}{(v,i) \in (V \setminus (S \cup \sinks)) \times [k]}
    \addConstraint{\vini{\sink}}{\geq \demands_{\sink i}, \quad}{(\sink,i) \in \sinks \times [k]}
    \addConstraint{\sum_{i=1}^{k}\abs{\TIntervals_{i}}\vouti{s}}{\leq \csupplies_s, \quad}{s \in \sources}
\end{mini*}
\end{program}
or equivalently
\begin{mini*}{x,y}{\sum_{s \in \sources}{\int_0^{\sum_{i=1}^{k}\abs{\TIntervals_{i}}\vouti{s}}{\prodcpus_s(\alpha)d\alpha}}}{}{}
    \addConstraint{y_a}{\leq\outfuns_a(x_a), \quad}{a \in A}
    \addConstraint{x_a}{\leq \caps_a, \quad}{a \in A}
    \addConstraint{y_a}{\geq 0, \quad}{a \in A}
    \addConstraint{\vin{v}}{\geq \vout{v}, \quad}{v \in V \setminus (S \cup \sinks)}
    \addConstraint{\vin{\sink}}{\geq \demands_\sink, \quad}{\sink\in\sinks}
    \addConstraint{\sum_{i=1}^{k}\abs{\TIntervals_{i}}\vouti{s}}{\leq \csupplies_s, \quad}{s \in \sources}
\end{mini*}

Given an instance $(g)$ of \cref{M:FunCP}, an optimal solution of $(g)$ can then be obtained by finding an optimal solution $z$ of its corresponding \cref{M:HardCP} instance, and then from $z$ constructing
\begin{equation}\label{E:PiecewiseConstantFromVec}
    [0,\dur] \ni t \mapsto
    \begin{cases}
        z_1, & \text{if } t \in \TIntervals_1\\
        \vdots &\\
        z_k, & \text{if } t \in \TIntervals_k.
    \end{cases}
\end{equation}

\begin{proposition}\label{Pr:FunCPReducesToHardCP}
    \cref{M:FunCP} reduces to \cref{M:HardCP}.
\end{proposition}
\begin{proof}
    Assume $(g)$ is an instance of \cref{M:FunCP} and $(h)$ is the corresponding instance of \cref{M:HardCP}. The encoding of $(g)$ requires at least one bit per breakpoint of the demands function $\dfuns$. Hence, the $\InsMap$ that maps $(g)$ to $(h)$ has its execution time and output encoding size $\bits{(h)}$ polynomial in
    \begin{equation*}
        \abs{A}+k,
    \end{equation*}
    thus in the number of bits encoding $(g)$, denoted by $\bits{(g)}$.
    
    For our $\SolMap$, we use
    \begin{equation*}
        (z_1,\ldots,z_k) \mapsto \left(t \mapsto
        \begin{cases}
            z_1, & \text{if } t \in \TIntervals_1\\
            \vdots &\\
            z_k, & \text{if } t \in \TIntervals_k
        \end{cases}\right).
    \end{equation*}
    By construction, this $\SolMap$ maps a feasible point of $(h)$ to a feasible point of $(g)$ with the same value. By implementing $\SolMap$ as one big switch statement, its execution time and $\bits{(g)}$ become polynomial in $\bits{(h)}$.
    
    $(\InsMap,\SolMap)$ is hence a reduction from \cref{M:FunCP} to \cref{M:HardCP}.
\end{proof}

\subsection{Reducing to a three times differentiable objective}\label{SS:HardCPToCP}
As $\prodcpus$ is non-decreasing, its integral is convex, so \cref{M:HardCP} is a \ac{CP}. However, to optimize it with performance guarantees, we want a three times differentiable objective function. To get that, we employ \cref{A:ToStatic} to map instance of \cref{M:HardCP} to instances of \cref{M:CP}.

\begin{program}[CP / min-cost concave generalized flow]\label{M:CP}
Given: A weakly connected digraph $G=(V,A,\outfuns,\costs,\caps,\demands)$, with $\sinks \subseteq V \ni \ssource$; non-decreasing and three times differentiable $\costs$, with $\costs(0)=0$; concave and strictly increasing $\outfuns_a$, with $\outfuns_a=0$ and $\outfuns_a \leq Id$.

\begin{mini*}{x,y}{\sum_{a \in A}{c_a(x_a)}}{}{}
    \addConstraint{y_a}{\leq\outfuns_a(x_a), \quad}{a \in A}
    \addConstraint{x_a}{\leq \caps_a, \quad}{a \in A}
    \addConstraint{y_a}{\geq 0, \quad}{a \in A}
    \addConstraint{\vin{v}}{\geq \vout{v}, \quad}{v \in V \setminus (\{\ssource\} \cup \sinks)}
    \addConstraint{\vin{\sink}}{\geq \demands_\sink, \quad}{\sink \in \sinks}
\end{mini*}
\end{program}

Note how \cref{M:CP} corresponds to a min-cost generalized $\ssource$-$\sinks$ flow problem permitting flow to be wasted. By \cref{L:OptIsFeasible} we moreover know that any optimal solution of \cref{M:CP} will have no waste. Hence, every optimal solution of \cref{M:CP} corresponds to a min-cost generalized $\ssource$-$\sinks$ flow.

\begin{algorithm}[htbp]
\caption{Simplify \acs{CP}}\label{A:ToStatic}
\KwData{\ac{GPGG} $G = (V,A,\outfuns,\caps,\csupplies,\dfuns,\prodcpus)$ as in \cref{P:GeneralizedConstant}}
\KwResult{$G' = (V',A',\outfuns,\caps,\costs)$ as in \cref{M:CP}}
    \begin{outline}[enumerate]
        \1 \label{I:duplicating} Make $k$ copies $G_1=(V_1,A_1),\ldots,G_k=(V_k,A_k)$ of $G=(V,A)$ and define $G'=(V',A')$ as their union (recall that $k$ is the number of pieces of $\dfuns$).
        \1 Add to $G'$ a super-source $\ssource$ with infinite supply.
        \1 For each $a' \in A'$: Define cost function $\costs_{a'}: x \mapsto 0$.
        \1 For each $s \in \sources$: For each (maximal) twice differentiable piece $I=[I^{min},I^{max})\subseteq[0,\csupplies_s]$ of $\prodcpus_s$:
            \2 Add to $V'$ a node $s_I$.
            \2 Add to $A'$ an arc $a_I=(\ssource,s_I)$ corresponding to cumulative production in $I$, but scaled to ensure $\outfuns \leq \Id$ on all arcs: capacity $\caps_{a_I}:=\abs{I}/\MinTIntervalLen$, cost function $\costs_{a_I}: x \mapsto \int_{I^{min}}^{I^{min} + \MinTIntervalLen x}\prodcpus_s(\alpha)d\alpha$, and $\outfuns_{a_I}:=\Id$.
            \2 For each $i \in [k]$: Add to $A'$ an arc $a_{Ii}:=(s_I,s_i)$, with $i$ being the copy of $s$ in graph $G_i$, corresponding to delegating cumulative production in $I$ to $\TIntervals_i$ and transforming it to production rate: capacity $\caps_{a_{Ii}}:=\abs{I}/\MinTIntervalLen$, cost function $\costs_{a_{Ii}}: x \mapsto 0$, and $\outfuns_{a_{Ii}}: x \mapsto \frac{x}{\abs{\TIntervals_i}/\MinTIntervalLen}$.
        \1 Return $G'$, with $\sources':=\{\ssource\}$ and $\sinks':=\cup_{i \in [k]}\sinks_i$, where $\sinks_i$ is the copy of $\sinks$ corresponding to graph copy $G_i$.
    \end{outline}
\end{algorithm}

\begin{figure}[htbp]
    \centering
    \begin{subfigure}[t]{0.2\linewidth}
        \centering
        \begin{tikzpicture}
            \node[source]      (s)                    {$s$};
            \node[sink]        (d)       [below=of s] {$\sink$};
            
            \draw[->] (s) -- (d);
        \end{tikzpicture}
        \caption{Input}
    \end{subfigure}
    \begin{subfigure}[t]{0.35\linewidth}
        \centering
        \begin{tikzpicture}
            \node[source]      (s*)                    {$\ssource$};
            \node[trans]        (s_{I})       [below left=of s*] {$s_{I}$};
            \node[trans]        (s_{I'})       [below right=of s*] {$s_{I'}$};
            \node[trans]        (s_1)       [below right=of s_{I}] {$s_1$};
            \node[sink]        (d_1)       [below=of s_1] {$\sink_1$};
            
            \draw[->] (s*) -- (s_{I});
            \draw[->] (s*) -- (s_{I'});
            \draw[->] (s_{I}) -- (s_1);
            \draw[->] (s_{I'}) -- (s_1);
            \draw[->] (s_1) -- (d_1);
        \end{tikzpicture}
        \caption{Output for $k=1$}
    \end{subfigure}
    \begin{subfigure}[t]{0.35\linewidth}
        \centering
        \begin{tikzpicture}
            \node[source]      (s*)                    {$\ssource$};
            \node[trans]        (s_{I})       [left=of s*] {$s_{I}$};
            \node[trans]        (s_{I'})       [right=of s*] {$s_{I'}$};
            \node[trans]        (s_1)       [below=of s*] {$s_1$};
            \node[sink]        (d_1)       [below=of s_1] {$\sink_1$};
            \node[trans]        (s_2)       [above=of s*] {$s_2$};
            \node[sink]        (d_2)       [above=of s_2] {$\sink_2$};
            
            \draw[->] (s*) -- (s_{I});
            \draw[->] (s*) -- (s_{I'});
            \draw[->] (s_{I}) -- (s_1);
            \draw[->] (s_{I'}) -- (s_1);
            \draw[->] (s_1) -- (d_1);
            \draw[->] (s_{I}) -- (s_2);
            \draw[->] (s_{I'}) -- (s_2);
            \draw[->] (s_2) -- (d_2);
        \end{tikzpicture}
        \caption{Output for $k=2$}
    \end{subfigure}
    \caption{Input and outputs of \cref{A:ToStatic} on a two-node instance of \cref{P:GeneralizedConstant}. In the instance, $\sources=\{s\}$, $\sinks=\{d\}$, $\prodcpus_s$ has one breakpoint, and $\dfuns$ has $k-1$ breakpoints.}
\end{figure}

\begin{definition}
Given a graph $G$ as in \cref{P:GeneralizedConstant} and a graph $G'$ corresponding to running \cref{A:ToStatic} on $G$:
\begin{itemize}
    \item $\FeasSetHardCP:=\{\text{feasible points of \cref{M:HardCP} on graph }G\}$
    \item $\FeasSetCP:=\{\text{feasible points of \cref{M:CP} on graph }G'\}$
    \item $\OptSetHardCP:=\{\text{optimal solutions of \cref{M:HardCP} on graph }G\}$
    \item $\OptSetCP:=\{\text{optimal solutions of \cref{M:CP} on graph }G'\}$
\end{itemize}
\end{definition}

\begin{definition}
    $\CumPurger: \OptSetCP \to \FeasSetHardCP$ is an orthogonal projection that for an input $z \in \OptSetCP$ removes those of its elements corresponding to arcs in $\delta_{\text{Neighbors}(\ssource)}$. $\otild{\CumPurger}$ is its extension to all of $\FeasSetCP$.
\end{definition}

\begin{lemma}\label{L:ToStaticPreservesValue}
    If $(f)$ is a \cref{M:HardCP} instance and $(g)$ is its corresponding \cref{M:CP} instance, then $\OPT[(f)]=\OPT[(g)]$.
\end{lemma}
\begin{proof}
    We will prove this by showing that $\CumPurger$ is value-preserving and maps to $\OptSetHardCP \subseteq \FeasSetHardCP$. Assume $z=(x,y)\in\OptSetCP$ and $z'=(x',y'):=\CumPurger(z)$. Then $z' \in \FeasSetHardCP$ by Step \ref{I:duplicating} of \cref{A:ToStatic}, so $\CumPurger$ is well-defined. For each source $s \in \sources$, define $\mathcal{I}_s$ as the set of all (maximal) twice differentiable pieces $I=[I^{min},I^{max})\subseteq[0,\csupplies_s]$ of $\prodcpus_s$. Then,
    \begin{align}\label{E:CumPurgerPreservesValue}
    \begin{split}
        &\sum_{a \in A} \costs_a(x_a)\\
        &= \sum_{s \in \sources}\sum_{I \in \mathcal{I}_s} \costs_{(\ssource, s_I)}(x_{(\ssource, s_I)})\\
        &= \sum_{s \in \sources}\sum_{I \in \mathcal{I}_s} \costs_{(\ssource, s_I)}(
        \left\{
        \begin{array}{@{}ll}
            0, & \text{if } \sum_{J \in \mathcal{I}_s}x_{(s^*,s_J)} < \frac{I^{min}}{\MinTIntervalLen}\\
            \sum_{J \in \mathcal{I}_s}x_{(s^*,s_J)} - \frac{I^{min}}{\MinTIntervalLen}, & \text{if } \sum_{J \in \mathcal{I}_s}x_{(s^*,s_J)} \in I/\MinTIntervalLen\\
            \frac{\abs{I}}{\MinTIntervalLen}, & \text{if } \sum_{J \in \mathcal{I}_s}x_{(s^*,s_J)} \geq \frac{I^{max}}{\MinTIntervalLen}
        \end{array} \right\})\\
        &= \sum_{s \in \sources}\sum_{I \in \mathcal{I}_s} \costs_{(\ssource, s_I)}(\\
        &\quad\left\{
        \begin{array}{@{}ll}
            0, & \text{if } \sum_{(J,i) \in \mathcal{I}_s \times [k]}\outfuns_{(s_J,s_i)}^{-1}(y_{(s_J,s_i)}) < \frac{I^{min}}{\MinTIntervalLen}\\
            \sum_{(J,i) \in \mathcal{I}_s \times [k]}\outfuns_{(s_J,s_i)}^{-1}(y_{(s_J,s_i)}) - \frac{I^{min}}{\MinTIntervalLen}, & \text{if } \sum_{(J,i) \in \mathcal{I}_s \times [k]}\outfuns_{(s_J,s_i)}^{-1}(y_{(s_J,s_i)}) \in I/\MinTIntervalLen\\
            \frac{\abs{I}}{\MinTIntervalLen}, & \text{if } \sum_{(J,i) \in \mathcal{I}_s \times [k]}\outfuns_{(s_J,s_i)}^{-1}(y_{(s_J,s_i)}) \geq \frac{I^{max}}{\MinTIntervalLen}
        \end{array} \right\})\\
        &= \sum_{s \in \sources}\sum_{I \in \mathcal{I}_s} \costs_{(\ssource, s_I)}(
        \left\{
        \begin{array}{@{}ll}
            0, & \text{if } \sum_{i=1}^{k}\frac{\abs{\TIntervals_i}}{\MinTIntervalLen}\vout{s_i} < \frac{I^{min}}{\MinTIntervalLen}\\
            \sum_{i=1}^{k}\frac{\abs{\TIntervals_i}}{\MinTIntervalLen}\vout{s_i} - \frac{I^{min}}{\MinTIntervalLen}, & \text{if } \sum_{i=1}^{k}\frac{\abs{\TIntervals_i}}{\MinTIntervalLen}\vout{s_i} \in I/\MinTIntervalLen\\
            \frac{\abs{I}}{\MinTIntervalLen}, & \text{if } \sum_{i=1}^{k}\frac{\abs{\TIntervals_i}}{\MinTIntervalLen}\vout{s_i} \geq \frac{I^{max}}{\MinTIntervalLen}
        \end{array} \right\})\\
        &= \sum_{s \in \sources}\sum_{I \in \mathcal{I}_s} \costs_{(\ssource, s_I)}(\frac{1}{\MinTIntervalLen}
        \left\{
        \begin{array}{@{}ll}
            0, & \text{if } \sum_{i=1}^{k}\abs{\TIntervals_i}\vout{s_i} < I^{min}\\
            \sum_{i=1}^{k}\abs{\TIntervals_i}\vout{s_i} - I^{min}, & \text{if } \sum_{i=1}^{k}\abs{\TIntervals_i}\vout{s_i} \in I\\
            \abs{I}, & \text{if } \sum_{i=1}^{k}\abs{\TIntervals_i}\vout{s_i} \geq I^{max}
        \end{array} \right\})\\
        &= \sum_{s \in \sources}\sum_{I \in \mathcal{I}_s}
        \begin{cases}
            0, & \text{if } \sum_{i=1}^{k}\abs{\TIntervals_i}\vout{s_i} < I^{min}\\
            \int_{I^{min}}^{\sum_{i=1}^{k}\abs{\TIntervals_i}\vout{s_i}}\prodcpus_s(\alpha)d\alpha, & \text{if } \sum_{i=1}^{k}\abs{\TIntervals_i}\vout{s_i} \in I\\
            \int_{I^{min}}^{I^{max}}\prodcpus_s(\alpha)d\alpha, & \text{if } \sum_{i=1}^{k}\abs{\TIntervals_i}\vout{s_i} \geq I^{max}
        \end{cases}\\
        &=\sum_{s \in \sources}{\int_0^{\sum_{i=1}^{k}\abs{\TIntervals_{i}}\vout{s_i}}{\prodcpus_s(\alpha)d\alpha}}\\
        &=\sum_{s \in \sources}{\int_0^{\sum_{i=1}^{k}\abs{\TIntervals_{i}}\vouti{s}'}{\prodcpus_s(\alpha)d\alpha}}.
    \end{split}
    \end{align}
    Here, the second equality stems from the non-decreasing nature of all $\prodcpus_s$ and the optimality of $z$, as this means that $(\ssource, s_{I_{j}})$ must be fully utilized by $z$ before $x_{(\ssource, s_{I_{j+1}})}$ can become non-zero (where $I_{j}^{max}=I_{j+1}^{min}$). The other equalities in \eqref{E:CumPurgerPreservesValue} are just the result of plugging in the relationships between $x$ and $y$, the definitions of $\mathcal{I}_s$, and the definitions of $\outfuns_a$ on the arcs added by \cref{A:ToStatic}. From \eqref{E:CumPurgerPreservesValue}, we conclude that $\CumPurger$ is value-preserving.
    
    Given an optimal $z' \in \OptSetHardCP$, one can always construct a $z\in\FeasSetCP$ abiding by the second equality of \eqref{E:CumPurgerPreservesValue}, which then by \eqref{E:CumPurgerPreservesValue} will have the same value as $z'$. Thus, $z \in \OptSetCP$, as otherwise the value-preservation property of $\CumPurger$ would imply that $z' \notin \OptSetHardCP$, and so both program instances have the same optimal value.
\end{proof}

\begin{proposition}\label{Pr:HardCPReducesToCP}
    \cref{M:HardCP} reduces to \cref{M:CP}.
\end{proposition}
\begin{proof}
    If we extend the domain of $\CumPurger$ to all of $\FeasSetCP$, then the second equality of \eqref{E:CumPurgerPreservesValue} should be replaced by a $\geq$ while all its other equalities remain true. As the execution time and $\bits{out}$ of $\CumPurger$ is clearly polynomial in its $\bits{in}$, and by \cref{L:ToStaticPreservesValue}, we thus conclude that $\otild{\CumPurger}$ can be used as our $\SolMap$. As running \cref{A:ToStatic} and initializing a \cref{M:CP} instance from its output makes for a corresponding $\InsMap$, we conclude that \cref{M:HardCP} reduces to \cref{M:CP}.
\end{proof}

\subsubsection{Special case: quadratic losses and piecewise linear $\prodcpus$}

Running \cref{A:ToStatic} on an instance of \cref{M:HardCP} with quadratic line losses and piecewise constant production costs $\prodcpus$ generates a graph with quadratic arc costs and $\outfuns$. Hence, the corresponding \cref{M:CP} instance is then an instance of the more specific \cref{M:QCQP}, which we recognize to be a convex \acf{QCQP}.

\begin{program}[Convex QCQP / min-cost generalized flow with quadratic losses]\label{M:QCQP}
Given: $G$ as in \cref{M:CP}, but where
\begin{equation*}
    \forall{a \in A}:
    \begin{cases}
        \outfuns_a: x \mapsto \outscaler_a x - r_a x^2,
         \quad \outscaler_a \in [0,1] \a r_a \geq 0\\
         c_a: x \mapsto \clin_a x + \cquad_a x^2,
         \quad \clin_a,\cquad_a \geq 0.
    \end{cases}
\end{equation*}

\begin{mini*}{x,y}{\sum_{a \in A}{\clin_a x_a + \cquad_a x_a^2}}{}{}
    \addConstraint{r_a x_a^2 - \outscaler_a x_a + y_a}{\leq 0, \quad}{a \in A}
    \addConstraint{x_a}{\leq \caps_a, \quad}{a \in A}
    \addConstraint{y_a}{\geq 0, \quad}{a \in A}
    \addConstraint{\vin{v}}{\geq \vout{v}, \quad}{v \in V \setminus (\{\ssource\} \cup \sinks)}
    \addConstraint{\vin{\sink}}{\geq \demands_\sink, \quad}{\sink \in \sinks}
\end{mini*}
\end{program}

Note how the $\InsMap$ of \crefrange{SS:OPToFunCP}{SS:HardCPToCP} hence map instances of \cref{P:SimpleConstant} to instances of \cref{M:QCQP}.

\subsection{Reducing the generalized problem}\label{SS:SolvingGeneral}
Using everything we have seen so far, we now wish to construct a reduction from \cref{M:OP} to \cref{M:CP}, hence from \cref{P:GeneralizedConstant} to \cref{M:CP}. We know from \cref{C:OPAndFunCPShareOptima} that relaxing an instance of \cref{M:OP} to a corresponding instance of \cref{M:FunCP} does not change the set of optimal solutions nor the optimal value. Hence, a first instinct may be to find a reduction from \cref{M:OP} to \cref{M:FunCP}, and then compose this with the other reductions seen so far. This is however challenging without additional constraints on \cref{M:FunCP}, as its feasible points can be highly exotic functions that are hard to manipulate into feasible points of \cref{M:OP}. Hence, we will here instead reduce \cref{M:OP} directly to \cref{M:CP}, as we can then make the necessary manipulations on the vectors found for \cref{M:CP}, before functions are generated from these.

Mapping an instance of \cref{M:OP} to a corresponding instance of \cref{M:FunCP} is as simple as replacing the appropriate equalities with inequalities. By \crefrange{SS:OPToFunCP}{SS:HardCPToCP}, we hence find that \cref{A:GeneralizedConstantToCP} is an $\InsMap$ from \cref{M:OP} to \cref{M:CP}.

\begin{algorithm}[htbp]
\caption{Map \cref{M:OP} instances to \cref{M:CP} instances}\label{A:GeneralizedConstantToCP}
\KwData{\cref{P:GeneralizedConstant} instance $(f)$ with corresponding \ac{GPGG} $G_{(f)}$.}
\KwResult{\cref{M:CP} instance $(g)$}
    \begin{outline}[enumerate]
        \1 \label{I:FirstRedundantReduction} $(f') \longleftarrow$ \eqref{E:RelaxOP} applied to $(f)$
        \1 \label{I:SecondRedundantReduction} $(f'') \longleftarrow$ \cref{M:HardCP} instance corresponding to $G_{(f')}=G_{(f)}$
        \1 $G \longleftarrow$ \cref{A:ToStatic} on $G_{(f'')}=G_{(f')}=G_{(f)}$
        \1 $(g) \longleftarrow$ \cref{M:CP} instance corresponding to $G''$
    \end{outline}
\end{algorithm}

Note how Steps \ref{I:FirstRedundantReduction} and \ref{I:SecondRedundantReduction} of \cref{A:GeneralizedConstantToCP} do not modify the underlying graph, which is all the next steps rely on, hence are completely redundant and can be skipped.

Assume now that $(f)$ is some instance of \cref{P:GeneralizedConstant} and that $(f')$ and $(g)$ are derived from $(f)$ as in \cref{A:GeneralizedConstantToCP}. Repeating the arguments of \cref{Pr:FunCPReducesToHardCP,Pr:HardCPReducesToCP}, we find that the $\SolMap$ used for reducing \cref{M:FunCP} to \cref{M:CP} can be used for reducing \cref{M:OP} to \cref{M:EffOP}.

\begin{program}\label{M:EffOP}
Given: $G$ as in \cref{M:CP}.
\begin{mini*}{x,y}{\sum_{a \in A}{c_a(x_a)}}{}{}
    \addConstraint{y_a}{= \outfuns_a(x_a), \quad}{a \in A}
    \addConstraint{x_a}{\leq \caps_a, \quad}{a \in A}
    \addConstraint{y_a}{\geq 0, \quad}{a \in A}
    \addConstraint{\vin{v}}{=\vout{v}, \quad}{v \in V \setminus (\{\ssource\} \cup \sinks)}
    \addConstraint{\vin{\sink}}{=\demands_\sink, \quad}{\sink \in \sinks}
\end{mini*}
\end{program}

Hence, for our $\SolMap$ we just have to prepend a \textit{rounding} step that maps a $(g)$-feasible point to a feasible point of the corresponding \cref{M:EffOP} instance. \cref{A:Round} is that step, either decreasing arc flows or leaving them unchanged, hence either decreasing the total solution cost or leaving it unchanged. Specifically, it alternates between finding flow-carrying paths from $\ssource$ to arcs or nodes with waste, and between removing as much flow as possible from those paths without breaking \cref{M:CP} feasibility.

\begin{algorithm}[htbp]
\caption{Remove all waste}\label{A:Round}
\KwData{$G=(V,A,\outfuns,\caps,\costs)$ as in \cref{M:CP} and $x,y$ feasible on the corresponding \cref{M:CP} instance.}
\KwResult{$x,y$}
    \begin{outline}[enumerate]
        \1 \label{I:defs} Define:
            \2 $A':=\{a \in A: x_a > 0\})$
            \2 $G':=(V,A')$ ($\OutArcs{}$ will then be in terms of this)
            \2 $q:=0^{A'}$
            \2 $parent:=\text{Null}^{A'}$
            \2 $queue:=\text{list}(\OutArcs{\ssource})$
            \2 $expanded:=\{\}$
        \1 \label{I:bfs} While $\text{length}(queue)>0$
            \2 $a \leftarrow (u,v) \leftarrow queue.$leftpop()
            \2 If $v \in expanded$: continue
            \2 $expanded.$add($v$)
            \2 If $parent_a=$Null: $q_a \leftarrow x_a$
            \2 Else: $q_a \leftarrow \min\{\outfuns_{parent_a}(q_{parent_a}),x_a\}$
            \2 \label{I:OnArcWaste} If $y_a < \outfuns_a(x_a)$:
                \3 \label{I:initd} $d:=\min\{q_a, x_a - \gamma_{a}^{-1}(y_a)\}$
                \3 \label{I:EliminationLoop} While $parent_a \neq $Null:
                    \4 $x_a \leftarrow x_a - d$
                    \4 \label{I:updy} $y_{parent_a} \leftarrow y_{parent_a} - d$
                    \4 $a \leftarrow parent_a$
                    \4 \label{I:updd} $d \leftarrow \outfuns_{a}^{-1}(d)$
                \3 $x_a \leftarrow x_a - d$
                \3 Go to Step \ref{I:defs}
            \2 \label{I:OnNodeWaste} If $\vin{v} > \vout{v} + \demands_v$:
                \3 $d:=\min\{q_a, \vin{v} - \vout{v} - \demands_v\}$
                \3 $y_a \leftarrow y_a - d$
                \3 Go to Step \ref{I:EliminationLoop}
            \2 For all $a' \in \OutArcs{v}$:
                \3 $parent_{a'} \leftarrow a$
                \3 $queue.$append($a'$)
    \end{outline}
\end{algorithm}

\begin{lemma}\label{L:RoundDoesNoHarm}
    Given a \cref{M:CP} instance $(f)$, a corresponding \cref{M:EffOP} instance $(f')$, and an $(f)$-feasible $z=(x,y)$, \cref{A:Round} in $\BigOh(\abs{A}^2)$ time returns an $(f')$-feasible point $z' \leq z$ of at most the same value as $z$.
\end{lemma}
\begin{proof}
    We start by proving correctness. As our input $z$ is $(f)$-feasible and steps \ref{I:OnArcWaste} and \ref{I:OnNodeWaste} are the only ones modifying $z$, it is enough to show that these do not introduce any constraint violations of $(f')$, keep running until $z$ is $(f')$-feasible, reduce the cost of $z$, and bring $z$ "closer" to $(f')$-feasibility.
    
    So assume first that we have just reached Step \ref{I:OnArcWaste} for some arc $a=(u,v) \in A' \subseteq A$, hence that $y_a < \outfuns_a(x_a)$. As Step \ref{I:OnArcWaste} or \ref{I:OnNodeWaste} would otherwise have been triggered before reaching $a$, hence causing the \ac{BFS} to start over, we know that
    \begin{equation*}
        P=(a_1,...,a_l):=(parent_{parent_{..._{parent_a}}},...,parent_{parent_a},parent_a)
    \end{equation*}
    is a path-carrying $\ssource$-$u$ path satisfying the constraints of $(f')$.
    
    Our goal in Step \ref{I:OnArcWaste} is thus to reduce $x_a$ by as much as possible without introducing any constraint violations of $(f')$. Step \ref{I:OnArcWaste} reduces $x_a$ by $d$ and then iterates backwards through $P$ to ensure all arcs and nodes of it satisfy the constraints of $(f')$. Hence, $d$ is the solution of
    \begin{maxi}{d}{d}{\label{E:MaxElimination}}{}
        \addConstraint{x_a - \gamma_{a}^{-1}(y_a)}{\geq d}
        \addConstraint{\outfuns_{a_{l}}(x_{a_{l}})}{\geq d}
        \addConstraint{\outfuns_{a_{l}}(\outfuns_{a_{l-1}}(x_{a_{l-1}}))}{\geq d}
        \addConstraint{\outfuns_{a_{l}}(\outfuns_{a_{l-1}}(\outfuns_{a_{l-2}}(x_{a_{l-2}})))}{\geq d}
        \addConstraint{...}{}
        \addConstraint{\outfuns_{a_{l}}(...(\outfuns_{a_1}(x_{a_1})))}{\geq d}.
    \end{maxi}
    Here, the first constraint ensures $\outfuns_a(x_a) \geq y_a$, hence that no constraint violations of $(f)$ are introduced, while the $i$-th constraint ensures that we can satisfy the new $x_a$ and strong flow conservation without making $a_{l+2-i}$ waste flow, hence that no constraint violations of $(f')$ are introduced. We note that the solution of \eqref{E:MaxElimination} is positive, since $P$ is flow-carrying and $a$ has waste.
    
    Since all $\outfuns$ are strictly increasing, the whole set of inequalities of \eqref{E:MaxElimination} is equivalent to the single inequality
    \begin{equation}\label{E:CombinedInequality}
    \begin{split}
        d \leq &\min\{x_a - \outfuns_{a}^{-1}(y_a), \outfuns_{a_{l}}(\\
            &\quad\min\{x_{a_{l}}, \outfuns_{a_{l-1}}(\\
                &\qquad\min\{x_{a_{l-1}}, \outfuns_{a_{l-2}}(...(\\
                    &\quad\qquad\min\{x_{a_2}, \outfuns_{a_1}(x_{a_1})\}
                ))\}
            )\}
        )\}.
    \end{split}
    \end{equation}
    
    Upon reaching Step \ref{I:OnArcWaste}, we have
    \begin{equation*}
    \begin{split}
        q_{a_1} &= x_{a_1}\\
        q_{a_2} &= \min\{x_{a_2}, \outfuns_{a_1}(q_{a_1})\}\\
        q_{a_3} &= \min\{x_{a_3}, \outfuns_{a_2}(q_{a_2})\}\\
        \vdots&\\
        q_{a_{l}} &= \min\{x_{a_{l}}, \outfuns_{a_{l-1}}(q_{a_{l-1}})\}\\
        q_{a} &= \min\{x_{a}, \outfuns_{a_{l}}(q_{a_{l}})\}.
    \end{split}
    \end{equation*}
    Hence, 
    \begin{equation*}
        d := \min\{q_a, x_a - \gamma_{a}^{-1}(y_a)\}
    \end{equation*}
    equals the right-hand side of \eqref{E:CombinedInequality}. This then means that the $d$ of Step \ref{I:OnArcWaste} maximizes \eqref{E:MaxElimination} and is tight in one of its constraints. Hence, Step \ref{I:OnArcWaste} neither increases any $x_{a'}$ nor breaks any
    \begin{equation*}
        \outfuns_{a'}(x_{a'})=y_{a'},
    \end{equation*}
    yet reduces $\vout{\ssource}$ and causes
    \begin{equation*}
        x_{a_i} = 0 \text{ for some } i \in [l]
    \end{equation*}
    and/or
    \begin{equation*}
        \outfuns_a(x_a) = y_a.
    \end{equation*}
    
    Therefore, running Step \ref{I:OnArcWaste} on an $(f)$-feasible $z$ retains the $(f)$-feasibility of $z$ while lowering its cost and causing at least one arc $a'$ to irreversibly and for the first time satisfy $\outfuns_{a'}(x_{a'})=y_{a'}$ or $x_{a'}=0$. A similar analysis tells us that running Step \ref{I:OnNodeWaste} on an $(f)$-feasible $z$ retains the $(f)$-feasibility of $z$ while lowering its cost and causing one node to irreversibly and for the first time satisfy $\vin{v}=\vout{v}$ and/or at least one arc $a'$ to irreversibly and for the first time satisfy $x_{a'}=0$.
    
    Thus, steps \ref{I:OnArcWaste} and \ref{I:OnNodeWaste} must run a total of at most $2\abs{A}+\abs{V}$ times before
    \begin{equation*}
        \outfuns_a(x_a) = y_a, \ \forall{a \in A}.
    \end{equation*}
    
    The only remaining question is whether we can be in a situation where $x,y$ has waste, yet neither Step \ref{I:OnArcWaste} nor \ref{I:OnNodeWaste} is reached. However, since Step \ref{I:bfs} is a \ac{BFS} on the graph of flow-carrying arcs, we know from \cref{L:WasteIsReachableFromSource} that Step \ref{I:OnArcWaste} and/or \ref{I:OnNodeWaste} is reached in Step \ref{I:bfs} iff $z$ has waste. To conclude:
    
    \begin{itemize}
        \item We start with an $(f)$-feasible flow.
        \item If we have an $(f)$-feasible flow, then steps \ref{I:OnArcWaste} and \ref{I:OnNodeWaste} give an $(f)$-feasible flow of strictly lower cost.
        \item Steps \ref{I:OnArcWaste} and \ref{I:OnNodeWaste} will keep being reached as long as the flow is not $(f')$-feasible.
        \item Steps \ref{I:OnArcWaste} and \ref{I:OnNodeWaste} can in total run at most $2\abs{A}+\abs{V}$ times.
    \end{itemize}
    Thus, given $(f)$-feasible $z$, \cref{A:Round} returns a lower-cost $(f')$-feasible $z'$ in finite time.
    
    To prove that the algorithm converges in $\BigOh(\abs{A}^2)$ time, first note that steps \ref{I:defs}, \ref{I:OnArcWaste} and \ref{I:OnNodeWaste} each run in $\BigOh(\abs{A})$ time.
    
    Assume the values of $\vin{v}$ and $\vout{v}$ are stored in node $v$, rather than having to be recomputed upon use. Then each iteration of the loop in Step \ref{I:bfs} either runs in $\BigOh(1)$ time or triggers Step \ref{I:OnArcWaste} or \ref{I:OnNodeWaste}, both of which break out of the loop. As Step \ref{I:bfs} visits no arc twice, it performs at most $\abs{A}$ iterations. Thus, Step \ref{I:bfs} runs in $\BigOh(A+A)=\BigOh(A)$ time.
    
    As we already saw, steps \ref{I:OnArcWaste} and \ref{I:OnNodeWaste} can in total be executed at most $2\abs{A}$ times. As steps \ref{I:OnArcWaste} and \ref{I:OnNodeWaste} are the only that make the algorithm start over from Step \ref{I:defs}, we conclude that steps \ref{I:defs} and \ref{I:bfs} run $\BigOh(A)$ times, each time at a cost of $\BigOh(A)$. Thus, the total time complexity of \cref{A:Round} is $\BigOh(\abs{A}^2)$.

\end{proof}

We now have everything we need to define the $\SolMap$ for reducing from \cref{M:OP} to \cref{M:CP}, which we do in \cref{A:CPToGeneralizedConstant}.

\begin{algorithm}[htbp]
\caption{Map feasible points of \cref{M:CP} instances to feasible points of \cref{P:GeneralizedConstant} instances}\label{A:CPToGeneralizedConstant}
\KwData{\cref{M:CP} instance $(f)$ and $(f)$-feasible $z=(x,y)$.}
\KwResult{$\otild{z}$}
    \begin{outline}[enumerate]
        \1 $z' \longleftarrow$ \cref{A:Round} on $z'$ and $G_{(f)}$
        \1 \label{I:LightPurge} $z'' \longleftarrow \otild{\CumPurger}(z')$
        \1 $\otild{z} \longleftarrow$ \eqref{E:PiecewiseConstantFromVec} with $z''$
    \end{outline}
\end{algorithm}

\begin{proposition}\label{Pr:GeneralizedConstantReducesToCP}
    \cref{P:GeneralizedConstant} reduces to \cref{M:CP} through $(\text{\cref{A:GeneralizedConstantToCP}}, \text{\cref{A:CPToGeneralizedConstant}})$.
\end{proposition}
\begin{proof}
    Assume $(f)$ is an instance of \cref{M:OP}, $(g)$ is the output of \cref{A:GeneralizedConstantToCP}, $z$ is a $(g)$-feasible point, and $z'$ is the output of \cref{A:CPToGeneralizedConstant} on $z$. By \crefrange{SS:OPToFunCP}{SS:HardCPToCP}, $(g)$ is an instance of \cref{M:CP}. Moreover, by \cref{Pr:FunCPReducesToHardCP,Pr:HardCPReducesToCP,L:RoundDoesNoHarm} and the arguments used in their proofs, $z'$ is an $(f)$-feasible point of at most the same $(f)$-value as the $(g)$-value of $z$. Finally, by \cref{L:ToStaticPreservesValue} and the constructions of \crefrange{M:OP}{M:HardCP}, we know that $(f)$ and $(g)$ have the same optimal value. We hence conclude that $(\text{\cref{A:GeneralizedConstantToCP}}, \text{\cref{A:CPToGeneralizedConstant}})$ is a reduction from \cref{M:OP} to \cref{M:CP}, thus from \cref{P:GeneralizedConstant} to \cref{M:CP}.
\end{proof}

\begin{corollary}\label{C:SimpleConstantReducesToQCQP}
    \cref{P:SimpleConstant} reduces to \cref{M:QCQP} through $(\text{\cref{A:GeneralizedConstantToCP}}, \text{\cref{A:CPToGeneralizedConstant}})$.
\end{corollary}

\subsection{Reducing the original problem}\label{SS:SolvingOriginal}
So far we have only looked at reductions of \cref{P:SimpleConstant,P:GeneralizedConstant}. It is however \cref{P:ConstantOP} that we are trying to solve, hence this is the problem that we actually want to reduce. The first instinct might be to seek a reduction from \cref{P:ConstantOP} to \cref{P:SimpleConstant}. However, as for reducing \cref{M:OP} to \cref{M:FunCP}, this is challenging without additional constraints on \cref{P:SimpleConstant}, and so we will here instead reduce \cref{P:ConstantOP} directly to \cref{M:QCQP}. As we already know how to reduce \cref{P:SimpleConstant} to \cref{M:QCQP}, we will for $\InsMap$ map from an instance $(f)$ of \cref{P:ConstantOP} to an instance $(g)$ of \cref{P:SimpleConstant}, and then employ \cref{A:GeneralizedConstantToCP} to map this $(g)$ to an instance $(h)$ of \cref{M:QCQP}. \cref{A:SimplifyConstant} performs that initial map from $(f)$ to $(g)$, or rather, to be exact, from the graph $G$ corresponding to $(f)$ to the graph $G'$ corresponding to $(g)$.

\begin{algorithm}[htbp]
\caption{Map \cref{P:ConstantOP} instances to \cref{P:SimpleConstant} instances}\label{A:SimplifyConstant}
\KwData{\ac{UPGG} $G = (V,E,r,\caps,\supplies,\csupplies,\dfuns,\prodcpus)$ as in \cref{P:ConstantOP}}
\KwResult{\ac{SPGG} $G' := (V',A,r,\caps,\csupplies,\dfuns,\prodcpus)$ as in \cref{P:SimpleConstant}}
\begin{outline}[enumerate]
    \1 Define $G'=(V',A')$ as the directed version of $G=(V,E)$:
    \begin{align*}
        V' &:= V, \\
        A &:= \cup_{\{u,v\} \in E}\{(u,v),(v,u)\}.
    \end{align*}
    \1 $\forall{v \in V}:$
        \2 If $\prodcpus_v$ depends on out-flow rate, rather than cumulative outflow: For each non-empty maximal constant piece $I \subseteq [0,\supplies_v]$ of $\prodcpus_v$:
            \3 Add a $v^I$ to $V'$ and $\sources$, with $\prodcpus_{v^I}: [0,\abs{I}] \to \prodcpus(I)$.
            \3 Add $a:=(v^I,v)$ to $A$, with capacity $\caps_a:=\abs{I}$ and resistance $r_a:=0$.
        \2 Else if $\csupplies_{v} > 0$:
            \3 Add a $v^s$ to $\sources$ and $V'$, with $\prodcpus_{v^s}:=\prodcpus_{v}$.
            \3 Add $a:=(v^s,v)$ to $A$, with $\caps_a:=\supplies_v$ and $r_a:=0$.
        \2 Add a $v^d$ to $\sinks$ and $V'$, with $\dfuns_{v^d}:=\dfuns_{v}$.
        \2 \label{I:NodePSinkArc} Add $a:=(v,v^d)$ to $A$, with $\caps_{a} := 2\max_{t \in [0,\dur]}\dfuns_{v}(t)$ and $r_{a}:=0$.
    \1 Remove from $G'$ all nodes and arcs that are not part of some $\sources$-$\sinks$ path.
\end{outline}
\end{algorithm}

\begin{figure}[htbp]
    \centering
    \begin{subfigure}[t]{0.45\linewidth}
        \centering
        \begin{tikzpicture}
            \node[bus]      (u)                    {$u$};
            \node[bus]      (v)     [right=of u]   {$v$};
            
            \draw[-] (u) -- (v);
        \end{tikzpicture}
        \caption{Example instance of \cref{P:ConstantOP}.}
    \end{subfigure}
    \begin{subfigure}[t]{0.45\linewidth}
        \centering
        \begin{tikzpicture}
            \node[trans]      (u)                    {$u$};
            \node[trans]      (v)     [right=of u]   {$v$};
            \node[source]      (u^s)   [above=of u]     {$u^s$};
            \node[sink]      (u^d)   [left=of u]     {$u^\sink$};
            \node[source]        (v^{I})       [above=of v] {$v^{I}$};
            \node[source]        (v^{I'})       [below=of v] {$v^{I'}$};
            \node[sink]      (v^d)   [right=of v]     {$v^\sink$};
            
            \draw[->] (u) to [out=50,in=120] (v);
            \draw[->] (v) to [out=-140,in=-30] (u);
            \draw[->] (u^s) -- (u);
            \draw[->] (u) -- (u^d);
            \draw[->] (v^{I}) -- (v);
            \draw[->] (v^{I'}) -- (v);
            \draw[->] (v) -- (v^d);
        \end{tikzpicture}
        \caption{Corresponding output of \cref{A:SimplifyConstant}.}
    \end{subfigure}
    \caption{Input and output of \cref{A:SimplifyConstant} on a two-node graph with $V=\{u,v\}$. In this graph, $\prodcpus_u$ is a function of cumulative production while $\prodcpus_v$ is a function of production rate with one breakpoint.}
\end{figure}
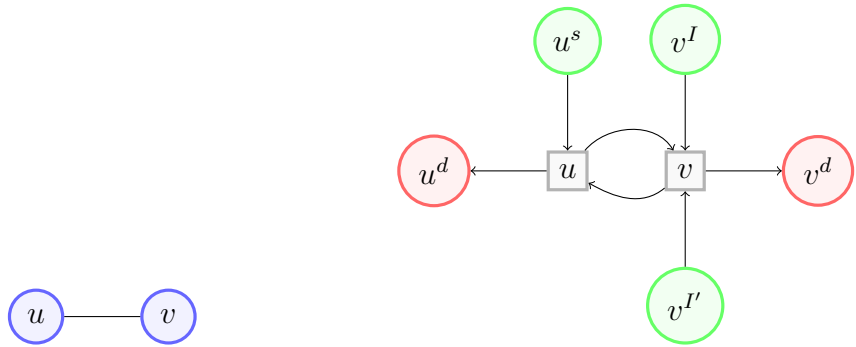

By similar arguments as in the proof of \cref{L:WasteIsReachableFromSource,Pr:HardCPReducesToCP}, \cref{A:SimplifyConstant} essentially acts as an $\InsMap$ from \cref{P:ConstantOP} to \cref{P:SimpleConstant}.

\begin{lemma}\label{L:SimplifyConstantIsInsMap}
    If $(f)$ is an instance of \cref{P:ConstantOP} with corresponding graph $G$, and $G'$ is the output of \cref{A:SimplifyConstant} on $G$. Then $G'$ corresponds to an instance $(g)$ of \cref{P:SimpleConstant}, and $\bits{(g)}$ and the execution time of obtaining this $(g)$ from $(f)$ is polynomial in $\bits{(f)}$.
\end{lemma}

As above, let $(f)$ be some instance of \cref{P:ConstantOP}, $(g)$ the corresponding instance of \cref{P:SimpleConstant} upon applying \cref{A:SimplifyConstant}, and $(h)$ the corresponding instance of \cref{M:QCQP} upon applying the whole $\InsMap$. To map a feasible point $z'$ of $(g)$ to a feasible point $z$ of $(f)$, we must merge all antiparallel arc flows, eliminate all the waste this introduces, and delete all entries of $z'$ corresponding to arcs added in \cref{A:SimplifyConstant}. As previously discussed, this can be hard to do for a feasible point of \cref{P:SimpleConstant}, yet it is highly doable for a feasible point $z''$ of $(h)$. In \cref{Pr:ConstantOPReducesToQCQP}, we will show that a valid $\SolMap$ for reducing \cref{P:ConstantOP} to \cref{M:QCQP} is to first call \cref{A:MergeAntiparallel}, then to call \cref{A:CPToGeneralizedConstant}, and finally to delete all set entries corresponding to arcs added in \cref{A:SimplifyConstant}.

\begin{algorithm}[htbp]
\caption{Merge antiparallel flows}\label{A:MergeAntiparallel}
\KwData{$G=(V,A)$ and $x,y \in \R_{\geq0}^{A}$}
\KwResult{$x,y$}
    \begin{outline}[enumerate]
        \1 $\forall{\text{antiparallel }a,a' \in A}$:
            \2 Define $\omega := \argmin_{q \in \{a,a'\}}y_q$ and $\Omega := \argmax_{q \in \{a,a'\}}y_q$ if $y_a \neq y_{a'}$ else $\omega=a$ and $\Omega=a'$
            \2 $x_\Omega \longleftarrow x_\Omega - y_\omega$
            \2 $y_\Omega \longleftarrow \outfuns_\Omega(x_\Omega)$
            \2 $x_\omega,y_\omega \longleftarrow 0$
    \end{outline}
\end{algorithm}

\begin{lemma}\label{L:MergingDoesNoHarm}
    If $(h)$ is an instance of \cref{M:QCQP} and $z=(x,y)$ is a feasible point of $(h)$, then applying \cref{A:MergeAntiparallel} to $z$ gives a feasible point $z'=(x',y')$ of $(h)$ with at most the same value as $z$.
\end{lemma}
\begin{proof}
    We will prove the statement by induction on the loop in \cref{A:MergeAntiparallel}. Denote the state right after performing $i$ iterations of the loop by $z^{i}$. Trivially, the value of $z^0$ equals the value of $z$. Moreover, by assumption, $z^0$ is feasible for $(h)$.
    
    Assume now that $z^{i-1}$ is feasible for $(h)$ and define $\omega,\Omega \in A$ as in iteration $i$ of the loop in \cref{A:MergeAntiparallel}. Then,
    \begin{align*}
        x_\omega^{i} &= 0 \leq x_\omega^{i-1}\\
        x_\Omega^{i} &= x_\Omega^{i-1} - y_\omega^{i-1} \leq x_\Omega^{i-1}\\
        x_a^{i} &= x_a^{i-1}, \ \forall{a \in A \setminus \{\omega,\Omega\}},
    \end{align*}
    hence the cost of $z^{i}$ is no higher than the cost of $z^{i-1}$.
    
    Pick now $u,v \in V$ such that $\omega=(u,v)$ and $\Omega=(v,u)$. Then,
    \begin{align*}
        &\vout{u}^{i}-\vin{u}^{i} \\
        &= \left[\vout{u}^{i-1} + \left(x_\omega^{i} - x_\omega^{i-1}\right)\right] - \left[\vin{u}^{i-1} + \left(y_\Omega^{i} - y_\Omega^{i-1}\right)\right] \\
        &= \vout{u}^{i-1} - \vin{u}^{i-1} + \left(y_\Omega^{i-1} - y_\Omega^{i} - x_\omega^{i-1}\right) \\
        &= \vout{u}^{i-1} - \vin{u}^{i-1} + \left(y_\Omega^{i-1} - \outfuns_\Omega(x_\Omega^{i}) - x_\omega^{i-1}\right) \\
        &= \vout{u}^{i-1} - \vin{u}^{i-1} + \left(y_\Omega^{i-1} - \outfuns_\Omega(x_\Omega^{i-1} - y_\omega^{i-1}) - x_\omega^{i-1}\right) \\
        &\leq \vout{u}^{i-1} - \vin{u}^{i-1} + \left(\outfuns_\Omega(x_\Omega^{i-1}) - \outfuns_\Omega(x_\Omega^{i-1} - y_\omega^{i-1}) - x_\omega^{i-1}\right) \\
        &\leq \vout{u}^{i-1} - \vin{u}^{i-1} + \left(y_\omega^{i-1} - x_\omega^{i-1}\right) \\
        &\leq \vout{u}^{i-1} - \vin{u}^{i-1} \\
        &\leq 0.
    \end{align*}
    Here, the first inequality made use of the sublinearity of $\outfuns_a$. Moreover,
    \begin{align*}
        &\vout{v}^{i}-\vin{v}^{i} \\
        &= \left[\vout{v}^{i-1} + \left(x_\Omega^{i} - x_\Omega^{i-1}\right)\right] - \left[\vin{v}^{i-1} + \left(y_\omega^{i} - y_\omega^{i-1}\right)\right] \\
        &= \vout{v}^{i-1} - \vin{v}^{i-1} + \left(x_\Omega^{i} - x_\Omega^{i-1} + y_\omega^{i-1}\right) \\
        &= \vout{v}^{i-1} - \vin{v}^{i-1} \\
        &\leq 0.
    \end{align*}
    Hence, $z^{i}$ is feasible for $(h)$.
    
    By induction, we conclude that the statement holds.
\end{proof}

\begin{algorithm}[htbp]
\caption{Map \cref{P:ConstantOP} instances to \cref{M:QCQP} instances}\label{A:ConstantOPToQCQP}
\KwData{\cref{P:ConstantOP} instance $(f)$ with corresponding \ac{UPGG} $G_{(f)}$.}
\KwResult{\cref{M:QCQP} instance $(g)$.}
    \begin{outline}[enumerate]
        \1 \label{I:FirstReduction} $G \longleftarrow$ \cref{A:SimplifyConstant} on $G_{(f)}$
        \1 \label{I:SecondReduction} $G' \longleftarrow$ \cref{A:ToStatic} on $G$
        \1 \label{I:MakeInstance} $(g) \longleftarrow$ \cref{M:QCQP} instance corresponding to $G'$
        \1 Return $(g)$
    \end{outline}
\end{algorithm}

\begin{figure}[htbp]
    \centering
    \begin{subfigure}[t]{0.9\linewidth}
        \centering
        \begin{tikzpicture}
            \node[bus](u){$u$};
            \node[bus](v)[right=of u]{$v$};
            
            \draw[-] (u) -- (v);
        \end{tikzpicture}
        \caption{Input $G_{(f)}$.}
    \end{subfigure}
    \par\bigskip
    \begin{subfigure}[t]{0.9\linewidth}
        \centering
        \begin{tikzpicture}
            \node[trans] (u) {$u$};
            \node[trans] (v) [right=of u] {$v$};
            \node[source] (u^s) [above=of u] {$u^s$};
            \node[sink] (u^d) [left=of u] {$u^\sink$};
            \node[source] (v^{I}) [above=of v] {$v^{I}$};
            \node[source] (v^{I'}) [below=of v] {$v^{I'}$};
            \node[sink] (v^d) [right=of v] {$v^\sink$};
            
            \draw[->] (u) to [out=60,in=120] (v);
            \draw[->] (v) to [out=-150,in=-30] (u);
            \draw[->] (u^s) -- (u);
            \draw[->] (u) -- (u^d);
            \draw[->] (v^{I}) -- (v);
            \draw[->] (v^{I'}) -- (v);
            \draw[->] (v) -- (v^d);
        \end{tikzpicture}
        \caption{$G$ obtained in Step \ref{I:FirstReduction}.}
    \end{subfigure}
    \par\bigskip
    \begin{subfigure}[t]{0.9\linewidth}
        \centering
        \begin{tikzpicture}
            \node[source](s^*){$s^*$};
            \node[trans](u_{I'}^s)[below left=of s^*]{$u_{I'}^s$};
            \node[trans](u_{I}^s)[left=of u_{I'}^s]{$u_{I}^s$};
            \node[trans](u_1^s)[below=of u_{I}^s]{$u_1^s$};
            \node[trans](u_2^s)[below=of u_{I'}^s]{$u_2^s$};
            \node[trans](u_1)[below=of u_1^s]{$u_1$};
            \node[trans](u_2)[below=of u_2^s]{$u_2$};
            \node[sink](u_1^d)[below=of u_1]{$u_1^d$};
            \node[sink](u_2^d)[below=of u_2]{$u_2^d$};
            \node[trans](v_{I}^{I})[below right=of s^*]{$v_{I}^{I}$};
            \node[trans](v_{I'}^{I'})[right=of v_{I}^{I}]{$v_{I'}^{I'}$};
            \node[trans](v_1^{I})[below left=of v_{I}^{I}]{$v_1^{I}$};
            \node[trans](v_2^{I})[below=of v_{I}^{I}]{$v_2^{I}$};
            \node[trans](v_1^{I'})[below=of v_{I'}^{I'}]{$v_1^{I'}$};
            \node[trans](v_2^{I'})[below right=of v_{I'}^{I'}]{$v_2^{I'}$};
            \node[trans](v_1)[below=of v_2^{I}]{$v_1$};
            \node[trans](v_2)[below=of v_1^{I'}]{$v_2$};
            \node[sink](v_1^d)[below=of v_1]{$v_1^d$};
            \node[sink](v_2^d)[below=of v_2]{$v_2^d$};
            
            \draw[->] (s^*) -- (u_{I}^s);
            \draw[->] (s^*) -- (u_{I'}^s);
            \draw[->] (u_{I}^s) -- (u_1^s);
            \draw[->] (u_{I}^s) -- (u_2^s);
            \draw[->] (u_{I'}^s) -- (u_1^s);
            \draw[->] (u_{I'}^s) -- (u_2^s);
            \draw[->] (u_1^s) -- (u_1);
            \draw[->] (u_2^s) -- (u_2);
            \draw[->] (u_1) -- (u_1^d);
            \draw[->] (u_2) -- (u_2^d);
            \draw[->] (s^*) -- (v_{I}^{I});
            \draw[->] (s^*) -- (v_{I'}^{I'});
            \draw[->] (v_{I}^{I}) -- (v_1^{I});
            \draw[->] (v_{I}^{I}) -- (v_2^{I});
            \draw[->] (v_{I'}^{I'}) -- (v_1^{I'});
            \draw[->] (v_{I'}^{I'}) -- (v_2^{I'});
            \draw[->] (v_1^{I}) -- (v_1);
            \draw[->] (v_1^{I'}) -- (v_1);
            \draw[->] (v_2^{I}) -- (v_2);
            \draw[->] (v_2^{I'}) -- (v_2);
            \draw[->] (v_1) -- (v_1^d);
            \draw[->] (v_2) -- (v_2^d);
            \draw[->] (u_1) to [out=25,in=155] (v_1);
            \draw[<-] (u_1) to [out=-25,in=205] (v_1);
            \draw[->] (u_2) to [out=25,in=155] (v_2);
            \draw[<-] (u_2) to [out=-25,in=205] (v_2);
        \end{tikzpicture}
        \caption{$G'$ obtained in Step \ref{I:SecondReduction}.}
    \end{subfigure}
    \caption{Input and output of \cref{A:ConstantOPToQCQP} on a two-node graph with $V=\{u,v\}$ and one breakpoint in $\dfuns$. In this graph, $\prodcpus_u$ is a function of cumulative production with one breakpoint while $\prodcpus_v$ is a function of production rate with one breakpoint.}
\end{figure}
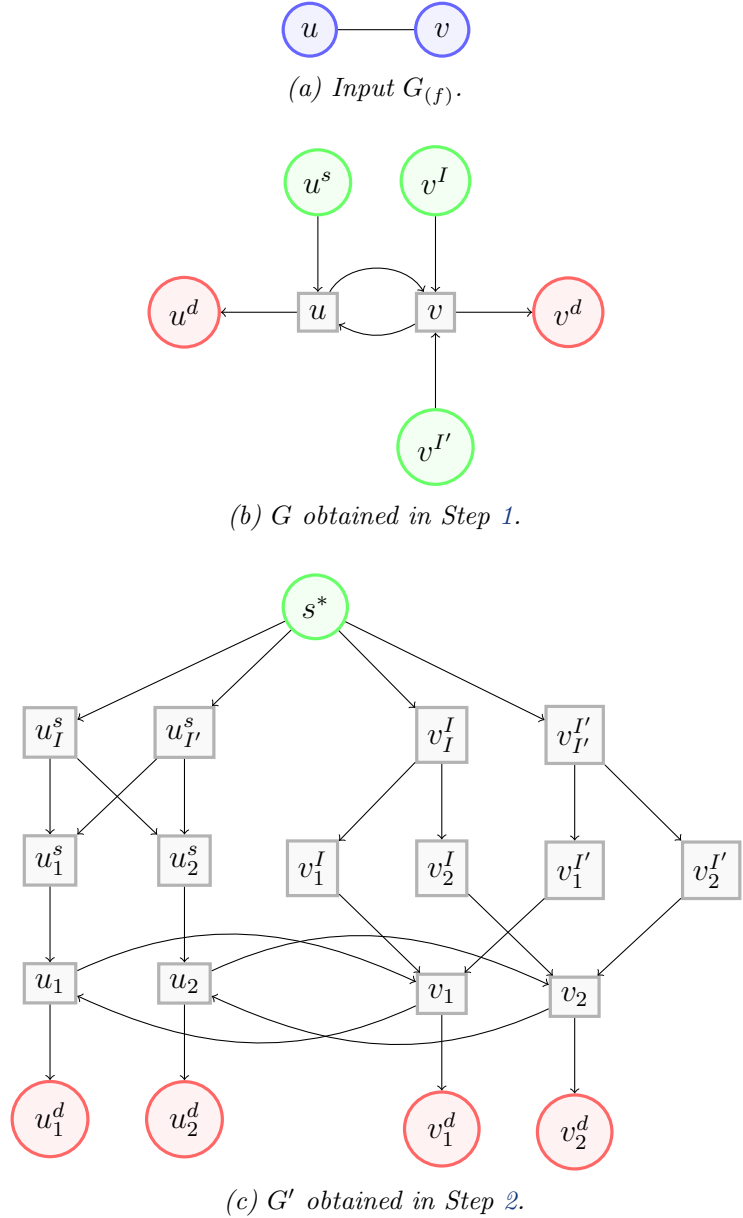

\begin{algorithm}[htbp]
\caption{Map feasible points of \cref{M:QCQP} instances to feasible points of \cref{P:ConstantOP} instances}\label{A:QCQPToConstantOP}
\KwData{\cref{M:QCQP} instance $(f)$ and $(f)$-feasible $z=(x,y)$.}
\KwResult{$\otild{z}$}
    \begin{outline}[enumerate]
        \1 \label{I:Merge} $z' \longleftarrow$ \cref{A:MergeAntiparallel} on $z$ and $G$
        \1 \label{I:Round} $z'' \longleftarrow$ \cref{A:Round} on $z'$ and $G$
        \1 \label{I:Purge} $z''' \longleftarrow$ $\forall v \in \text{Neighbors}(\text{Neighbors}(\ssource)) \cup \sinks: \forall a \in \delta_v:$ del $z''_a$
        \1 \label{I:VecToFunc} $\otild{z} \longleftarrow$ \eqref{E:PiecewiseConstantFromVec} with $z'''$
    \end{outline}
\end{algorithm}

\begin{proposition}\label{Pr:ConstantOPReducesToQCQP}
    \cref{P:ConstantOP} reduces to \cref{M:QCQP} through $(\text{\cref{A:ConstantOPToQCQP}}, \text{\cref{A:QCQPToConstantOP}})$.
\end{proposition}
\begin{proof}
    As noted for \cref{A:GeneralizedConstantToCP}, we can skip the reductions that do not modify the underlying graph. \cref{A:ConstantOPToQCQP} is the result of prepending a (redundant) step to \cref{A:GeneralizedConstantToCP} that maps a \cref{P:ConstantOP} instance to a \cref{P:SimpleConstant} instance, and then skipping all redundant steps. Hence, by \cref{Pr:GeneralizedConstantReducesToCP}, we know that \cref{A:ConstantOPToQCQP} is a valid $\InsMap$ for reducing from \cref{P:ConstantOP} to \cref{M:QCQP}.

    Assume now that we are given some instance $(f)$ of \cref{P:ConstantOP} and that $(f'),(f'')$, and $(f''')$ are the \cref{M:OP,M:FunCP,M:HardCP} instances implicitly produced during execution of \cref{A:ConstantOPToQCQP}. Assume moreover that we are given an $(f''')$-feasible $z$. By \cref{L:MergingDoesNoHarm,L:RoundDoesNoHarm}, we know that Steps \ref{I:Merge} and \ref{I:Round} of \cref{A:QCQPToConstantOP} run in time polynomial in $\bits{(f)}$ and that $z''$ is $(f''')$-feasible with at most the same value as $z$. Hence, if Step \ref{I:Purge} of \cref{A:QCQPToConstantOP} is replaced by
    \begin{equation*}
        z''' \longleftarrow \ \forall v \in \text{Neighbors}(\ssource): \forall a \in \delta_v: \text{ del } z''_a,
    \end{equation*}
    which is equivalent to
    \begin{equation*}
        z''' \longleftarrow \ \otild{\CumPurger}(z''),
    \end{equation*}
    then $\otild{z}$ is $(f'')$-feasible, by \cref{Pr:FunCPReducesToHardCP,Pr:HardCPReducesToCP}. However, by \cref{L:RoundDoesNoHarm}, $\otild{z}$ is then also $(f')$-feasible. Moreover, by \cref{L:MergingDoesNoHarm,L:RoundDoesNoHarm}, $\otild{z}$ has no pair of antiparallel arcs with flow on both arcs and the $(f')$-value of $\otild{z}$ is no higher than the $(f''')$-value of $z$. As the extra entries deleted in the true Step \ref{I:Purge} of \cref{A:QCQPToConstantOP} correspond to those added in the $\InsMap$ from \cref{P:ConstantOP} to \cref{P:SimpleConstant}, we conclude that applying \cref{A:QCQPToConstantOP} to $z$ returns an $(f)$-feasible $\otild{z}$ of no higher $(f)$-value than the $(f''')$-value of $z$.

    Finally, by \cref{L:ToStaticPreservesValue} and the constructions of \cref{P:ConstantOP,P:SimpleConstant} and \crefrange{M:OP}{M:HardCP}, we know that $(f)$ and $(f''')$ have the same optimal value. We thus conclude that $(\text{\cref{A:ConstantOPToQCQP}}, \text{\cref{A:QCQPToConstantOP}})$ is a reduction from \cref{P:ConstantOP} to \cref{M:QCQP}.
\end{proof}

\subsection{Min-cost production distribution at a single point in time}
Another question that may be of interest is to optimize production to minimize instantaneous costs at a single point in time, rather than over time. In this case we have no cumulative supplies, and $\prodcpus$ is instead only based on the instantaneous production rates. As before, we lose no generality by assuming that $G$ is directed, $V$ is reachable from $S$, sources have no in-arcs, and sinks have no out-arcs.

\begin{problem}[Min-cost production distribution at a single point in time]\label{P:StaticOP}
    Given: a power grid graph $G=(V,A,r,\caps,\supplies,\demands,\prodcpus)$ at a single point in time and nondecreasing and piecewise constant $\prodcpus: [0,\supplies] \to \R_+$ with finitely many breakpoints, where every source $s \in S \subseteq V$ has no in-arcs, every $\sink \in \sinks \subseteq V$ has no out-arcs, and every $v \in V$ is reachable from $S$.
    
    Task: find a min-cost production distribution and generalized flow $x$ that satisfies demands $\demands$ and respects all constraints and capacities.
\end{problem}

At first thought one would maybe think that \cref{P:SimpleConstant} reduces to \cref{P:StaticOP}, by acting according to \cref{P:StaticOP} at every time-step $t\in[0,\dur]$. This is however not the case, as the cumulative constraints of \cref{P:SimpleConstant} make the instantaneous costs and constraints in later time-steps dependent on the decisions taken in all earlier time-steps, as a solution that in each time-step acts greedily according to \cref{P:StaticOP} can in later time-steps find itself in an optimization problem with no good solution, thus becoming highly suboptimal. \cref{Ex:GreedyIsSuboptimal} demonstrates a simple example of this situation.

\begin{example}\label{Ex:GreedyIsSuboptimal}
    Assume $\dur=2$ and $G$ is as in \cref{P:SimpleConstant}, with $V=\sources\cup\sinks=\{s_1,s_2\}\cup\{d\}$, $A=\{(s_1,d),(s_2,d)\}$, $\dfuns_d:t \mapsto 1$, $\caps_{(s_1,d)}=1$, $\caps_{(s_2,d)}=\frac{1}{2}$, $\forall{a \in A}: r_a = 0$ (so $\outfuns_a=\Id$),
    \begin{equation*}
        \prodcpus_{s_1}: [0,2] \ni \beta \mapsto
        \begin{cases}
            1, &\text{if } \beta<1,\\
            L, &\text{if } \beta \geq 1,
        \end{cases}
    \end{equation*}
    with $L>2$, and
    \begin{equation*}
        \prodcpus_{s_2}: [0,2] \ni \beta \mapsto 2.
    \end{equation*}
    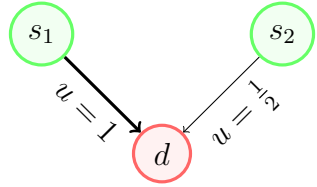
\begin{figure}[H]
        \centering
        \begin{tikzpicture}
            \node[sink]        (d)        {$d$};
            \node[source]      (s_1)     [above left=of d]   {$s_1$};
            \node[source]      (s_2)     [above right=of d]   {$s_2$};
            
            \draw[very thick, ->] (s_1) -- (d) node[pos=.5,sloped,below] {$\caps = 1$};
            \draw[->] (s_2) -- (d) node[pos=.5,sloped,below] {$\caps = \frac{1}{2}$};
        \end{tikzpicture}
        \caption{Graph of \cref{P:SimpleConstant} instance used in \cref{Ex:GreedyIsSuboptimal}. Arc thickness illustrates capacity.}
    \end{figure}
    
    Here $s_1$ is cheaper than $s_2$ until the cumulative supply of $s_1$ reaches $1$, at which point $s_2$ is cheaper than $s_1$ but has too little capacity on its out-arc to satisfy $d$ by itself. The solution that at every $t \in [0,\dur]$ acts greedily on \cref{P:StaticOP} will not look ahead, thus will output
    \begin{equation*}
        (x_{s_1}(t),x_{s_2}(t))=
        \begin{cases}
            (1,0), &\text{if } t<1,\\
            (\frac{1}{2},\frac{1}{2}), &\text{if } t \geq 1,
        \end{cases}
    \end{equation*}
    for a cumulative cost of
    \begin{equation*}
        1+(\frac{L}{2}+\frac{2}{2})=2+\frac{L}{2}.
    \end{equation*}
    
    The optimal solution will take into account that surpassing a cumulative supply of $1$ on $s_1$ is very expensive, and will thus act conservatively to ensure this never happens, e.g. through
    \begin{equation*}
        (x_{s_1}(t),x_{s_2}(t))=(\frac{1}{2},\frac{1}{2}),
    \end{equation*}    
    which has a cumulative cost of
    \begin{equation*}
        2\cdot(\frac{1}{2}+\frac{2}{2})=3.
    \end{equation*}
    
    We thus see that we can make the solution that in every time-step acts greedily on \cref{P:StaticOP} arbitrarily suboptimal for \cref{P:SimpleConstant} by making $L$ arbitrarily large.
\end{example}

The other direction however partially holds, with \cref{P:StaticOP} reducing to \cref{P:SimpleConstant} if we to the latter add the additional constraint that the solution must be constant, e.g. if we solve it by \cref{M:CP}. Such a reduction is achieved by starting at an instance $G$ of \cref{P:StaticOP} and creating an identical instance for \cref{P:SimpleConstant} with constant demands, $\dur:=1$, and $\csupplies:=\supplies$, as a constant optimal solution of \cref{P:SimpleConstant} will then solve \cref{P:StaticOP} at every $t \in [0,\dur]$.
    \section{Approximately solving Problem \ref{P:GeneralizedConstant}}\label{S:Boyd} Having successfully deduced a reduction from \cref{P:GeneralizedConstant} to \cref{M:CP}, we will now look at how \cref{M:CP} instances can be efficiently solved using the \textit{barrier method}. However, by the nature of the barrier method, we must introduce the additional requirement that these instances are \textit{strictly feasible}.

\begin{definition}[Strictly feasible]
    If $(f)$ is some program instance and $x$ is a feasible point of $(f)$, then we call $x$ a \textit{strictly feasible} point (of $(f)$) if it satisfies all inequalities constraints of $(f)$ strictly. If $(f)$ has a strictly feasible point, we say that $(f)$ is \textit{strictly feasible}. 
\end{definition}

In the case of \cref{M:CP}, a strictly feasible point corresponds to a generalized flow satisfying all demands and
\begin{equation}
\begin{array}{@{}r@{\;}ll}
    0 < y_a < \outfuns_a(x_a) &< \caps_a, &\forall{a \in A},\\
    \vin{v} &> \vout{v}, &\forall{v \in V \setminus (\sources \cup \sinks)},\\
    \vin{\sink} &> \demands_\sink, &\forall{\sink \in \sinks}.
\end{array}
\end{equation}
Simple reasoning then tells us that a \cref{M:CP} instance is strictly feasible iff it is possible to satisfy all its demands without fully utilizing any of its arcs.

\subsection{General runtime analysis with barrier method}

The following subsection is entirely based on \cite{boyd2004convex}.

\begin{program}[Three times differentiable CP]\label{M:General}
Given: Convex and three times differentiable $f_0,f_1,\ldots,f_m: \R^n \to \R$ and affine $h$.
\begin{mini*}{z}{f_0(z)}{}{}
    \addConstraint{f_i(z)}{\leq 0, \quad}{i \in [m]}
    \addConstraint{h(z)}{= 0}
\end{mini*}
\end{program}

Assume we are given a strictly feasible instance $(f)$ of \cref{M:General}. The barrier method then solves $(f)$ in two phases, with the first phase really just being the second phase applied to an auxiliary program instance.

Phase II starts at a strictly feasible point of $(f)$, and returns an $(f)$-feasible point $x$ such that
\begin{equation*}
    \val[(f)]{x} \leq \OPT[(f)] + \varepsilon.
\end{equation*}
Roughly speaking, it achieves this by solving a sequence of gradually closer approximations of $(f)$, using the solution of one program as the starting point for the next. Specifically, for a sequence
\begin{equation*}
    (t^{(0)},\mu t^{(0)},\mu^2 t^{(0)},\ldots,\mu^l t^{(0)}),
\end{equation*}
where
\begin{align*}
     t^{(0)} &> 0, \\
    \mu &= 1+1/\sqrt{m+1}, \\
    l &= \ceil*{\frac{\log{m}-\log{(\varepsilon t^{(0)})}}{\log{\mu}}},
\end{align*}
it employs Newton's method with backtracking line search to solve the corresponding instance $(f_{(t)})$ of \cref{M:CentralPath}.

\begin{program}[Central path program]\label{M:CentralPath}
Given: $t>0$ and $f_0,f_1,\ldots,f_m,h$ as in \cref{M:General}.
\begin{mini*}{z}{tf_0(z)-\sum_{i \in [m]}{\log(-f_i(z))}}{}{}
    \addConstraint{h(z)}{= 0}
\end{mini*}
\end{program}

If $z$ is the solution of $(f_{(t)})$ for $t=t'$ we say that it is on the \textit{central path} of $(f)$ and set $z_{(t')}:=z$.

As finding a strictly feasible initial point of $(f)$ is not always trivial, we first perform a Phase I to find such a point. The algorithm in Phase I is the same as the algorithm in Phase II, solving a sequence of central path programs. However, the difference is that in Phase I we have replaced $(f)$ with a corresponding instance $(g)$ of some auxiliary program. We want $(g)$ to be an instance of \cref{M:General}, have an easily obtainable strictly feasible point, and have an orthogonal projection of its central path intersect with that of $(f)$. We achieve all these objective by making $(g)$ an instance of \cref{M:PhaseI}.

\begin{program}[Phase I]\label{M:PhaseI}
Given: $f_0,f_1,\ldots,f_m,h$ as in \cref{M:General} and appropriately chosen $M$ and $a$.
\begin{mini*}{z,s}{s}{}{}
    \addConstraint{f_i(z)}{\leq s, \quad}{i \in [m]}
    \addConstraint{h(z)}{= 0}
    \addConstraint{f_0(z)}{\leq M}
    \addConstraint{a^Tx}{\leq 1}.
\end{mini*}
\end{program}

As adding redundant constraints to a program can change its corresponding central path programs, $a,s^{(0)}$ and $\tilde{t}^{(0)}$ are here chosen such that $a^Tx \leq 1$ is a redundant constraint for $(g)$ that makes $(0,s^{(0)})$ an optimal solution of $(g_{(\tilde{t}^{(0)})})$. The interested reader is referred to \cite{boyd2004convex} for details of exactly how such $a,s^{(0)},\tilde{t}^{(0)}$ are computed. For $M$, we are free to choose any number larger than $\OPT[(f)]$. Introducing the constraint $f_0(z) \leq M$ then ensures that there exists a reasonably small $\tilde{t}$ and a reasonably large $t$ such that if $(z,0)$ is an optimal solution of $(g_{(\tilde{t})})$, then $z$ is an optimal solution of $(f_{(t)})$.

By starting Phase I at $(0,s^{(0)})$, we thus ensure that we in Phase I are always on the central path of $(g)$ and then in Phase II are always on the central path of $(f)$. Relying on this, we can under some extra assumptions compute an upper bound on both
\begin{equation*}
    f_0(z_{(t)}) - \OPT[(f)]
\end{equation*}
and on the number of steps of Newton's method with backtracking line search required to for any $t>0$ solve $(f_{(\mu t)})$ to within some absolute error $\varepsilon'$ when starting at $z_{(t)}$. With these upper bounds, we can hence upper bound the number of Newton's method with backtracking line search steps required in Phase II. As the same holds for Phase I, we add these upper bounds together to arrive at an upper bound on the steps of Newton's method with backtracking line search required by the barrier method to solve $(f)$ to within some absolute error $\varepsilon$.

To be exact, the extra assumptions required to obtain these bounds are that every sublevel set of $(f)$ and $(g)$ are bounded and that the objective function of every visited $(f_{(t)})$ and $(g_{(t)})$ is self-concordant and closed.

\begin{definition}[Sublevel set]
    Given an instance $(f)$ of \cref{M:General}, the $\alpha$-sublevel set is the set of all feasible $z$ such that $\val[(f)]{z} \leq \alpha$. In other words, it is the set
    \begin{equation*}
        \{z \given f_0(z) \leq \alpha \a h(z)=0 \a \forall{i \in [m]}: f_i(z) \leq 0\}
    \end{equation*}
\end{definition}

\begin{definition}[Closed function]
    A function $f$ is closed if all its sublevel sets are closed.
\end{definition}

\begin{definition}[Self-concordant function]
    A three times differentiable function $f: \R \to \R$ is self-concordant if $\abs{f'''} \leq 2(f'')^{3/2}$.
\end{definition}

\begin{definition}[Self-concordant function]
    A three times differentiable function $f: \R^n \to \R$ is self-concordant if its restriction to any line is self-concordant.
\end{definition}

Given these assumptions and the definitions
\begin{equation}
\begin{split}
    \MaxGradNorm &:= \max_{i \in [m]}\norm{\nabla f_i(0)}_2, \\
    \MaxPointNorm &:= \max\{\norm{z}_2 \given h(z) = 0 \a \forall{i \in [m]}: f_i(z) \leq 0\},
\end{split}
\end{equation}
the total number of Newton's method with backtracking line search steps performed before the barrier method finds a feasible $z$ for which
\begin{equation*}
    f_0(z) \leq \OPT[(f)]+\varepsilon
\end{equation*}
is upper bounded by
\begin{equation}\label{E:BarrierNewtonBound}
\begin{split}
    &\ceil*{\sqrt{m+2}\log_2{\frac{(m+1)(m+2)\MaxGradNorm\MaxPointNorm}{\abs{\OPT[(g)]}}}}(\frac{1}{2\zeta}+\xi) \\
    &+ \ceil*{\sqrt{m+1}\log_2{\frac{(m+1)(M-\OPT[(f)])}{\varepsilon}}}(\frac{1}{2\zeta}+\xi) \\
    &= \BigOh(\sqrt{m}\log{\frac{m\MaxGradNorm\MaxPointNorm(M-\OPT[(f)])}{\abs{\OPT[(g)]}\varepsilon}}).
\end{split}
\end{equation}
Here, $\zeta$ depends on the parameters used by Newton's method and $\xi$ depends on the accuracy to which we require that each central path program is solved. We treat both of these as fixed and constant, as they are not of much interest to us. As previously mentioned, $m$ is the number of inequality constraints of $(f)$, while $M$ is some upper bound of $\OPT[(f)]$ chosen by us.

Note how \eqref{E:BarrierNewtonBound} is thus infinite unless
\begin{equation*}
    \{z \given h(z) = 0 \a \forall{i \in [m]}: f_i(z) \leq 0\}
\end{equation*}
is bounded. Moreover, note how the upper bound here is in terms of the number of steps of Newton's method with backtracking line search, not in terms of arithmetic operations.

Additionally, for completeness, please note that in the derivation in \cite{boyd2004convex} of \eqref{E:BarrierNewtonBound}, every central path program is only solved approximately, yet it is assumed that every intermediary central path program gets the exact solution to the previous central path program as its starting point, rather than the approximate solution actually found. However, as Newton's method with backtracking line search gets close to the optimum, it enters a region of quadratic convergence. Hence, the upper bound on the number of steps Newton's method needs to beat $\varepsilon'$ only depends on $\varepsilon'$ through the term $\log\log\varepsilon'$. We can thus essentially treat this $\varepsilon'$ as not factoring into the complexity and solve these central path programs to whatever accuracy it takes to make their approximate solutions practically indistinguishable from the corresponding optimal solutions.

With those caveats in mind, we summarize our discussion in the form of \cref{T:BarrierNewtonBound}.

\begin{theorem}\label{T:BarrierNewtonBound}
    Assume $f_0,f_1,\ldots,f_m,h$ correspond to a strictly feasible \cref{M:General} instance $(f)$, with a corresponding \cref{M:PhaseI} instance $(g)$, such that
    \begin{itemize}
        \item All sublevel set of $(f)$ and $(g)$ are bounded,
        \item All central path programs of $(f)$ and $(g)$ have self-concordant and closed objective functions,
    \end{itemize}
    \begin{align*}
        \MaxGradNorm &= \max_{i \in [m]}\norm{\nabla f_i(0)}_2, \\
        \MaxPointNorm &= \max\{\norm{z}_2 \given h(z) = 0 \a \forall{i \in [m]}: f_i(z) \leq 0\}, \\
        M &\geq \OPT[(f)].
    \end{align*}
    Then, the barrier method obtains an $(f)$-feasible $z$ with $f_0(z)-\OPT[(f)] \leq \varepsilon$ in
    \begin{equation}
        \BigOh(\sqrt{m}\log{\frac{m\MaxGradNorm\MaxPointNorm(M-\OPT[(f)])}{\abs{\OPT[(g)]}\varepsilon}})
    \end{equation}
    steps of Newton's method with backtracking line search.
\end{theorem}

Finally worth noting, if we run the barrier method on a \cref{M:General} instance that is infeasible but otherwise satisfies the conditions of \cref{T:BarrierNewtonBound}, then the same complexity bound holds for obtaining a certificate of infeasibility.

\subsection{Applied to a generalized flow program}\label{SS:BarrierOnGeneralizedFlow}

\cref{M:General} is a generalization of \cref{M:CP}, so we will here combine \cref{T:BarrierNewtonBound} with the unique properties of \cref{M:CP} to arrive at a stronger corollary. Hence, assume $(f)$ is a \cref{M:CP} instance and $(g)$ is its corresponding \cref{M:PhaseI} instance.

Pick an arbitrary $a=(u,v) \in A$. Then $x_a$ is involved in the constraints
\begin{align*}
    y_a - \outfuns_a(x_a) &\leq 0, \\
    x_a - \caps_a &\leq 0, \\
    \vout{u} - \vin{u} &\leq 0, \\
\end{align*}
while $y_a$ is involved in the constraints
\begin{align*}
    y_a - \outfuns_a(x_a) &\leq 0 \\
    -y_a &\leq 0,
\end{align*}
and either
\begin{equation*}
     \vout{v} - \vin{v} \leq 0
\end{equation*}
or
\begin{equation*}
     \demands_v - \vin{v} \leq 0.
\end{equation*}
As $0 \leq \outfuns_a \leq \Id$, we note that $\outfuns_a(0)=0$, hence that
\begin{equation*}
    0 \leq \at{\frac{\partial \outfuns_a(x)}{\partial x}}{x=0} \leq 1.
\end{equation*}
Thus,
\begin{equation}
    \begin{cases}
    \at{\frac{\partial f_i(z)}{\partial x_a}}{z=0} \in
    \begin{cases}
        [-1,0]\cup\{1\}, &\text{ if $f_i$ involves $x_a$}, \\
        \{0\}, &\text{ otherwise},
    \end{cases}\\
    \at{\frac{\partial f_i(z)}{\partial y_a}}{z=0} \in
    \begin{cases}
        \{-1,1\}, &\text{ if $f_i$ involves $y_a$}, \\
        \{0\}, &\text{ otherwise}.
    \end{cases}
\end{cases}
\end{equation}
We can hence conclude that
\begin{equation}
    \MaxGradNorm \leq \norm{(\underbrace{1,1,\ldots,1}_{2\abs{A}})}_2 = \sqrt{2\abs{A}}.
\end{equation}

From the properties and structure of \cref{M:CP} we moreover know that
\begin{equation*}
    \OPT[(f)] > 0
\end{equation*}
and
\begin{equation*}
    m = \BigOh(\abs{A}).
\end{equation*}
Also, as
\begin{equation}\label{E:BoundedFeasibleRegion}
    0 \leq y_a \leq x_a \leq \caps_a < \infty, \ \forall{a \in A},
\end{equation}
we observe that
\begin{equation*}
    \norm{(x,y)}_2 \leq \norm{(\caps,\caps)}_2 \leq \sqrt{2}\norm{\caps}_2 \leq \sqrt{2\abs{A}}\caps_{max},
\end{equation*}
hence conclude that
\begin{equation}
    \MaxPointNorm \leq \sqrt{2\abs{A}}\caps_{max}.
\end{equation}
Finally, as we assumed $(f)$ is feasible, the demands must be satisfiable by fully utilizing every arc, so
\begin{equation}
    M \leq \sum_{a \in A}\costs_a(\caps_a) \leq \abs{A}\max_{a \in A}{\costs_a(\caps_a)} =: \abs{A}\costs_{max}.
\end{equation}
As $\MaxPointNorm \leq \infty$, we can by the definitions of sublevel sets and $\MaxPointNorm$ hence conclude that every sublevel set of $(f)$ is bounded. It turns out that all sublevel sets of $(g)$ are also bounded, but deducing this requires a more careful analysis, such as that shown in \cref{L:BoundedPhaseISublevelSets}.

\begin{lemma}\label{L:BoundedPhaseISublevelSets}
    If $(f)$ is a \cref{M:CP} instance and $(g)$ is its corresponding \cref{M:PhaseI} instance, then all sublevel sets of $(g)$ are bounded.
\end{lemma}
\begin{proof}
    Assume for contradiction that a $(g)$-feasible $(z,s)$ with $s < -\caps_{max}$ exists. Pick $a \in A$ arbitrarily. By the definitions of \cref{M:CP,M:PhaseI}, we know that
    \begin{equation*}
        y_a \geq -s > \caps_{max} \geq 0.
    \end{equation*}
    However, we by these definitions also know that
    \begin{equation*}
        x_a \leq \caps_{a} + s < \caps_{a} - \caps_{max} \leq 0
    \end{equation*}
    and
    \begin{equation*}
        y_a \leq \outfuns_a(x_a) + s < x_a - \caps_{max} \leq 0,
    \end{equation*}
    hence that
    \begin{equation*}
        y_a < -\caps_{max} \leq 0.
    \end{equation*}
    We have arrived at the contradiction $0 < y_a < 0$, and thus conclude that any $(g)$-feasible $(z,s)$ must have $s \geq -\caps_{max}$.

    For arbitrary $s$, we find that
    \begin{equation*}
        -s \leq y_a \leq \outfuns_a(x_a) + s \leq x_a + s \leq \caps_a + 2s, \ \forall{a \in A},
    \end{equation*}
    hence that necessary (but not sufficient) conditions for $(z,s)$ to be $(g)$-feasible is that
    \begin{align*}
        x_a &\in [-2s, \caps_a + s], \ \forall{a \in A} \\
        y_a &\in [-s, \caps_a + 2s], \ \forall{a \in A}.
    \end{align*}

    Combining our results, we find that the $\alpha$-sublevel set of $(g)$ is a subset of
    \begin{equation*}
        [-2\alpha, \caps_{max} + \alpha]^{A} \times [-\alpha, \caps_{max} + 2\alpha]^{A} \times [-\caps_{max}, \alpha],
    \end{equation*}
    which is clearly bounded. We hence conclude that all sublevel sets of $(g)$ are bounded.
\end{proof}

We combine these findings with \cref{T:BarrierNewtonBound} to arrive at the stronger \cref{C:BarrierNewtonBound}.

\begin{corollary}\label{C:BarrierNewtonBound}
    Assume $(f)$ is a strictly feasible \cref{M:CP} instance, $(g)$ is its corresponding \cref{M:PhaseI} instance, and all central path programs of $(f)$ and $(g)$ have self-concordant and closed objective functions.
    
    Then, the barrier method obtains an $(f)$-feasible $z$ with
    \begin{equation*}
        \val[(f)]{z} - \OPT[(f)] \leq \varepsilon
    \end{equation*}
    in
    \begin{equation}
        \BigOh(\sqrt{\abs{A}}\log{\frac{\abs{A}\costs_{max}\caps_{max}}{\abs{\OPT[(g)]}\varepsilon}})
    \end{equation}
    steps of Newton's method with backtracking line search.
\end{corollary}

\subsubsection{For our convex \acs{QCQP}}\label{SS:QuadraticLosses}

Assume now the even more special case of $(f)$ being an instance of \cref{M:QCQP}, with $(g)$ being its corresponding \cref{M:PhaseI} instance. Then, every central path programs of $(f)$ and $(g)$ has a self-concordant objective function \citep{boyd2004convex}. \cite{boyd2004convex} also tells us that if a function is continuous, finite on an open domain, and tends to infinity as one approaches the boundary of that domain, then it is also closed. As this is the case for the objective functions of all central path programs of $(f)$ and $(g)$, we conclude that all of these objective functions are closed. We can hence combine these findings with \cref{C:BarrierNewtonBound} to arrive at the even stronger \cref{C:QCQPBarrierNewtonBound}.

\begin{corollary}\label{C:QCQPBarrierNewtonBound}
    Assume $(f)$ is a strictly feasible \cref{M:QCQP} instance and $(g)$ is its corresponding \cref{M:PhaseI} instance.
    
    Then, the barrier method obtains an $(f)$-feasible $z$ with
    \begin{equation*}
        \val[(f)]{z} - \OPT[(f)] \leq \varepsilon
    \end{equation*}
    in
    \begin{equation}
        \BigOh(\sqrt{\abs{A}}\log{\frac{\abs{A}\costs_{max}\caps_{max}}{\abs{\OPT[(g)]}\varepsilon}})
    \end{equation}
    steps of Newton's method with backtracking line search.
\end{corollary}

    \section{An FPTAS for Problem \ref{P:ConstantOP}}\label{S:Nesterov} As already noted, \cref{T:BarrierNewtonBound,C:BarrierNewtonBound,C:QCQPBarrierNewtonBound} provide bounds on the number of steps of Newton's method with backtracking line search required by the barrier method, but not on the actual runtime complexity in terms of algorithmic operations. \cite{nesterov1994interior} provides and analyses a related algorithm, which we will call the \textit{path-following method}. This method finds almost-feasible almost-optimal points of a convex \ac{QCQP} instance in polynomial time, even if the instance is only feasible but not strictly feasible. In this section, we will first outline that method, and then combine it with the unique properties of \cref{M:QCQP}, to under a mild assumption deduce an \ac{FPTAS} for \cref{M:QCQP}, hence for \cref{P:ConstantOP}.

\subsection{For the general convex \acs{QCQP}}

The following subsection is entirely based on \cite{nesterov1994interior}.

\begin{program}[Bounded and convex \acs{QCQP}]\label{M:ArbitraryQCQP}
Given: $f_0,f_1,\ldots,f_m$ and $\sigma \geq 0$, such that with
\begin{equation*}
    B := \{z\in\R^n : \norm{z}_2 \leq \sigma\},
\end{equation*}
we find
\begin{align*}
    &\forall{i \in \{0,1,\ldots,m\}}: \left\{
    \begin{array}{@{}r@{\;}l}
        A_i &\in \R^{n \times n} \text{ symmetric and positive-semidefinite}, \\
        b_i &\in \R^n, \\
        f_i&: B \ni z \mapsto \frac{1}{2} z^T A_i z - b_i^T z + c_i \in \R,
    \end{array} \right.
\end{align*}
and
\begin{equation*}
    \{z \in B : \forall{i \in [m]}: f_i(z) \leq 0\} \neq \emptyset.
\end{equation*}

\begin{mini*}{z}{f_0(z)}{}{}
    \addConstraint{f_i(z)}{\leq 0, \quad}{i \in [m]}.
\end{mini*}
\end{program}

Assume we are given a feasible \cref{M:ArbitraryQCQP} instance $(f)$, hence some feasible, bounded, convex \ac{QCQP} instance. The path-following method will then give us an \textit{$\varepsilon$-solution} to $(f)$.

\begin{definition}[$\varepsilon$-solution]
    If $(f)$ is a \cref{M:ArbitraryQCQP} instance with functions $f_0,\ldots,f_m$, then an $\varepsilon$-solution to $(f)$ is a $z$ for which
    \begin{align*}
        f_i(z) &\leq \varepsilon, \ \forall{i \in [m]}, \\
        f_0(z)-\OPT &\leq \varepsilon.
    \end{align*}
\end{definition}

One way to obtain such an $\varepsilon$-solution to $(f)$ is by approximately minimizing the auxiliary function $\AuxFunI$:
\begin{equation}
\begin{split}\label{E:AuxFunI}
    &\FunctionsBound := \max_{i \in \{0,1,\ldots,m\}}\sup_{z \in B}f_i(z),\\
    &\AuxConst := \frac{\varepsilon}{3\FunctionsBound},\\
    &\AuxFunI: B \ni z \mapsto \max\{\AuxConst(f_0(z)+\FunctionsBound),f_1(z),f_2(z),\ldots,f_m(z)\}.
\end{split}
\end{equation}
To be exact, if
\begin{equation*}
    \AuxFunI(z) - \min_{z' \in B}\AuxFunI(z') \leq \frac{\varepsilon^2}{3\FunctionsBound},
\end{equation*}
then $z$ is an $\varepsilon$-solution to $(f)$. Hence, if we have a method for minimizing $\AuxFunI$, then we have a method for finding an $\varepsilon$-solution.

The path-following method however takes it one step further, rather minimizing the extended space formulation $\AuxFunII$ of $\AuxFunI$:
\begin{equation}
\begin{split}
    &\AuxSpace := \{(z,q) \in B \times \R : \AuxFunI(z) \leq q \leq 2\FunctionsBound\}, \\
    &\AuxFunII: \AuxSpace \ni (z,q) \mapsto q.
\end{split}
\end{equation}
Specifically, the path-following method finds an $\varepsilon$-solution to $(f)$ by obtaining a $(z,s) \in \AuxSpace$ that gets within $\frac{\varepsilon^2}{3\FunctionsBound}$ absolute error of minimizing $\AuxFunII$, and then extracting $z$ from this. To find such a $(z,s)$, the path-following method employs a modified barrier method.

This modified barrier method is similar to the barrier method of \cref{S:Boyd} in that it makes use of barrier functions and consists of a Phase II iteratively approximating the optimal solution and a Phase I finding an appropriate starting point for Phase II. Although it like the barrier method performs steps of Newton's method against a sequence of central path programs, it differs in the choice of central path programs, in that it for every central path program only performs a single step of Newton's method before moving to the next program, and in that this step is a pure step of Newton's method without any line search, which thus runs in $\BigOh(1)(mn^2+n^3)$ time.

The main trick behind making this work is to ensure we for every central path program start in the region of quadratic convergence of Newton's method, as every step of Newton's method will then be highly efficient. Let $t^{(i)}$ denote the $i$-th value of $t$, $(f^{(i)})$ denote the corresponding $i$-th central path program, and $(z,q)^{(i)}$ denote our output for $(f^{(i)})$, hence the point obtained after a total of $i$ steps of Newton's method. What we want is then that for every $i \in \N$, $(z,q)^{(i)}$ is close enough to the optimizer of $(f^{(i+1)})$ to be in the region of quadratic convergence of Newton's method for $(f^{(i+1)})$. This is achieved by having each phase start with a $((z,q)^{(0)},t^{(0)})$ such that $(z,q)^{(0)}$ is deep inside the region of quadratic convergence of $(f^{(0)})$, and then in each step changing $t$ by a small enough value that if $(z,q)^{(i)}$ is deep inside the quadratic convergence region of $(f^{(i)})$, then $(z,q)^{(i)}$ is also inside the quadratic convergence region of $(f^{(i+1)})$. A single step of Newton's method then brings $(z,q)^{(i+1)}$ deep inside the quadratic convergence region of $(f^{(i+1)})$, ensuring the pattern repeats, hence by induction that $(z,q)^{(i)}$ is always in the quadratic convergence region of $(f^{(i+1)})$.

For reference, with
\begin{equation}\label{E:AuxFunIII}
\begin{split}
    \AuxFunIII &: \AuxSpace \ni (z,q) \mapsto\\
    &-\log\left(\frac{q}{\AuxConst}-\FunctionsBound-f_0(z)\right)\\
    &- \sum_{i\in[m]}\log(q-f_i(z))\\
    &- \log(\sigma^2-\norm{z}_2^2)\\
    &- \log(2\FunctionsBound-q) \in \R,
\end{split}
\end{equation}
the central path programs of Phase I and II are captured by \cref{M:PathFollowingPhase1CentralPath,M:PathFollowingPhase2CentralPath} respectively.

\begin{program}\label{M:PathFollowingPhase1CentralPath}
Given: $\FunctionsBound$ as in \eqref{E:AuxFunI}, $\AuxFunIII$ as in \eqref{E:AuxFunIII}, and $t > 0$.
\begin{mini*}{z,q}{t \cdot (\at{-\nabla \AuxFunIII}{(0,\frac{3}{2}\FunctionsBound)})^T (z,q-\frac{3}{2}\FunctionsBound) + \AuxFunIII(z,q)}{}{}
\end{mini*}
\end{program}

\begin{program}\label{M:PathFollowingPhase2CentralPath}
Given: $\AuxFunIII$ as in \eqref{E:AuxFunIII} and $t > 0$.
\begin{mini*}{z,q}{tq + \AuxFunIII(z,q)}{}{}
\end{mini*}
\end{program}

Without going into detail, $(0,\frac{3}{2}\FunctionsBound)$ is in $\interior{\AuxSpace}$ and is the solution to \cref{M:PathFollowingPhase1CentralPath} for $t=1$, while $\kappa_1,\kappa_2>1$ are chosen such that our argument above about staying in the quadratically convergent region holds. The path-following method thus starts in Phase I with $(z,q)^{(0)}:=(0,\frac{3}{2}\FunctionsBound)$ and $t^{(0)}:=1$, defines $t^{(1)}:=\kappa_1^{-1}$, defines $(z,q)^{(1)}$ as the result of performing one step of Newton's method from $(z,q)^{(0)}$ for \cref{M:PathFollowingPhase1CentralPath} with $t=t^{(1)}$, defines $t^{(2)}:=\kappa_1^{-2}$, defines $(z,q)^{(2)}$ as the result of performing one step of Newton's method from $(z,q)^{(1)}$ for \cref{M:PathFollowingPhase1CentralPath} with $t=t^{(2)}$, and so on. This continues until $(z,q)^{(i)},t^{(i)}$ fulfill some conditions that allow us to find a $\tilde{t}^{(1)}$ such that $(z,q)^{(i)}$ is in the region of quadratic convergence of Newton's method on the \cref{M:PathFollowingPhase2CentralPath} instance corresponding to $(f)$ and $\tilde{t}^{(1)}$. The path-following method then enters Phase II, which is executed like Phase I, but with \cref{M:PathFollowingPhase2CentralPath} instances and $\tilde{t}^{(i)}=\tilde{t}^{(1)}\kappa_2^{i-1}$.

In a total of at most
\begin{equation}
    \BigOh(1)\sqrt{m}\log{\frac{2m\FunctionsBound}{\varepsilon}}
\end{equation}
steps of Newton's method, this procedure obtains an $(z,q) \in \interior{\AuxSpace}$ such that $q \leq \frac{\varepsilon^2}{3\FunctionsBound}$, hence such that $z$ is an $\varepsilon$-solution to our \cref{M:ArbitraryQCQP} instance $(f)$. Multiplying this bound with the complexity bound of a single step of the pure Newton's method, we obtain a corresponding time complexity bound, as captured by \cref{T:PathFollowingComplexityBound}.

\begin{theorem}\label{T:PathFollowingComplexityBound}
    If $(f)$ is a \cref{M:ArbitraryQCQP} instance, then the path-following method obtains an $\varepsilon$-solution to $(f)$ in time complexity
    \begin{equation}
        \BigOh(1)\sqrt{m}(mn^2+n^3)\log{\frac{2m\FunctionsBound}{\varepsilon}}.
    \end{equation}
\end{theorem}

\subsection{For our special case}

If $(f)$ is a \cref{M:QCQP} instance, then restricting $\costs$ and $\outfuns$ of $(f)$ to some ball $B$ gives us an instance of \cref{M:ArbitraryQCQP}, hence allowing us to utilize the path-following method and its corresponding analysis to optimize it.

We hence aim to make such a restriction while leaving $(f)$ unchanged inside its feasible region, as then any $\varepsilon$-solution to the modified instance will be an $\varepsilon$-solution to $(f)$. Since we know that $0 \leq y \leq x \leq \caps$, we know that
\begin{equation}
    \norm{(x,y)}_2 \leq \norm{(\caps,\caps)}_2 = \sqrt{2}\norm{\caps}_2.
\end{equation}
We can thus achieve this goal by picking
\begin{equation}
    \sigma := \sqrt{2}\norm{\caps}_2
\end{equation}
and restricting our functions $\costs$ and $\outfuns$ to
\begin{equation}
\begin{array}{r@{}l}
    B :&= \{z \in \R^A \times \R^A: \norm{z}_2 \leq \sigma\}\\
    &= \{z \in \R^A \times \R^A: \norm{z}_2 \leq \sqrt{2}\norm{\caps}_2\}.
\end{array}
\end{equation}

As we are hence dealing with a trivial reduction from \cref{M:QCQP} to \cref{M:ArbitraryQCQP}, we will from now treat \cref{M:QCQP} as a subclass of \cref{M:ArbitraryQCQP}, although this is technically not the case.

By analyzing the objective function and every set of inequalities of \cref{M:QCQP}, we find that we can pick
\begin{equation}\label{E:FuncsBoundBound}
\begin{array}{r@{}l}
    \FunctionsBound :&= \max\{\sqrt{2\abs{A}}\clin_{max}\sigma + \cquad_{max}\sigma^2, 2\sigma + r_{max}\sigma^2, \sigma + \caps_{max}, \sqrt{2\abs{A}}\sigma + \demands_{max}\}\\
    &= \BigOh\left(\clin_{max}\sqrt{\abs{A}}\norm{\caps}_2 + (\cquad_{max} + r_{max})\norm{\caps}_2^2\right).
\end{array}
\end{equation}
This complexity bound assumes the instance is feasible, as then
\begin{equation*}
    \demands_{max} \leq \norm{\caps}_1 \leq \sqrt{2\abs{A}}\norm{\caps}_2.
\end{equation*}

Finally, by weak graph connectivity, $\abs{V}=\BigOh(\abs{A})$. Combining this with \eqref{E:FuncsBoundBound} and \cref{T:PathFollowingComplexityBound}, we arrive at an explicit complexity bound for approximately solving \cref{M:QCQP} instances, captured by \cref{C:PathFollowingComplexityBound}.

\begin{corollary}\label{C:PathFollowingComplexityBound}
    If $(f)$ is a \cref{M:QCQP} instance, then the path-following method obtains an $\varepsilon$-solution to $(f)$ in time complexity
    \begin{equation}
        \BigOh\left(\abs{A}^{7/2}\log{\frac{\abs{A}(\clin_{max} + \cquad_{max} + r_{max})\norm{\caps}_2}{\varepsilon}}\right).
    \end{equation}
\end{corollary}

\subsection{Finding a feasible almost-optimal point}\label{SS:FeasibleEpsSolution}

We now have a method that for any feasible \cref{M:ArbitraryQCQP} instance $(f)$ obtains an $\varepsilon$-solution in polynomial time. However, our goal is not just to obtain $\varepsilon$-solutions, but rather to obtain feasible $\varepsilon$-solutions.

\begin{definition}[Feasible $\varepsilon$-solution]
    If $(f)$ is a \cref{M:ArbitraryQCQP} instance with functions $f_0,\ldots,f_m$, then a feasible $\varepsilon$-solution to $(f)$ is a $z$ for which
    \begin{align*}
        f_i(z) &\leq 0, \ \forall{i \in [m]}, \\
        f_0(z)-\OPT &\leq \varepsilon.
    \end{align*}
\end{definition}

So assume that $(f)$ is a \cref{M:ArbitraryQCQP} instance and $\varepsilon > 0$. Then we will find a feasible $\varepsilon$-solution to $(f)$ by making an appropriate choice of $\varepsilon'$ and finding an $\varepsilon'$-solution to to the $\varepsilon'$-hardened program instance $(f^{\varepsilon'})$ corresponding to $(f)$ and $\varepsilon'$.

\begin{program}[$\varepsilon$-hardened program]
Given: $f_0,\ldots,f_m$ as in \cref{M:ArbitraryQCQP} and $\varepsilon > 0$.
\begin{mini*}{z}{f_0(z)}{}{}
    \addConstraint{f_i(z) + \varepsilon}{\leq 0, \quad}{i \in [m]}.
\end{mini*}
\end{program}

By definition, we see that a $z \in B$ is an $\varepsilon'$-solution to $(f^{\varepsilon'})$ iff
\begin{align}
\begin{split}
    f_i(z) &\leq 0, \ \forall{i \in [m]}, \\
    f_0(z)-\OPT[(f^{\varepsilon'})] &\leq \varepsilon.
\end{split}
\end{align}
Hence, if $z$ is an $\varepsilon'$-solution to $(f^{\varepsilon'})$, then it is a feasible point of $(f)$. The problem is that we then only know that
\begin{equation}
    f_0(z)-\OPT[(f)] \leq \left(\OPT[(f^{\varepsilon'})] - \OPT[(f)]\right) + \varepsilon',
\end{equation}
which is of little use without also knowing
\begin{equation*}
    \OPT[(f^{\varepsilon'})] - \OPT[(f)].
\end{equation*}

However, if $(f)$ is a \cref{M:QCQP} instance and $(f^{\varepsilon'})$ is feasible, then we will see in \cref{S:Flows} that
\begin{equation}
    \OPT[(f^{\varepsilon'})] - \OPT \in [0, \varepsilon' \cdot (\clin_{max} + 2\cquad_{max}\caps_{max}) (\abs{V} + 3\abs{A}) \prod_{a \in A}{(1 - 2r_a u_a)^{-1}}].
\end{equation}
Thus, if
\begin{equation}\label{E:EpspCondition}
    \varepsilon' \cdot (\clin_{max} + 2\cquad_{max}\caps_{max}) (\abs{V} + 3\abs{A}) \prod_{a \in A}{(1 - 2r_a u_a)^{-1}} + \varepsilon' \leq \varepsilon,
\end{equation}
and $(f^{\varepsilon'})$ is feasible, then an $\varepsilon'$-solution to $(f^{\varepsilon'})$ is a feasible $\varepsilon$-solution to $(f)$. Defining
\begin{equation}
    \InitialGuarantor_{\varepsilon} := \frac{\varepsilon}{1 + (\clin_{max} + 2\cquad_{max}\caps_{max}) (\abs{V} + 3\abs{A}) \prod_{a \in A}{(1 - 2r_a u_a)^{-1}}},
\end{equation}
we find that \eqref{E:EpspCondition} is equivalent to
\begin{equation*}
    \varepsilon' \leq \InitialGuarantor_{\varepsilon}.
\end{equation*}

By definition, if $(f)$ is strictly feasible, then there must exist an $\varepsilon'>0$ such that $(f^{\varepsilon'})$ is feasible. Thus, if $(f)$ is a strictly feasible instance of \cref{M:QCQP}, then a feasible $\varepsilon$-solution to $(f)$ can be obtained by finding an
\begin{equation*}
    \varepsilon' \in (0, \InitialGuarantor_{\varepsilon}]
\end{equation*}
such that $(f^{\varepsilon'})$ is feasible, and then obtaining an $\varepsilon'$-solution to $(f^{\varepsilon'})$ using the path-following method.

\Cref{C:PathFollowingComplexityBound} already tells us the complexity of obtaining an $\varepsilon'$-solution to $(f^{\varepsilon'})$ for a known $\varepsilon'$. The remaining question is then how we can obtain an appropriate $\varepsilon'$, what complexity that procedure has, and what, if any, lower bound this $\varepsilon'$ has.

Our procedure is inspired by the Phase I program from \cref{S:Boyd}, as it tells us how much we can tighten all constraints at once without losing feasibility. However, we also add support for hardening, as its utility will become apparent soon.

\begin{program}\label{M:Slack}
Given: $f_0,f_1,\ldots,f_m$ corresponding to the constraints of \cref{M:QCQP} and $\varepsilon \geq 0$.
\begin{mini*}{z}{s}{}{}
    \addConstraint{f_i(z) + \varepsilon}{\leq s, \quad}{i \in [m]},
\end{mini*}
or equivalently
\begin{mini*}{z}{s + \varepsilon}{}{}
    \addConstraint{f_i(z)}{\leq s, \quad}{i \in [m]}.
\end{mini*}
\end{program}

Assume that $(f)$ is a strictly feasible \cref{M:QCQP} instance and denote by $(h^{\varepsilon'})$ the corresponding $\varepsilon'$-hardened \cref{M:Slack} instance and by $-b^{(\varepsilon')}$ the optimal value of $(h^{\varepsilon'})$. Then,
\begin{align}\label{E:bFeasibleRelationships}
\begin{split}
    b^{(\varepsilon')} < 0 &\iff (f^{\varepsilon'}) \text{ infeasible}, \\
    b^{(\varepsilon')} \geq 0 &\iff (f^{\varepsilon'}) \text{ feasible}, \\
    b^{(\varepsilon')} > 0 &\iff (f^{\varepsilon'}) \text{ strictly feasible},
\end{split}
\end{align}
and
\begin{equation}\label{E:bRelations}
    b^{(\varepsilon')} = b^{(0)} - \varepsilon'.
\end{equation}
Assume that $\varepsilon' \leq \frac{b^{(0)}}{3}$ and $(z,s)$ is an $\varepsilon'$-solution to $(h^{2\varepsilon'})$. Then,
\begin{equation}
    s \leq -b^{(2\varepsilon')} + \varepsilon' = -(b^{(0)} - 2\varepsilon') + \varepsilon' = 3 \cdot \left(\varepsilon' - \frac{b^{(0)}}{3}\right) \leq 0,
\end{equation}
and
\begin{equation}
    \forall{i \in [m]}: f_i(z) + 2\varepsilon' \leq s + \varepsilon'.
\end{equation}
Hence,
\begin{equation}
    \forall{i \in [m]}: f_i(z) + \varepsilon' \leq s \leq 0,
\end{equation}
and so $(z,s)$ is $(h^{\varepsilon'})$-feasible. As $s \leq 0$, $(z,s)$ is a certificate that
\begin{equation}
    b^{(\varepsilon')} \geq 0.
\end{equation}
Thus, by \eqref{E:bFeasibleRelationships}, $(z,s)$ is a certificate that $(f^{\varepsilon'})$ is feasible.

As every $(h^{2\varepsilon'})$ allows $s$ to get arbitrarily large, we know that $(h^{2\varepsilon'})$ is feasible for all $\varepsilon'$, so we can always use the path-following method to obtain an $\varepsilon'$-solution to $(h^{2\varepsilon'})$. We can thus efficiently find an appropriate $\varepsilon'$ by alternating between halving $\varepsilon'$ and between finding an $\varepsilon'$-solution to $(h^{2\varepsilon'})$. Starting this procedure at
\begin{equation}
    \varepsilon' := \min\{\InitialGuarantor_{\varepsilon}, 1\}
\end{equation}
then ensures that we at most have to run the path-following method
\begin{equation}
    \max(0,\ceil{\log_2{\frac{3}{b^{(0)}}}}) = \BigOh(\log{\frac{1}{b^{(0)}}})
\end{equation}
times. As a side note, the $b^{(0)}$ here equals the $\abs{\OPT[(g)]}$ of \cref{T:BarrierNewtonBound} in \cref{S:Boyd}, or more specifically $-\OPT[(g)]$.

\Cref{A:FindEps} captures the whole procedure we have here discussed, hence for any $\varepsilon > 0$ and strictly feasible \cref{M:QCQP} instance $(f)$ finding an
\begin{equation}\label{E:EpspLB}
    \varepsilon' \geq \min\{\frac{b^{(0)}}{6}, \InitialGuarantor_{\varepsilon}\} > 0
\end{equation}
such that all $\varepsilon'$-solutions to the $\varepsilon'$-hardened $(f^{\varepsilon'})$ are feasible $\varepsilon$-solutions to $(f)$.

\begin{algorithm}[htbp]
\caption{Find appropriate $\varepsilon'$}\label{A:FindEps}
\KwData{$\varepsilon > 0$ and a strictly feasible \cref{M:QCQP} instance $(f)$.}
\KwResult{$\varepsilon' > 0$}
    \begin{outline}[enumerate]
        \1 $\varepsilon' \longleftarrow \min\{\InitialGuarantor_{\varepsilon}, 1\}$
        \1 While True:
            \2 $(h^{2\varepsilon'}) \longleftarrow$ $\varepsilon'$-hardened \cref{M:Slack} instance corresponding to $(f)$
            \2 $(z,s) \longleftarrow $ path-following method on $(h^{2\varepsilon'})$ with precision $\varepsilon'$
            \2 If $s \leq 0$: return $\varepsilon'$
            \2 Else: $\varepsilon' \longleftarrow \varepsilon'/2$
    \end{outline}
\end{algorithm}

Since the complexity bound for the path-following method in \cref{C:PathFollowingComplexityBound} increases as $\varepsilon'$ grows smaller, we can combine \eqref{E:EpspLB} and \cref{C:PathFollowingComplexityBound} to obtain a complexity bound on \cref{A:FindEps}, as captured by \cref{L:FindEpsComplexity}.

\begin{lemma}\label{L:FindEpsComplexity}
    \Cref{A:FindEps} returns an appropriate $\varepsilon'$ in time complexity

    \begin{equation}
        \BigOh\left(\abs{A}^{7/2}\log\left(\frac{\abs{A}(\clin_{max} + \cquad_{max} + r_{max})\norm{\caps}_2}{\min\{b^{(0)}, \InitialGuarantor_{\varepsilon}\}}\right)\log\frac{1}{b^{(0)}}\right)
    \end{equation}
\end{lemma}

Once an appropriate $\varepsilon'$ has been obtained, a feasible $\varepsilon$-solution to $(f)$ is obtained by running the path-following method once on $(f^{2\varepsilon'})$ with $\varepsilon'$ precision. The full procedure for finding an $\varepsilon$-optimal solution to a strictly feasible \cref{M:QCQP} instance is shown in \cref{A:FeasibleEpsSolution}.

\begin{algorithm}[htbp]
\caption{Compute a feasible $\varepsilon$-solution}\label{A:FeasibleEpsSolution}
\KwData{$\varepsilon > 0$ and a strictly feasible \cref{M:QCQP} instance $(f)$.}
\KwResult{Feasible $\varepsilon$-solution $x,y$ of $(f)$.}
    \begin{enumerate}
        \item $\varepsilon' \longleftarrow$ \cref{A:FindEps} for $(f)$ and $\varepsilon$
        \item $(f^{2\varepsilon'}) \longleftarrow$ $2\varepsilon'$-hardened $(f)$
        \item \label{I:FinalPathFollow} $x,y \longleftarrow$ path-following method on $(f^{2\varepsilon'})$ with $\varepsilon'$ precision
    \end{enumerate}
\end{algorithm}

This extra run of the path-following method in Step \ref{I:FinalPathFollow} obeys the same complexity bound of \cref{C:PathFollowingComplexityBound} as was used in \cref{L:FindEpsComplexity}. Hence, \cref{A:FindEps,A:FeasibleEpsSolution} have the same complexity, bounded by \cref{L:FindEpsComplexity}.

\begin{corollary}\label{C:FeasibleEpsComplexity}
    If $\varepsilon > 0$ and $(f)$ is a strictly feasible \cref{M:QCQP} instance, then \cref{A:FeasibleEpsSolution} returns a feasible $\varepsilon$-solution to $(f)$ in time complexity
    \begin{equation}
        \BigOh\left(\abs{A}^{7/2}\log\left(\frac{\abs{A}(\clin_{max} + \cquad_{max} + r_{max})\norm{\caps}_2}{\min\{b^{(0)}, \InitialGuarantor_{\varepsilon}\}}\right)\log\frac{1}{b^{(0)}}\right).
    \end{equation}
\end{corollary}

\subsection{From beating absolute error in polynomial-time to FPTAS}

Assume $(f)$ is a feasible \cref{M:QCQP} instance. Then $\OPT > 0$, and so
\begin{equation}
    \val{z}-\OPT \leq \varepsilon \iff \val{z} \leq \OPT+\varepsilon \iff \frac{\val{z}}{\OPT} \leq \frac{\OPT+\varepsilon}{\OPT} = 1 + \frac{\varepsilon}{\OPT}.
\end{equation}
Thus, if $\OPT \geq 1$, then
\begin{equation}
    \val{z}-\OPT \leq \varepsilon \implies \frac{\val{z}}{\OPT} \leq 1+\varepsilon,
\end{equation}
while if $\OPT \in (0,1)$, then
\begin{equation}
    \val{z}-\OPT \leq \varepsilon \OPT \iff \frac{\val{z}}{\OPT} \leq 1+\varepsilon.
\end{equation}

Hence, if we want a complexity bound for beating the relative error $\varepsilon$, we just have to replace $\varepsilon$ by $\varepsilon \OPT$ in \cref{C:FeasibleEpsComplexity}. Since \cref{A:FeasibleEpsSolution} takes the desired absolute precision as input, we need to replace $\OPT$ with something we know up front. This is however unproblematic, as
\begin{equation}
    \OPT \geq F := \min_{a \in \vout{\ssource}}\clin_{a} \sum_{\sink \in \sinks}\demands_\sink.
\end{equation}

\begin{algorithm}[htbp]
\caption{Beat \cref{M:QCQP} instance to within relative error $\varepsilon$}\label{A:ProgramFPTAS}
\KwData{$\varepsilon > 0$ and a strictly feasible \cref{M:QCQP} instance $(f)$.}
\KwResult{$(f)$-feasible point $x,y$ within $\varepsilon$ relative error of $\OPT$.}
    $x,y \longleftarrow$ \cref{A:FeasibleEpsSolution} on $(f)$ and $\varepsilon F$.
\end{algorithm}

\begin{corollary}\label{C:RelativeEpsComplexity}
    \Cref{A:ProgramFPTAS} returns a feasible point within $\varepsilon$ relative error of $\OPT$ in time complexity
    \begin{equation}
        \BigOh\left(\abs{A}^{7/2}\log\left(\frac{\abs{A}(\clin_{max} + \cquad_{max} + r_{max})\norm{\caps}_2}{\min\{b^{(0)}, F\InitialGuarantor_{\varepsilon}\}}\right)\log\frac{1}{b^{(0)}}\right)
    \end{equation}
\end{corollary}

\subsection{Final algorithm}

Combining the reductions of \cref{S:Reducing} with the optimizers of the current section, we can finally combine everything into \cref{A:FPTAS}.

\begin{algorithm}[htbp]
\caption{Beat \cref{P:ConstantOP} instance to within relative error $\varepsilon$}\label{A:FPTAS}
\KwData{\cref{P:ConstantOP} instance $(f)$ and target relative error $\varepsilon$}
\KwResult{$(f)$-feasible point $\otild{z}$ within $\varepsilon$ relative error of $\OPT$.}
    \begin{outline}[enumerate]
        \1 $(g) \longleftarrow$ \cref{A:ConstantOPToQCQP} on $(f)$
        \1 \label{I:SolveQCQP} $z \longleftarrow$ \cref{A:ProgramFPTAS} on $\varepsilon$ and $(g)$
        \1 $\otild{z} \longleftarrow$ \cref{A:QCQPToConstantOP} on $(g)$ and $z$
    \end{outline}
\end{algorithm}

\begin{proposition}\label{Pr:FinalAlgIsPolynomial}
    If $(f)$ is a strictly feasible \cref{P:ConstantOP} instance and $\varepsilon > 0$, then \cref{A:FPTAS} returns a feasible point within relative error $\varepsilon$ of $\OPT$ in time polynomial in
    \begin{equation*}
        \bits{(f)}+\log{\frac{1}{\varepsilon}}+\log{\frac{1}{b^{(0)}}}.
    \end{equation*}
\end{proposition}
\begin{proof}
    This follows immediately from \cref{Pr:ConstantOPReducesToQCQP} and \cref{C:RelativeEpsComplexity}.
\end{proof}

\begin{theorem}\label{T:FPTAS}
    If M is a set of strictly feasible instances of \cref{P:ConstantOP} on which $\log{\frac{1}{b^{(0)}}}$ is polynomially bounded in $\bits{in}$, then \cref{A:FPTAS} is an \ac{FPTAS} on M.
\end{theorem}
\begin{proof}
    This follows immediately from \cref{Pr:FinalAlgIsPolynomial}.
\end{proof}

    \section{Linear losses}\label{S:Linear} Another special case of \cref{P:GeneralizedConstant} is that of piecewise constant $\prodcpus$ paired with linear losses:
\begin{equation}
    \outfuns_a:
    \begin{cases}
        [0,\caps_a] \to R_{\geq 0}, \\
        x \mapsto x - r_a x =: \outscaler_a x,
    \end{cases}
    \quad \forall{a \in A}.
\end{equation}

If $(f)$ is such an instance, then applying \cref{A:ToStatic} to $(f)$ will produce an instance $(f')$ of \cref{M:QCQP} where all quadratic terms are 0, hence a \acf{LP} and an instance of \cref{M:LP}.

\begin{program}\label{M:LP}
Given: A weakly connected digraph $G=(V,A,\outscaler,\clin,\caps,\demands)$ such that $\forall{a \in A}: \outscaler_a \in [0,1] \a \clin_a \geq 0$.

\begin{mini*}{x,y}{\sum_{a \in A}{\clin_a x_a}}{}{}
    \addConstraint{y_a}{\leq \outscaler_a x_a, \quad}{a \in A}
    \addConstraint{x_a}{\leq \caps_a, \quad}{a \in A}
    \addConstraint{y_a}{\geq 0, \quad}{a \in A}
    \addConstraint{\vin{v}}{\geq \vout{v}, \quad}{v \in V \setminus (\{\ssource\} \cup \sinks)}
    \addConstraint{\vin{\sink}}{\geq \demands_\sink, \quad}{\sink \in \sinks}
\end{mini*}
\end{program}

First, as $(f')$ is an \acs{LP}, we can in the case of rational coefficients employ the Ellipsoid method to in polynomial time exactly solve $(f')$ hence $(f)$ \citep{grötschel1988geometric}. Moreover, as \cref{M:LP} is a special case of \cref{M:QCQP}, we know that \cref{C:QCQPBarrierNewtonBound} holds for $(f')$, hence that \cref{T:FPTAS} holds for $(f)$. Finally, by the same logic as in \cref{L:OptIsFeasible}, optimally solving this \ac{LP} corresponds to finding a min-cost generalized flow with linear losses, for which \cite{wayne2002polynomial} provides an exact polynomial-time combinatorial algorithm and a corresponding combinatorial \ac{FPTAS}.

\chapter{Sensitivity analysis}\label{Ch:Sensitivity}
    As models simplify reality, it is important that a good solution to one model instance is also a good solution to similar model instances. Otherwise, optimizing on the model is not of much use, as any obtained solutions may anyway fare poorly in practice. As a proxy to this objective, \textit{sensitivity analysis} looks at how the optimal value of a program changes as instance parameters change.

We will in this chapter first define this sensitivity notion rigorously. Then, we will review some general \ac{CP} theory for local and global sensitivity analysis. Finally, we will make use of the flow properties of our optimal solutions to derive more explicit sensitivity bounds.
    \section{For convex programs}\label{S:Convex} \begin{program}[General \acs{CP}]\label{M:GeneralCP}
Given: Convex $f_0,g_1,g_2,\ldots,g_m: \R^n \to \R$, affine $h: \R^n \to \R^l$, and $g: x \mapsto (g_1,\ldots,g_m)$.
\begin{mini*}{z}{f_0(z)}{}{}
    \addConstraint{g(z)}{\leq 0}
    \addConstraint{h(z)}{= 0}.
\end{mini*}
\end{program}

Assume we are given some \cref{M:GeneralCP} instance $(f)$, and denote by $(f_{(\IneqPerturb,\EqPerturb)})$ the corresponding $(\IneqPerturb,\EqPerturb)$-perturbed program instance.

\begin{program}[$(\IneqPerturb,\EqPerturb)$-perturbed program]
Given: $f_0,g,h$ as in \cref{M:GeneralCP}, $\IneqPerturb \in \R^m$ and $\EqPerturb \in \R^l$.
\begin{mini*}{z}{f_0(z)}{}{}
    \addConstraint{g(z)}{\leq \IneqPerturb}
    \addConstraint{h(z)}{= \EqPerturb}.
\end{mini*}
\end{program}

For ease of notation, we define
\begin{equation}
\begin{split}
    &\FunOPT: \R^m \times \R^l \ni (\IneqPerturb,\EqPerturb) \mapsto \OPT[(f_{(\IneqPerturb,\EqPerturb)})] \in \extR, \\
    &\DeltaOPT{(\IneqPerturb,\EqPerturb)} := \FunOPT(\IneqPerturb,\EqPerturb) - \FunOPT(0,0) = \OPT[(f_{(\IneqPerturb,\EqPerturb)})] - \OPT[(f)].
\end{split}
\end{equation}

\subsection{Global sensitivity}

The following subsection is entirely based on \cite{boyd2004convex}.

\begin{definition}[Lagrangian and dual function]\label{D:Lagrangian}
    Assume $f_0$, $g$ and $h$ correspond to some \cref{M:GeneralCP} instance $(f)$. The corresponding \textit{Lagrangian} is then
    \begin{equation*}
        \hat{L}: \R^n \times \R^m \times \R^l \ni (z,\lambda,\nu) \mapsto f_0(z) + \lambda^T g(z) + \nu^T h(z) \in \R,
    \end{equation*}
    while the corresponding (Lagrangian) dual function is
    \begin{equation*}
        L: \R^m \times \R^l \ni (\lambda,\nu) \mapsto \min_{z \in \R^n}\hat{L}(z,\lambda,\nu) \in \extR.
    \end{equation*}
\end{definition}

To analyze the sensitivity of some \cref{M:GeneralCP} instance $(f)$, we will heavily rely on its corresponding dual program instance $(L)$.

\begin{program}[Dual program]
    Given: $L$ as in \cref{D:Lagrangian}.
    \begin{maxi*}{\lambda,\nu}{L(\lambda,\nu)}{}{}
        \addConstraint{\lambda}{\geq 0}
    \end{maxi*}    
\end{program}

Now suppose $(f)$ is a strictly feasible \cref{M:GeneralCP} instance, $(L)$ is its dual, and $\lambda_*,\nu_*$ optimize $(L)$. By Slater's condition, strong duality must then hold, so
\begin{equation}
    L(\lambda_*,\nu_*)=\FunOPT(0,0).
\end{equation}

Assume moreover that $\IneqPerturb$ and $\EqPerturb$ are fixed and $z$ is a $(f_{(\IneqPerturb,\EqPerturb)})$-feasible point, hence that
\begin{align*}
    g(z) &\leq \IneqPerturb, \\
    h(z) &= \EqPerturb.
\end{align*}

Then,
\begin{align*}
    &\FunOPT(0,0) \\
    &=L(\lambda_*,\nu_*) \\
    &\leq \hat{L}(z,\lambda_*,\nu_*) \\
    &= f_0(z) + \lambda_*^T g(z) + \nu_*^T h(z) \\
    &= f_0(z) + \lambda_*^T g(z) + \nu_*^T \EqPerturb \\
    &\leq f_0(z) + \lambda_*^T \IneqPerturb + \nu_*^T \EqPerturb.
\end{align*}
As $z$ was an arbitrary $(f_{(\IneqPerturb,\EqPerturb)})$-feasible point, we conclude that
\begin{equation}
    \FunOPT(0,0) \leq \FunOPT(\IneqPerturb,\EqPerturb) + \lambda_*^T \IneqPerturb + \nu_*^T \EqPerturb.
\end{equation}

\begin{lemma}\label{L:GlobalSensitivityLB}
    If $(f)$ is a strictly feasible \cref{M:GeneralCP} instance, $(L)$ is its dual, $\lambda,\nu$ optimize $(L)$, and $\FunOPT$ corresponds to $(f)$, then
    \begin{equation*}
        \forall{(\IneqPerturb,\EqPerturb) \in \R^m \times \R^l}: \Delta_{(\IneqPerturb,\EqPerturb)} \geq - \lambda^T \IneqPerturb - \nu^T \EqPerturb.
    \end{equation*}
\end{lemma}

\subsection{Local sensitivity}

The following subsection is entirely based on \cite{boyd2004convex}.

Assume $(f)$ is a strictly feasible \cref{M:GeneralCP} instance, $(L)$ is its corresponding dual, and $\lambda,\nu$ optimize $(L)$. We will denote by $e_i$ the $i$-th unit vector of $\R^{m}$:
\begin{equation}
    e_i := (\underbrace{0,...,0}_{i-1},1,\underbrace{0,...,0}_{m-i}).
\end{equation}

By \cref{L:GlobalSensitivityLB}, we then know that
\begin{equation*}
    \DeltaOPT{(te_i,0)} \geq -\lambda^T te_i = -t\lambda_{i}.
\end{equation*}
Hence,
\begin{equation}
    \forall{t > 0}: \frac{\Delta_{(te_i,0)}}{t} \geq -\lambda_{i}
\end{equation}
and
\begin{equation}
    \forall{t < 0}: \frac{\Delta_{(te_i,0)}}{t} \leq -\lambda_{i}.
\end{equation}

Thus, if the partial derivatives of $\FunOPT$ exist at $(0,0)$, then we conclude that
\begin{equation}
    \at{\frac{\partial \FunOPT(\IneqPerturb,\EqPerturb)}{\partial \IneqPerturb_i}}{(\IneqPerturb,\EqPerturb)=(0,0)} = -\lambda_{i}.
\end{equation}
By a similar argument, we then also find that
\begin{equation}
    \at{\frac{\partial \FunOPT(\IneqPerturb,\EqPerturb)}{\partial \EqPerturb_i}}{(\IneqPerturb,\EqPerturb)=(0,0)} = -\nu_{i}.
\end{equation}

\begin{proposition}\label{Pr:LocalSensitivity}
    If $(f)$ is a \cref{M:GeneralCP} instance, $(L)$ is its dual, $\lambda,\nu$ optimize $(L)$, $\FunOPT$ corresponds to $(f)$, and $\at{\nabla \FunOPT}{(0,0)}$ exists, then
    \begin{align*}
        \at{\frac{\partial \FunOPT(\IneqPerturb,\EqPerturb)}{\partial \IneqPerturb_i}}{(\IneqPerturb,\EqPerturb)=(0,0)} &= -\lambda_{i}, \\
        \at{\frac{\partial \FunOPT(\IneqPerturb,\EqPerturb)}{\partial \EqPerturb_i}}{(\IneqPerturb,\EqPerturb)=(0,0)} &= -\nu_{i}.
    \end{align*}
\end{proposition}

\subsection{For our parameters}

\cref{M:CP} has no equality constraints, so we only have to consider $\IneqPerturb$ when applying the above analysis to it. Assume then that $(f)$ is a strictly feasible \cref{M:CP} instance, $(L)$ is its dual, and $\lambda$ optimizes $(L)$.

In \cref{M:CP}, we solve
\begin{mini*}{x,y}{\sum_{a \in A}{c_a(x_a)}}{}{}
    \addConstraint{y_a}{\leq\outfuns_a(x_a), \quad}{a \in A}
    \addConstraint{x_a}{\leq \caps_a, \quad}{a \in A}
    \addConstraint{y_a}{\geq 0, \quad}{a \in A}
    \addConstraint{\vin{v}}{\geq \vout{v}, \quad}{v \in V \setminus (\{\ssource\} \cup \sinks)}
    \addConstraint{\vin{\sink}}{\geq \demands_\sink, \quad}{\sink \in \sinks}.
\end{mini*}
This can equivalently be formulated as
\begin{mini}{x,y}{\sum_{a \in A}{c_a(x_a)}}{\label{E:SensitivityCP}}{}
    \addConstraint{y_a - \outfuns_a(x_a)}{\leq 0, \quad}{a \in A}
    \addConstraint{x_a - \caps_a}{\leq 0, \quad}{a \in A}
    \addConstraint{-y_a}{\leq 0, \quad}{a \in A}
    \addConstraint{\vin{v} - \vout{v}}{\geq 0, \quad}{v \in V \setminus (\{\ssource\} \cup \sinks)}
    \addConstraint{\vin{\sink} - \demands_\sink}{\geq 0, \quad}{\sink \in \sinks}.
\end{mini}
We will assume that $\IneqPerturb,\lambda$ follow the structure of \eqref{E:SensitivityCP}, and use the notation
\begin{align}\label{E:CompositeDualStructure}
\begin{split}
    \IneqPerturb &= (\IneqPerturb^{(1)},\IneqPerturb^{(2)},\IneqPerturb^{(3)},\IneqPerturb^{(4)},\IneqPerturb^{(5)}), \\
    \lambda &= (\lambda^{(1)},\lambda^{(2)},\lambda^{(3)},\lambda^{(4)},\lambda^{(5)}).
\end{split}
\end{align}
The superscripts here refer to the corresponding constraint groups in \eqref{E:SensitivityCP}.

If $\at{\nabla \FunOPT}{\IneqPerturb=0}$ exists, then \cref{Pr:LocalSensitivity} we can infer the local sensitivity of $(f)$ from $\lambda$. For example, the local sensitivities of $\caps_a$ and $\demands_\sink$ in $(f)$ are then $-\lambda_a^{(2)}$ and $-\lambda_\sink^{(5)}$ respectively.

However, this assumption on $\at{\nabla \FunOPT}{\IneqPerturb=0}$ warrants a cautionary note. Namely, \cref{A:GeneralizedConstantToCP} will always output an instances of \cref{M:CP} where many of the arcs have their $\costs_a$ or $\outfuns_a$ defined only on exactly $[0,\caps_a]$. Hence, if these are not easily extendable beyond $[0,\caps_a]$, it makes little sense to talk about relaxing the constraints guaranteeing $x_a,y_a \in [0,\caps_a]$, and so $\at{\nabla \FunOPT}{\IneqPerturb=0}$ cannot entirely exist. However, the partial derivatives of all the other inequality constraints may then still exist, in which case we can obtain local sensitivities for all these other constraints.

Assume now furthermore that $(f)$ has quadratic losses in its arcs, $z^*$ solves $(f)$ optimally, and $z=z^*$. Then,
\begin{equation}
    y_a - \outfuns_a(x_a) \leq 0 \iff y_a - x_a \leq -r_a x_a^2 \leq 0 \iff y_a - x_a \leq -r_a \cdot (x_a^*)^2.
\end{equation}
Denote by $z_{(\IneqPerturb)}^*$ all optimal solutions to $(f^{(\IneqPerturb)})$, and assume moreover that
\begin{equation*}
    z_{(\IneqPerturb)}^* \to z^* \text{ as } \IneqPerturb \to 0.
\end{equation*}
The, by the univariate chain rule, the local sensitivity of $(f)$ on $r_a$ is
\begin{equation}
    \at{\frac{\partial \FunOPT(\IneqPerturb)}{\partial \IneqPerturb_a^{(1)}}}{\IneqPerturb=0}\frac{d (-r_a \cdot (x_a^*)^2)}{dr_a} = \lambda_a^{(1)} \cdot (x_a^*)^2.
\end{equation}

\begin{corollary}\label{C:OurCPSensitivity}
    If $(f)$ is a \cref{M:CP} instance, $(L)$ is its dual, $\lambda$ optimizes $(L)$, $\FunOPT$ corresponds to $(f)$, $\at{\nabla \FunOPT}{(0,0)}$ exists, and $\IneqPerturb,\lambda$ are structured as in \eqref{E:CompositeDualStructure}, then the local sensitivity of $\caps_a$ and $\demands_\sink$ at $\IneqPerturb=0$ are $-\lambda_a^{(2)}$ and $-\lambda_\sink^{(5)}$ respectively.

    If moreover
    \begin{equation*}
        z_{(\IneqPerturb)}^* \to z^* \text{ as } \IneqPerturb \to 0,
    \end{equation*}
    then the local sensitivity of $(f)$ to $r_a$ is
    \begin{equation*}
        \lambda_a^{(1)} \cdot (x_a^*)^2.
    \end{equation*}
\end{corollary}

\subsection{With barrier method}

Assume that $(f)$ is a \cref{M:GeneralCP} instance and $(L)$ is its dual. The above analysis assumes we have access to exact optimizers of $(L)$, but the barrier method will never find the exact solution to $(f)$ nor to $(L)$. However, if $z_{(t)}$ is on the central path of $(f)$, then \cite{boyd2004convex} provides a method to directly obtain $(L)$-feasible $\lambda_{(t)},\nu_{(t)}$ such that

\begin{equation}
    f_0(z_{(t)})-L(\lambda_{(t)},\nu_{(t)}) = \frac{m}{t}.
\end{equation}

Let $t^{(i)}$ denote the $i$-th $t$-value, $z^{(i)}$ denote the corresponding central point, and $(\lambda^{(i)},\nu^{(i)})$ denote the $(L)$-feasible point obtained from $z^{(i)}$. If the barrier method exactly solves every central path program it encounters, which it almost does, then we can hence obtain a sequence
\begin{equation}
    \{(\lambda^{(i)},\nu^{(i)})\}_{i\in\N}
\end{equation}
such that
\begin{equation}
    f_0(z^{(i)}) - L(\lambda^{(i)},\nu^{(i)}) \to 0 \text{ as } i \to \infty.
\end{equation}
By weak duality, we conclude that
\begin{equation}
    L(\lambda^{(i)},\nu^{(i)}) \to \OPT[(L)] \text{ as } i \to \infty.
\end{equation}
Hence, $(\lambda^{(i)},\nu^{(i)})$ must converge to some optimal $\lambda_*,\nu_*$, since $L$ is the pointwise minimum of a set of affine functions $\hat{L}(z,\cdot,\cdot)$, thus concave. As \cref{L:GlobalSensitivityLB} holds for all optimal $\lambda_*,\nu_*$, this implies that
\begin{equation}
    \Delta_{(\IneqPerturb,\EqPerturb)} + \varepsilon^{(i)} \geq - \IneqPerturb^T \lambda^{(i)} - \EqPerturb^T \nu^{(i)},
\end{equation}
with
\begin{equation}
    0 \leq \varepsilon^{(i)} \to 0 \text{ as } i \to \infty.
\end{equation}
Hence, as $i$ grows large, \cref{L:GlobalSensitivityLB} approximately holds for the $z^{(i)},\lambda^{(i)},\nu^{(i)}$ returned by the barrier method.

Moreover, if the partial derivatives of $\FunOPT$ exist at $(0,0)$, then \cref{Pr:LocalSensitivity} implies that the $(\lambda,\nu)$ optimizing the original dual program is unique, so then
\begin{align}
\begin{split}
    -(\lambda^{(i)})_j &\to \at{\frac{\partial \FunOPT(\IneqPerturb,\EqPerturb)}{\partial \IneqPerturb_j}}{(\IneqPerturb,\EqPerturb)=(0,0)} \\
    -(\nu^{(i)})_j &\to \at{\frac{\partial \FunOPT(\IneqPerturb,\EqPerturb)}{\partial \EqPerturb_j}}{(\IneqPerturb,\EqPerturb)=(0,0)} \\
     \text{ as } i &\to \infty.
\end{split}
\end{align}
Thus, if $\at{\nabla \FunOPT}{(0,0)}$ exists and $i$ is large, then \cref{Pr:LocalSensitivity} approximately holds for the $z^{(i)},\lambda^{(i)},\nu^{(i)}$ returned by the barrier method.

\begin{corollary}
    If $(f)$ is a \cref{M:GeneralCP} instance, $\{z^{(i)}\}_{i \in \N}$ is a sequence of central points, and $(\lambda^{(i)},\nu^{(i)})$ is computed from $z^{(i)}$, then
    \begin{align*}
        \Delta_{(\IneqPerturb,\EqPerturb)} + \varepsilon^{(i)} &\geq - \IneqPerturb^T \lambda^{(i)} - \EqPerturb^T \nu^{(i)}\\
        0 \leq \varepsilon^{(i)} &\to 0 \text{ as } i \to \infty.
    \end{align*}

    If moreover $\at{\nabla \FunOPT(\IneqPerturb,\EqPerturb)}{(\IneqPerturb,\EqPerturb)=(0,0)}$ exists, then
    \begin{align*}
        -(\lambda^{(i)})_j &\to \at{\frac{\partial \FunOPT(\IneqPerturb,\EqPerturb)}{\partial \IneqPerturb_j}}{(\IneqPerturb,\EqPerturb)=(0,0)} \\
        -(\nu^{(i)})_j &\to \at{\frac{\partial \FunOPT(\IneqPerturb,\EqPerturb)}{\partial \EqPerturb_j}}{(\IneqPerturb,\EqPerturb)=(0,0)} \\
         \forall{j} \text{ as } i &\to \infty.
    \end{align*}
    
\end{corollary}
    \section{For generalized flows}\label{S:Flows} As previously mentioned, if $(f)$ is a \cref{M:CP} instance with graph $G_{(f)}$ and $z$ is an optimal solution to $(f)$, then $z$ is a concave generalized $\ssource$-$\sinks$ flow on $G_{(f)}$. We can take advantage of this to deduce bounds on the optimal value increase $\DeltaOPT{\omega}$ associated with some $\omega$-perturbation of $(f)$. Unlike in \cref{S:Convex}, these bounds are functions of the instance parameters, hence do not require $(f)$ to first be solved optimally.

Note that all bounds deduced in this section are under the assumption that the perturbed program is still well-defined and feasible.

\subsection{Changing single demand}

Our entire analysis will be based on changes in demand. Hence, assume first the simplest situation: $(f)$ is a \cref{M:CP} instance with graph $G_{(f)}=(V,A)$, $z^*=(x^*,y^*)$ is an optimal solution to $(f)$, and $\IneqPerturb$ corresponds to changing the demand of a single sink by some $\delta > 0$.

\begin{definition}[Residual graph]\label{D:ResidualGraph}
    Assume $(f)$ is a \cref{M:CP} instance with graph $G_{(f)}=(V,A)$ and $z^*=(x^*,y^*)$ is an optimal solution to $(f)$. The \textit{residual graph} of $z^*,(f)$ is then defined as
    \begin{equation*}
        G_{(f)}^{z^*} = (V,\cup_{a \in A}\{a',a''\}),
    \end{equation*}
    where the $a',a''$ corresponding to an arc $a \in A$ are defined as follows:
    \begin{align*}
        a' &:= (u,v), \\
        a'' &:= (v,u), \\
        \caps_{a'} &:= \caps_a - x_a^*, \\
        \caps_{a''} &:= y_a^*, \\
        \outfuns_{a'}&: x \mapsto \outfuns_a(x_a^* + x) - y_a^*, \\
        \outfuns_{a''}&: y \mapsto x_a^* - \outfuns_a^{-1}(y_a^* - y).
    \end{align*}
\end{definition}

We see from \cref{D:ResidualGraph} that a residual graph for $(f)$ and $z^*$ can easily be generated in polynomial time. Moreover, we see that if $a \in A$, then sending $x$ flow into the corresponding $a'$ of $G_{(f)}^{z^*}$ and adding the resulting $(x,y)$ to $(x_a^*,y_a^*)$ corresponds to increasing $x_a^*$ by $x$ and enforcing
\begin{equation}\label{E:NoWaste}
    \outfuns_a(x_a^*)=\outfuns_a(y_a^*).
\end{equation}
Similarly, sending $y$ flow along the arc $a''$ then corresponds to reducing $y_a^*$ by $y$ and enforcing \eqref{E:NoWaste}. Thus, if a flow on $G_{(f)}$ is modified according to a flow on $G_{(f)}^{z^*}$ we obtain another flow on $G_{(f)}$. Moreover, any flow $z$ in $G_{(f)}$ can be rewritten as $z^*$ plus the modification corresponding to some $z'$ in the residual graph $G_{(f)}^{z^*}$. Hence, to modify $z^*$ into a flow with $\delta$ higher demand in some sink $\sink$ is equivalent to finding a flow in $G_{(f)}^{z^*}$ that sends $\delta$ flow into $\sink$.

By construction, the forward arcs $A'$ of $G_{(f)}^{z^*}$ are lossy, while its backward arcs $A''$ are gainful. The loss associated with getting $\delta$ more flow to $\sinks$ is thus upper bounded by the hypothetical situation of sending all flow in $G_{(f)}^{z^*}$ through all of the forward arcs once and none of the backward arcs. This corresponds to sending residual flow along a path containing all arcs of $A$. For the outgoing flow of such a hypothetical path to be $\delta$, its incoming flow is at most
\begin{equation}
    \delta\prod_{a \in A}{\MinGrad_a^{-1}},
\end{equation}
where
\begin{equation}
    \MinGrad_a^{-1} := \sup_{y\in[0,\outfuns_a(\caps_a)]}(\outfuns_{a}^{-1})'(y) = (\inf_{x\in[0,u_a]}\outfuns_{a}'(x))^{-1} = (\outfuns_{a}'(u_a))^{-1}.
\end{equation}

Thus, the new demand can be beaten by a modification of $z^*=(x^*,y^*)$ that increases $\vout{\ssource}^*$ by at most
\begin{equation}
    \delta\prod_{a \in A}{\MinGrad_a^{-1}}.
\end{equation}
The corresponding cost increase is then upper bounded by
\begin{equation}
    \costs'_{max}\delta\prod_{a \in A}{\MinGrad_a^{-1}},
\end{equation}
where
\begin{equation}
    \costs'_{max} := \max_{a \in A}\max_{x \in [0,\caps_a]} \frac{d\costs_a(x)}{dx} = \max_{a \in A} \left( \at{\frac{d\costs_a(x)}{dx}}{x = \caps_a} \right).
\end{equation}

Since $z^*$ is optimal and $\outfuns$ are concave and upper-bounded by $\Id$, we know that $\vout{\sources}$ must increase by at least $\delta$ for $\vin{\sinks}$ to increase by $\delta$. Hence,
\begin{equation}
    \DeltaOPT{\IneqPerturb} \geq \costs'_{min}\delta,
\end{equation}
where
\begin{equation}
    \costs'_{min} := \min_{a \in A}\min_{x \in [0,\caps_a]} \frac{d\costs_a(x)}{dx} = \min_{a \in A} \left( \at{\frac{d\costs_a(x)}{dx}}{x = 0} \right).
\end{equation}

We combine this with similar reasoning for the case where instead $\delta < 0$, arriving at \cref{Pr:DemandSensitivity}.

\begin{proposition}\label{Pr:DemandSensitivity}
    Assume that $(f)$ is a \cref{M:CP} instance, the perturbation $\IneqPerturb$ changes the demand of a single sink in $(f)$ by some $\delta \in \R$, and the $\IneqPerturb$-perturbed instance $(f_\IneqPerturb)$ is feasible. Then,
    \begin{equation*}
        \DeltaOPT{\IneqPerturb} \in
        \begin{cases}
            [\costs'_{min}\delta,\ \costs'_{max}\delta\prod_{a \in A}{\MinGrad_a^{-1}}], &\text{if } \delta \geq 0, \\
            [\costs'_{max}\delta\prod_{a \in A}{\MinGrad_a^{-1}},\ \costs'_{min}\delta], &\text{if } \delta < 0.
        \end{cases}
    \end{equation*}
\end{proposition}

\subsection{More general perturbations}\label{SS:GeneralPerturbationsSensitivity}

Imagine now instead that $\IneqPerturb$ corresponds to changing some arc capacity $\caps_a$ in $(f)$ by $\delta < 0$. In the best case, there exists an optimal solution $z^*$ of $(f)$ such that $x_a^* \leq \caps_a + \delta$. Then $z^*$ is also an optimal solution of the perturbed program $(f_\IneqPerturb)$, so $\DeltaOPT{\IneqPerturb} = 0$. In the worst case, every optimal solution of $(f)$ utilizes arc $a$ fully. However, then $(f_\IneqPerturb)$ is still no harder than a perturbation of $(f)$ that keeps $\caps_a$ constant but adds
\begin{equation}
    \outfuns_a(\caps_a) - \outfuns_a(\caps_a + \delta) \leq -\delta
\end{equation}
demand to head$(a)$.

If $\caps_a$ is instead changed by $\delta > 0$, we similarly see that in the worst case $(f_\IneqPerturb)$ has an optimal solution that sends at most $\caps_a$ flow into arc $a$, and that $(f_\IneqPerturb)$ is no easier than if some demand is reduced by
\begin{equation}
    \outfuns_a(\caps_a) - \outfuns_a(\caps_a + \delta) \leq -\delta.
\end{equation}

We combine these observations with the arguments used to derive \cref{Pr:DemandSensitivity}, arriving at \cref{Pr:CapsSensitivity}.

\begin{proposition}\label{Pr:CapsSensitivity}
    Assume that $(f)$ is a \cref{M:CP} instance, the perturbation $\IneqPerturb$ changes the capacity of a single arc in $(f)$ by some $\delta \in \R$, and the $\IneqPerturb$-perturbed instance $(f_\IneqPerturb)$ is feasible. Then,
    \begin{equation*}
        \DeltaOPT{\IneqPerturb} \in
        \begin{cases}
            [-\costs'_{max}\delta\prod_{a \in A}{\MinGrad_a^{-1}},\ 0], &\text{if } \delta \geq 0, \\
            [0,\ -\costs'_{max}\delta\prod_{a \in A}{\MinGrad_a^{-1}}], &\text{if } \delta < 0.
        \end{cases} 
    \end{equation*}
\end{proposition}

Similar arguments can be made for any perturbation affecting only a single inequality constraint of $(f)$.

If $\IneqPerturb$ is instead an arbitrary vector, hence a combination of many such "atomic" perturbations, then the corresponding upper (lower) bound is just the sum of the upper (lower) bounds for the atomic perturbations.

\begin{proposition}\label{Pr:CompositeSensitivity}
    Assume that $(f)$ is a \cref{M:CP} instance and the $\IneqPerturb$-perturbed instance $(f_\IneqPerturb)$ is feasible. Then,
    \begin{equation*}
        \DeltaOPT{\IneqPerturb} \in [\sum_{i \in [m]}\text{lower-bound}(\DeltaOPT{\IneqPerturb*e_i}), \sum_{i \in [m]}\text{upper-bound}(\DeltaOPT{\IneqPerturb*e_i})],
    \end{equation*}
    where
    \begin{equation*}
        \IneqPerturb * e_i = (\underbrace{0,...,0}_{i-1},\IneqPerturb_i,\underbrace{0,...,0}_{m-i}).
    \end{equation*}
\end{proposition}

\subsection{For our \acs{QCQP}}

In the special case where $(f)$ is an instance of \cref{M:QCQP}, we find that
\begin{equation}
\begin{array}{r@{\;}l}
    \MinGrad_a &= \at{\frac{d}{dx}\left( x - r_a x^2 \right)}{x=u_a} = 1 - 2 r_a u_a, \quad \forall{a \in A}, \\
    \costs'_{min} &= \clin_{min}, \\
    \costs'_{max} &= \max_{a \in A} \left( \clin_a + 2\cquad_a\caps_a \right) \leq \clin_{max} + 2\cquad_{max}\caps_{max}.
\end{array}
\end{equation}
Combining this with \cref{Pr:DemandSensitivity,Pr:CapsSensitivity}, we arrive at \cref{C:QCQPDemandSensitivity,C:QCQPCapsSensitivity}.

\begin{corollary}\label{C:QCQPDemandSensitivity}
    Assume that $(f)$ is a \cref{M:QCQP} instance, the perturbation $\IneqPerturb$ changes the demand of a single sink in $(f)$ by some $\delta \in \R$, and the $\IneqPerturb$-perturbed instance $(f_\IneqPerturb)$ is feasible. Then,
    \begin{equation*}
        \DeltaOPT{\IneqPerturb} \in
        \begin{cases}
            [\clin_{min}\delta,\ (\clin_{max} + 2\cquad_{max}\caps_{max}) \delta\prod_{a \in A}\left( 1 - 2 r_a u_a \right)^{-1}], &\text{if } \delta \geq 0, \\
            [(\clin_{max} + 2\cquad_{max}\caps_{max}) \delta\prod_{a \in A}\left( 1 - 2 r_a u_a \right)^{-1},\ \clin_{min}\delta], &\text{if } \delta < 0.
        \end{cases}
    \end{equation*}
\end{corollary}

\begin{corollary}\label{C:QCQPCapsSensitivity}
    Assume that $(f)$ is a \cref{M:QCQP} instance, the perturbation $\IneqPerturb$ changes the capacity of a single arc in $(f)$ by some $\delta \in \R$, and the $\IneqPerturb$-perturbed instance $(f_\IneqPerturb)$ is feasible. Then,
    \begin{equation*}
        \DeltaOPT{\IneqPerturb} \in
        \begin{cases}
            [-\left( \clin_{max} + 2\cquad_{max}\caps_{max} \right)\delta\prod_{a \in A}\left( 1 - 2 r_a u_a \right)^{-1},\ 0], &\text{if } \delta \geq 0, \\
            [0,\ -\left( \clin_{max} + 2\cquad_{max}\caps_{max} \right)\delta\prod_{a \in A}\left( 1 - 2 r_a u_a \right)^{-1}], &\text{if } \delta < 0.
        \end{cases} 
    \end{equation*}
\end{corollary}

\subsubsection{Finishing the proof in \Cref{SS:FeasibleEpsSolution}}

Assume $(f)$ is a \cref{M:QCQP} instance with graph $G_{(f)}=(V,A)$, and recall that given some $\varepsilon$, $(f^\varepsilon)$ denotes the corresponding $\varepsilon$-hardened instance. In \cref{SS:FeasibleEpsSolution}, all of our polynomial-time complexity bounds built on the assumption that if $(f^\varepsilon)$ is feasible, then
\begin{equation}\label{E:SlackPerturbationUB}
    \OPT[(f^\varepsilon)] - \OPT[(f)] \leq \varepsilon \cdot (\clin_{max} + 2\cquad_{max}\caps_{max}) (\abs{V} + 3\abs{A}) \prod_{a \in A}{(1 - 2r_a u_a)^{-1}}.
\end{equation}
We will now prove that \eqref{E:SlackPerturbationUB} indeed holds for all feasible $(f^\varepsilon)$. Roughly speaking, we will go about this by seeing that an $(f^{\varepsilon})$ is no harder than a specific perturbed $(f_\IneqPerturb)$ for which $\IneqPerturb$ represents an increase in $(\abs{V}+2\abs{A})\varepsilon$ of total demand.

We first construct an auxiliary graph $G'$ from $G$: for each arc $a = (u,w) \in A$, give $G'$ the nodes
\begin{align*}
    u,w &\in V, \\
    v_{uw},v_{uw}',v_{uw}'' &\notin V,
\end{align*}
and the arcs
\begin{align*}
    a^1 &:= (u,v_{uw}), \\
    a^2 &:= (v_{uw},v_{uw}'), \\
    a^3 &:= (v_{uw}',v_{uw}''), \\
    a^4 &:= (v_{uw}'',w),
\end{align*}
where
\begin{equation*}
\begin{array}{r@{\;}l}
    \costs_{a^1} &:= \costs_a, \\
    \costs_{a^2}, \costs_{a^3}, \costs_{a^4}&: x \mapsto 0, \\
    \outfuns_{a^2} &:= \outfuns_a, \\
    \outfuns_{a^1}, \outfuns_{a^3}, \outfuns_{a^4}&: x \mapsto 0, \\
    \caps_{a^1}, \caps_{a^2}, \caps_{a^3}, \caps_{a^4} &:= \caps_a.
\end{array}
\end{equation*}

\begin{figure}[H]
    \centering
    \begin{tikzpicture}
        \node[trans]        (u)        {$u$};
        \node[sink]      (v_{uw})     [right=2cm of u]   {$v_{uw}$};
        \node[sink]      (v_{uw}')     [right=2cm of v_{uw}]   {$v_{uw}'$};
        \node[sink]      (v_{uw}'')     [right=2cm of v_{uw}']   {$v_{uw}''$};
        \node[sink]      (w)     [right=2cm of v_{uw}'']   {$w$};
        
        \draw[->] (u) -- (v_{uw}) node[pos=.5,sloped,below] {$\caps_{a^1} = \caps_a$} node[pos=.5,sloped,above] {$\costs_{a^1} = \costs_a$};
        \draw[->] (v_{uw}) -- (v_{uw}') node[pos=.5,sloped,below] {$\caps_{a^2} = \caps_a$} node[pos=.5,sloped,above] {$\outfuns_{a^2} = r_a$};
        \draw[->] (v_{uw}') -- (v_{uw}'') node[pos=.5,sloped,below] {$\caps_{a^3} = \caps_a$};
        \draw[->] (v_{uw}'') -- (w) node[pos=.5,sloped,below] {$\caps_{a^4} = \caps_a$};
    \end{tikzpicture}
    \caption{The five nodes and four arcs of $G'$ corresponding to the arc $a=(u,v) \in A$.}
\end{figure}
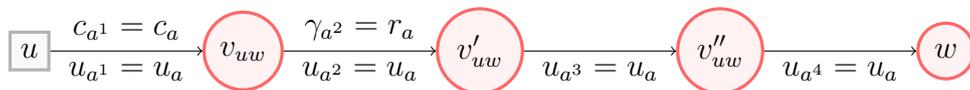

These four arcs together match the behaviour of $a$, in terms of what can be sent in from $u$ and how this maps to what comes out at $w$. Hence, the \cref{M:QCQP} instance $(f')$ corresponding to $G'$ must have the same optimal value as $(f)$. However, then $(f^{\varepsilon})$ is no harder than a perturbation of $(f')$ that adds $\varepsilon$ demand at each node of $G'$, as then
\begin{outline}
    \1 $\demands_{uw}$ captures $\displaystyle x_a + \varepsilon \leq \caps_a$,
    \1 $\demands_{uw'}$ captures $\displaystyle y_a + \varepsilon \leq x_a - r_a x_a^2$,
    \1 $\demands_{uw''}$ captures $\displaystyle y_a - \varepsilon \geq 0$,
    \1 $\demands_{w}$ captures $\displaystyle \vout{w} + \varepsilon \leq \vin{w}$.
\end{outline}

As $G'$ has $\abs{V} + 3\abs{A}$ nodes, it follows from \cref{Pr:CompositeSensitivity,C:QCQPDemandSensitivity} and the arguments used to derive these that \eqref{E:SlackPerturbationUB} holds if $(f^{\varepsilon})$ is feasible.

\chapter{Implementation}\label{Ch:Implementation} A good algorithm should ideally not just have nice theoretical bounds, but also perform well in practice. For that reason, we implemented a modification of \cref{A:FPTAS} in Python and tested its performance on various networks\footnote{The code is available at \url{https://github.com/JaHeRoth/SimplifiedDOPFSolver}}. The modification was to omit rounding from strictly feasible to feasible points, and to return vectors rather than the piecewise constant functions they represent. The graph and its modifications were implemented in NetworkX, while the convex \ac{QCQP} was formulated and solved using the gurobipy interface of Gurobi. \cref{A:Implemented} captures our implementation.

\begin{algorithm}[htbp]
\caption{Practical implementation of \cref{A:FPTAS}}\label{A:Implemented}
\KwData{NetworkX graph $G$ corresponding to an instance $(f)$ of \cref{P:ConstantOP}}
\KwResult{Almost-feasible almost-optimal $x'',y''$ of $(f)$}
    \begin{enumerate}
        \item $G' \longleftarrow$ \cref{A:SimplifyConstant} on $G$ (replace edges with arcs; add dedicated sources and sinks; split any source whose marginal cost depends on supply rate).
        \item $G'' \longleftarrow$ \cref{A:ToStatic} on $G'$ (time-expand network; split all sources at their marginal production cost function breakpoints).
        \item $x,y \longleftarrow$ Gurobi solution of the \cref{M:QCQP} instance corresponding to $G''$.
        \item $x',y' \longleftarrow$ all $x_a,y_a$ corresponding to flow at some time-step in some direction in the edges of $G$.
        \item $x'',y'' \longleftarrow$ $\forall{\text{antiparallel }a,a' \in A''}$: If $y'_{a} > x'_{a'}$: del $x'_{a'},y'_{a'}$.
    \end{enumerate}
\end{algorithm}

All execution times listed in this chapter were recorded on an ordinary HP Envy x360 laptop equipped with an 8-core AMD Ryzen 7 5800U and running in performance mode.

The first grid on which this algorithm was tested was the well-studied IEEE14 grid, as this allowed us to hopefully create a somewhat realistic instance by copying the supplies and demands from \cite{salman2022optimal}. With constant demands, hence a single time-step, our full algorithm solved this instance in 10-15ms.
\begin{figure}[htbp]
    \centering
    \includegraphics[width=0.425\linewidth]{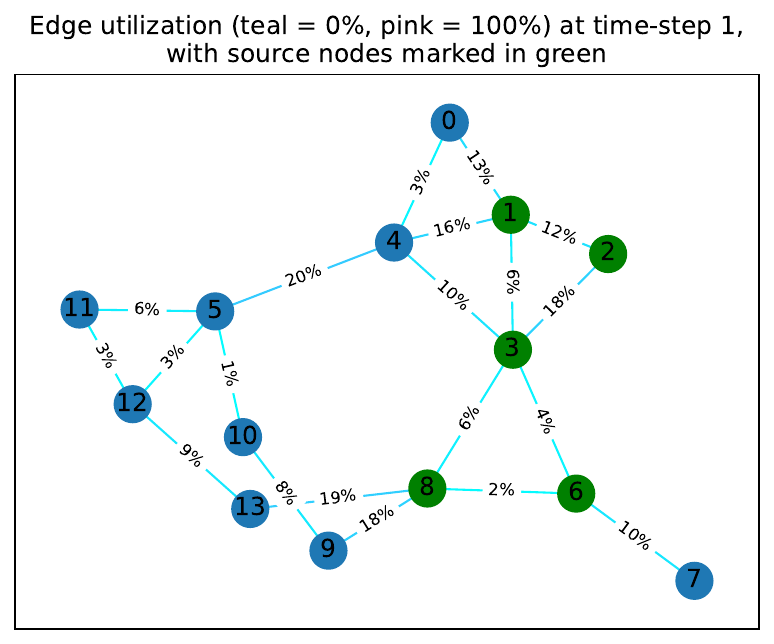}
    \caption{Optimal utilization for a reasonable \cref{P:ConstantOP} instance. Edge labels refer to how many percent of the individual edge capacities are utilized.}
\end{figure}

\Cref{T:FPTAS} does not just assume strict feasibility, but also that $\log{\frac{1}{b^{(0)}}}$ is polynomial in the instance size. We hence wanted to test how the execution time changes as we approach infeasibility. Gurobi allows its solutions to be slightly infeasible, so to capture the hardest possible case we wanted an instance that Gurobi would struggle to find such a slightly infeasible solution to. Hence, rather than trying to by hand find an instance that truly was barely feasible, we empirically found an instance which Gurobi would consider to be on the edge of infeasibility. Specifically, the edge capacities were reduced and the edge resistances increased, until Gurobi started regularly, but not always, reporting infeasibility. Rather than the execution time exploding, this just caused it to double, landing at 20-30ms.
\begin{figure}[htbp]
    \centering
    \includegraphics[width=0.425\linewidth]{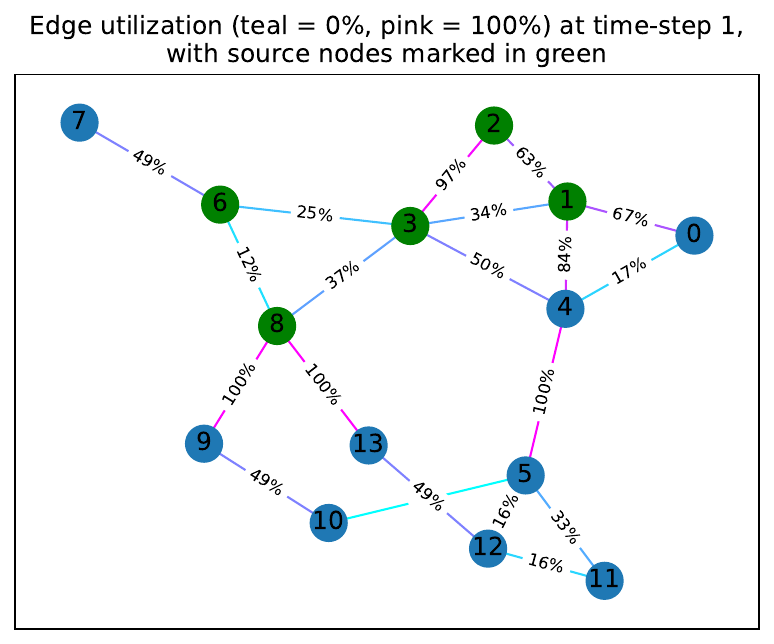}
    \caption{Optimal utilization for a \cref{P:ConstantOP} instance at the edge of Gurobi infeasibility.}
\end{figure}

Next, and possibly most importantly, we wanted to infer the relationship between execution time and graph size. The algorithm was hence benchmarked on cycle graphs, circular ladder graphs, and complete graphs of various node counts.
\begin{figure}[htbp]
    \centering
    \begin{subfigure}[t]{0.3\linewidth}
        \includegraphics[width=\linewidth]{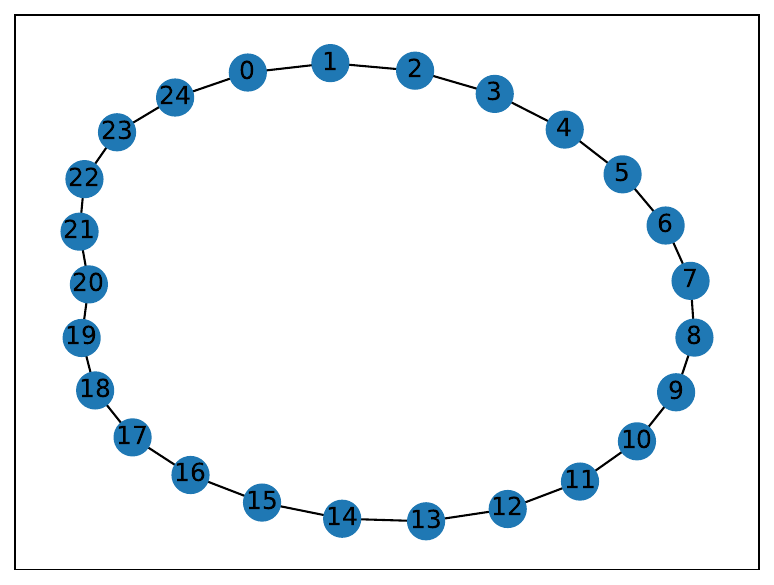}
        \caption{Cycle graph}
    \end{subfigure}
    \begin{subfigure}[t]{0.3\linewidth}
        \includegraphics[width=\linewidth]{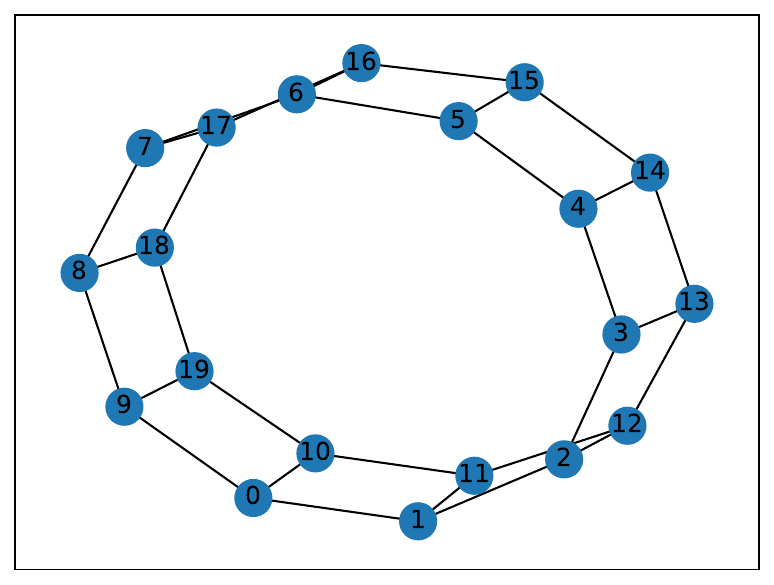}
        \caption{Circular ladder graph}
    \end{subfigure}
    \begin{subfigure}[t]{0.3\linewidth}
        \includegraphics[width=\linewidth]{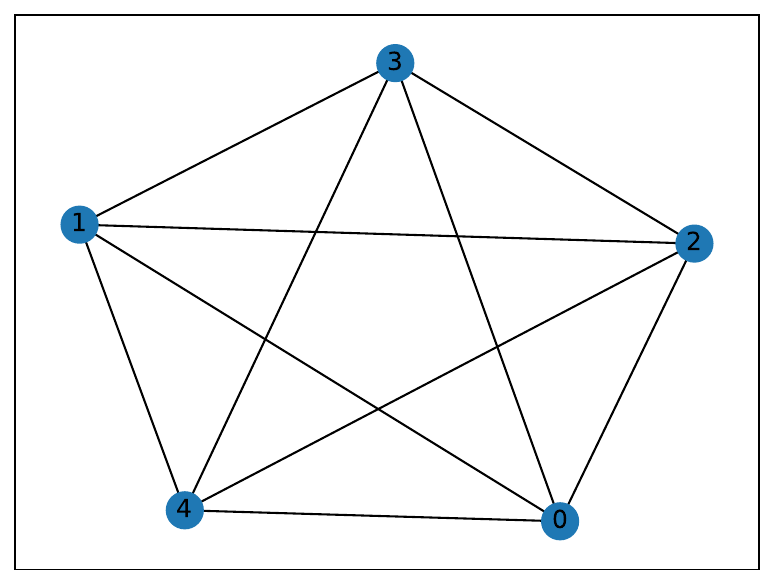}
        \caption{Complete graph}
    \end{subfigure}
\end{figure}

In each benchmark, $N$ equidistant points were chosen in $\{1,2,\ldots,n_{max}\}$, the first being $1$ and the last being $n_{max}$. For each chosen $n$, a graph of the benchmarked type and $n$ nodes was created, and the algorithm was then run $M$ times for this graph, each time with a different random set of parameters. The median of the corresponding $M$ execution times was then recorded as $e_n$. Taking the median ensured any outliers were automatically discarded, while the randomization of the parameters helped to break unrealistic symmetries for Gurobi to take advantage of. Specifically, every node $v$ was given a random demand $\dfuns_v$ with three breakpoints, a larger, random cumulative supply $\csupplies_v$, and a random marginal production cost function $\prodcpus_v$ with a single breakpoint. Every edge $e$ was moreover given a random capacity $\caps_e$ and a random resistance $r_e$. We hence ensured that the trivial 0-flow, corresponding to every node supplying its own demand, would be feasible but suboptimal, hence ensuring the instances were feasible but non-trivial.
\begin{figure}[htbp]
    \centering
    \includegraphics[width=0.425\linewidth]{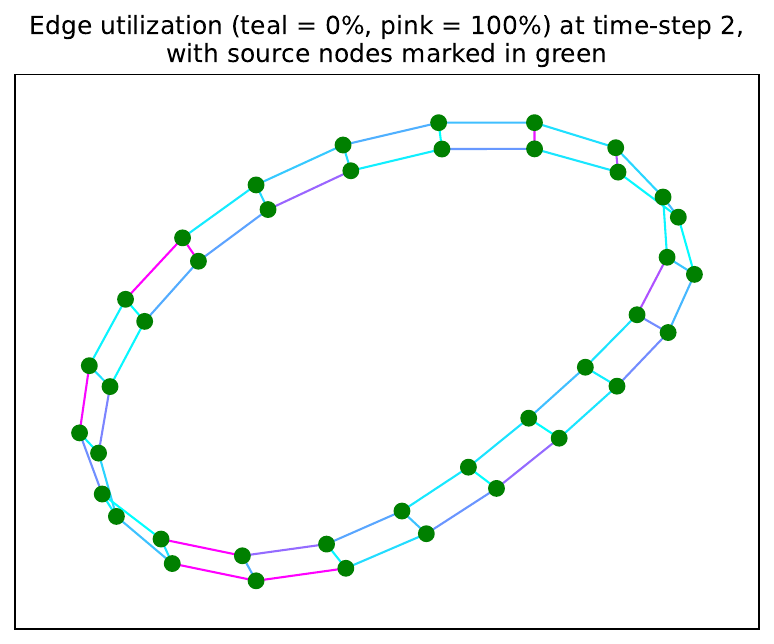}
    \caption{Optimal edge utilization for one of our randomly parameterized instance on a 200-node circular ladder graph.}
\end{figure}

Finally, to reduce the variance of the medians, any serial correlation between the $M$ measurements used to compute a median was broken by randomizing the execution order of the $N \cdot M$ total algorithm executions.

Plotting each recorded $(n,e_n)$ together with the linear and quadratic fits obtained through \ac{OLS}, we see that the execution time increases linearly or slightly quadratic in the node count for the cycle and circular ladder graphs.
\begin{figure}[htbp]
    \centering
    \begin{subfigure}[t]{0.45\linewidth}
        \includegraphics[width=\linewidth]{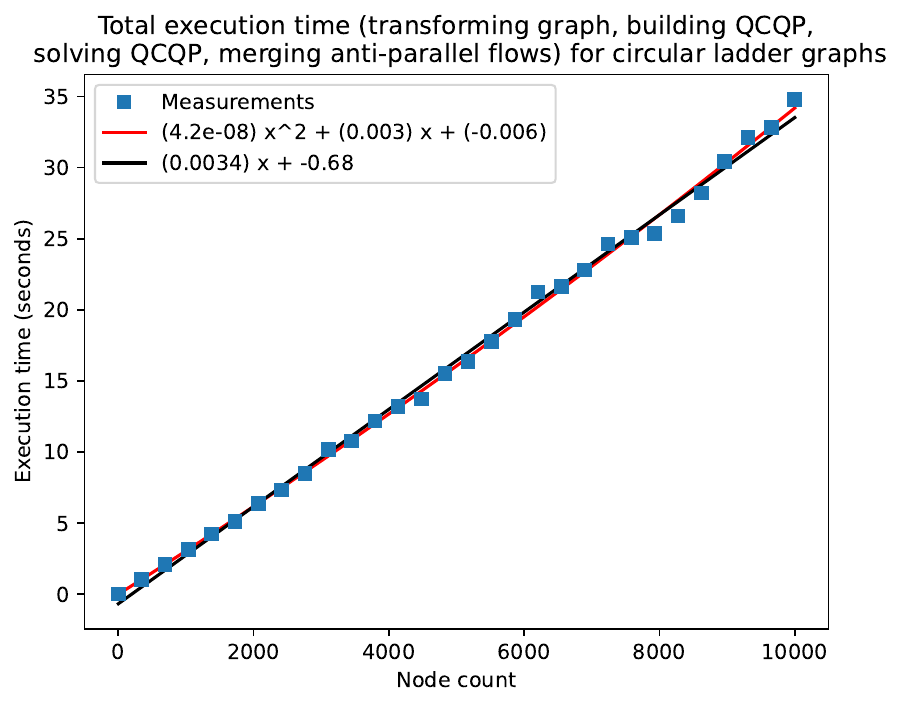}
    \end{subfigure}
    \begin{subfigure}[t]{0.45\linewidth}
        \includegraphics[width=\linewidth]{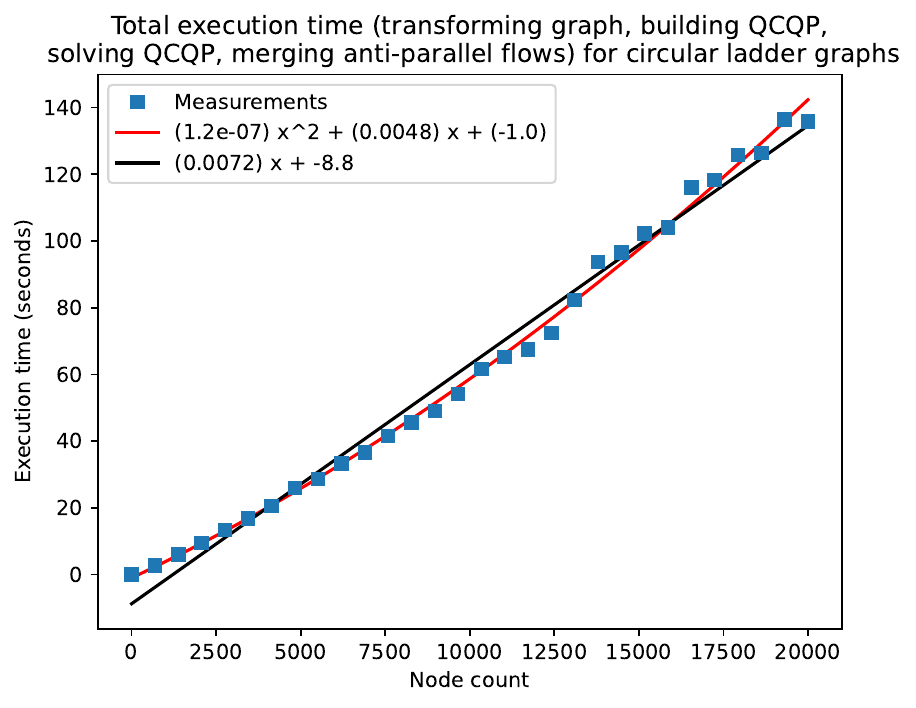}
    \end{subfigure}
    \caption{Execution time in seconds for circular ladder graphs of up to 10k and 20k nodes. In both cases $N=30 \a M=5$.}
\end{figure}
\begin{figure}[htbp]
    \centering
    \begin{subfigure}[t]{0.3\linewidth}
        \includegraphics[width=\linewidth]{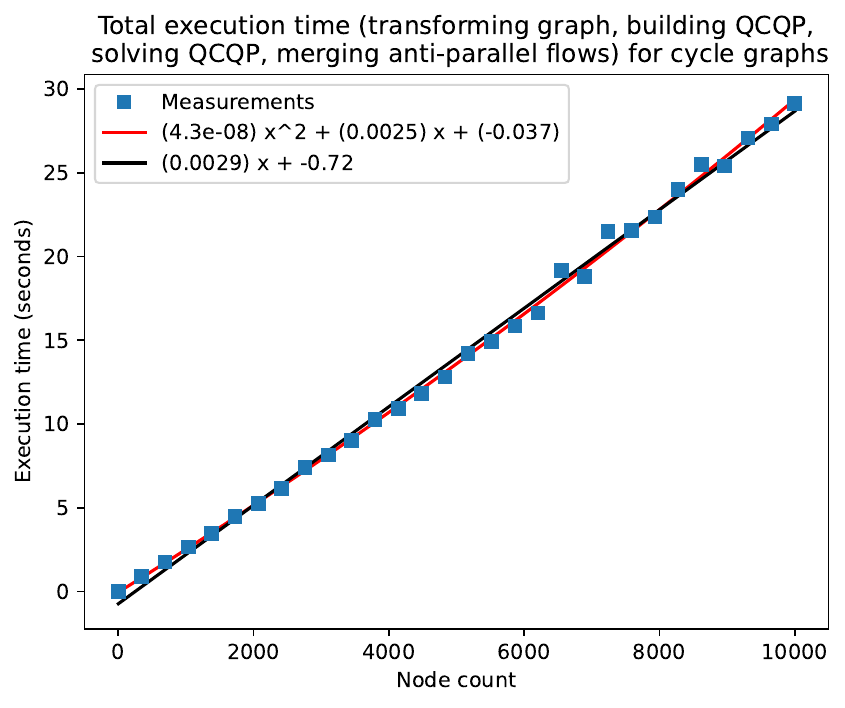}
        \caption{$N=30 \a M=5$}
    \end{subfigure}
    \begin{subfigure}[t]{0.3\linewidth}
        \includegraphics[width=\linewidth]{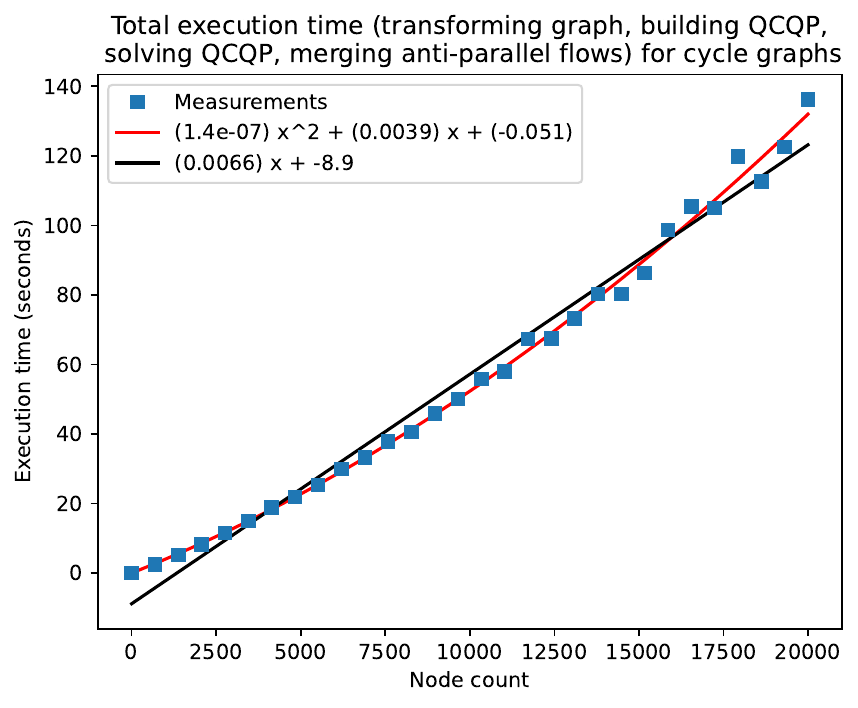}
        \caption{$N=30 \a M=5$}
    \end{subfigure}
    \begin{subfigure}[t]{0.3\linewidth}
        \includegraphics[width=\linewidth]{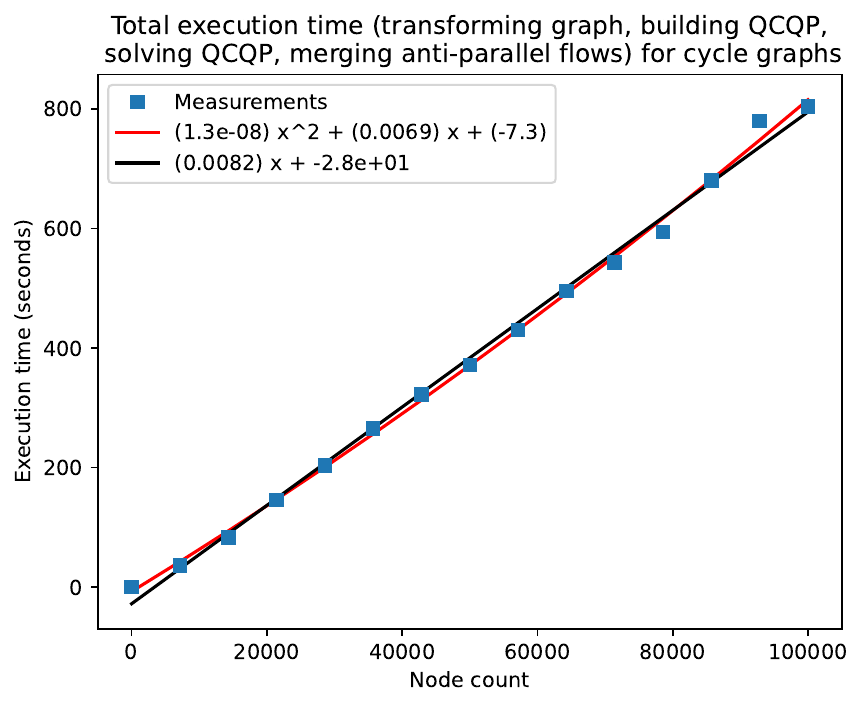}
        \caption{$N=5\a M=3$}
    \end{subfigure}
    \caption{Execution time in seconds for cycle graphs of up to 10k, 20k, and 100k nodes.}
    \label{F:CycleBenchmarks}
\end{figure}

For the complete graphs, linear and quadratic growths in the graph size correspond to quadratic and 4-th order growth in the node count, hence we replace the linear fit with a 4-th order fit. As the 4-th order fit is similar to the quadratic fit, we see indications of a quadratic growth in the node count, hence a linear growth in the graph size.
\begin{figure}[htbp]
    \centering
    \begin{subfigure}[t]{0.45\linewidth}
        \includegraphics[width=\linewidth]{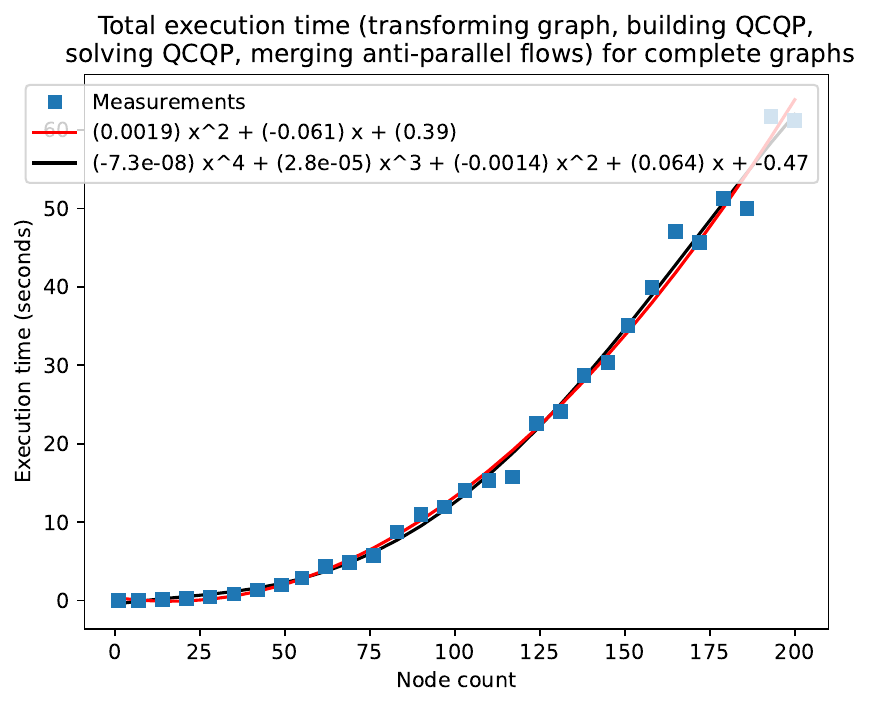}
    \end{subfigure}
    \begin{subfigure}[t]{0.45\linewidth}
        \includegraphics[width=\linewidth]{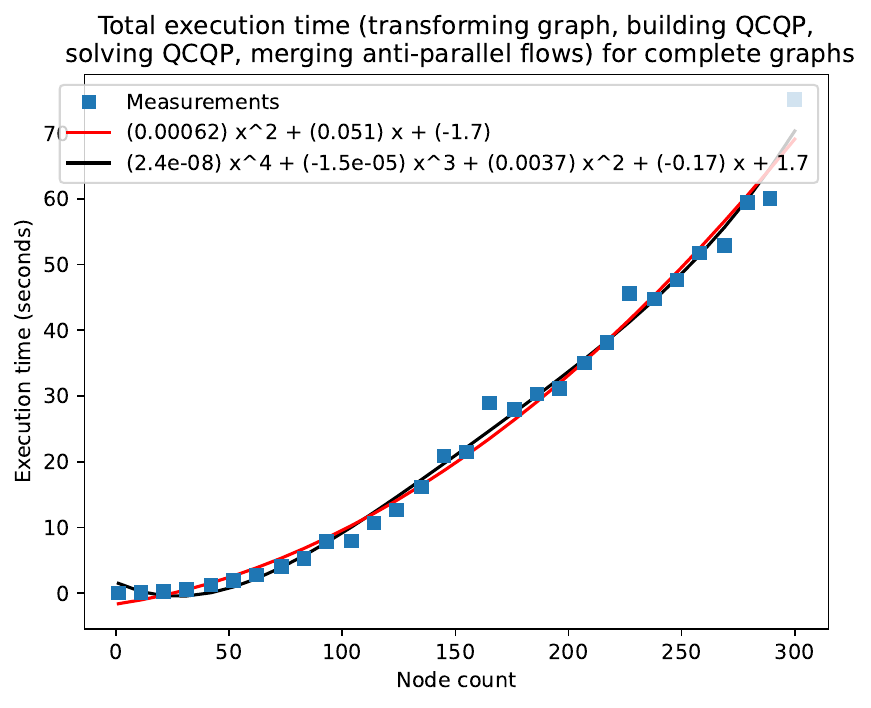}
    \end{subfigure}
    \caption{Execution time in seconds for complete graphs of up to 200 and 300 nodes. In both cases $N=30 \a M=5$.}
\end{figure}

To further investigate whether the execution time grows linearly or superlinearly in the graph size, 40 benchmarks with $n_{max}=8001$, $N=10$, and $M=3$ were run for cycle graphs, with an \ac{OLS} quadratic fit computed individually for each.
\begin{figure}[htbp]
    \centering
    \includegraphics[width=\linewidth]{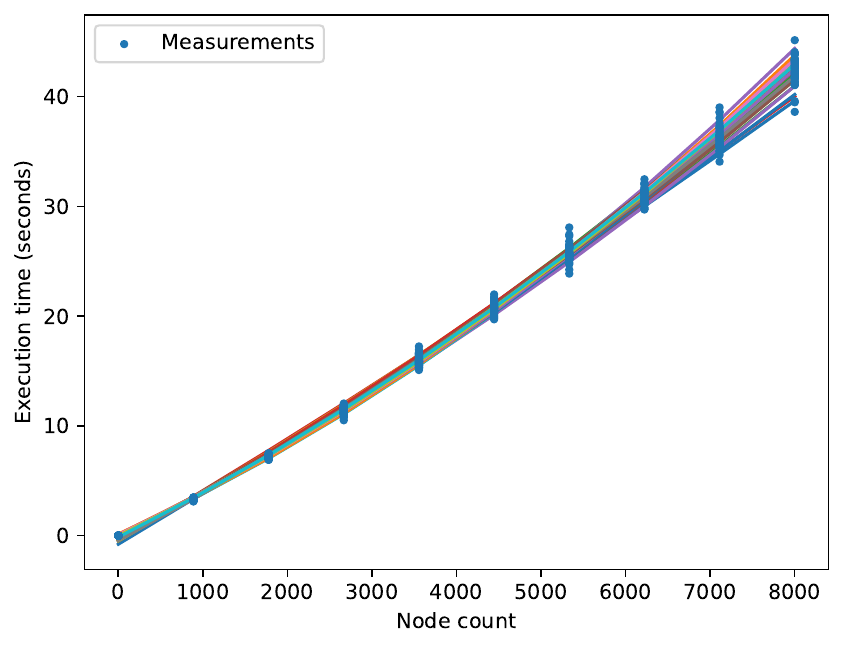}
    \caption{Scatter plot of all $40 \cdot 10 = 400$ observations, together with the 40 corresponding \acs{OLS} quadratic fits.}
    \label{F:Scatter}
\end{figure}
\begin{figure}[htbp]
    \centering
    \begin{subfigure}[t]{0.45\linewidth}
        \includegraphics[width=\linewidth]{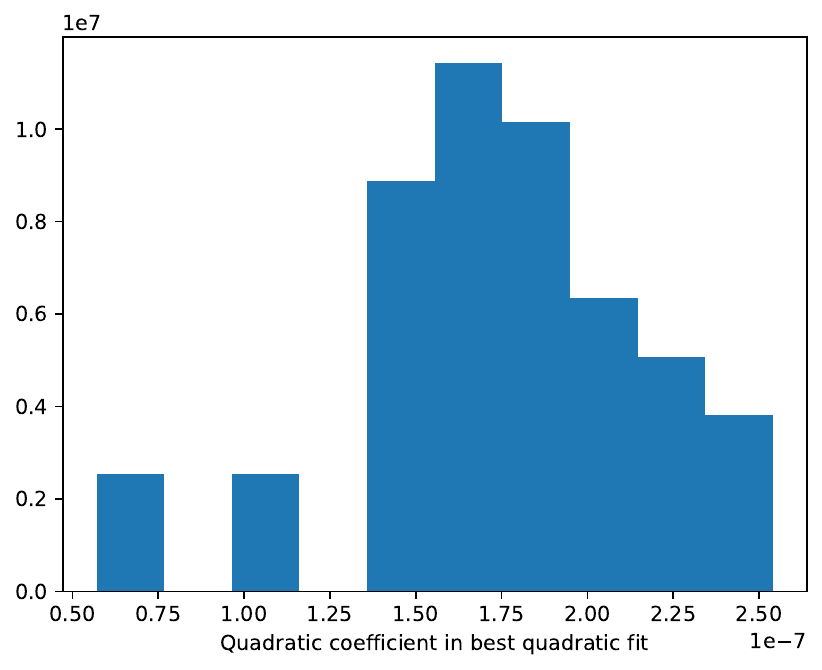}
        \caption{Distribution over these 40 benchmarks of the quadratic coefficient of the quadratic fits.}
    \end{subfigure}
    \hspace{0.5cm}
    \begin{subfigure}[t]{0.45\linewidth}
        \includegraphics[width=\linewidth]{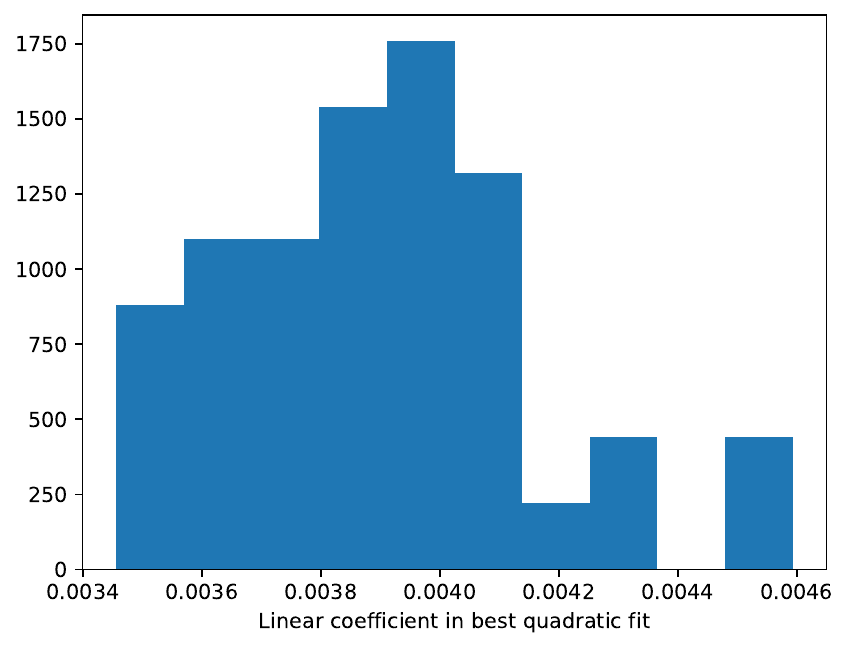}
        \caption{Distribution over these 40 benchmarks of the linear coefficient of the quadratic fits.}
    \end{subfigure}
\end{figure}
The corresponding two-sided 95\% confidence interval of the quadratic coefficient was then
\begin{equation}\label{E:ConfidenceInterval}
    [6.7 \cdot 10^{-8},\ 2.5 \cdot 10^{-7}],
\end{equation}
so a statistical test of significance would here reject the null hypothesis of a linear fit at the 0.025 level. Given how small the numbers in \eqref{E:ConfidenceInterval} are, one may be tempted to think of them as irrelevant in practice. However, at 8000 nodes this confidence interval corresponds to
\begin{equation}
    [4.3 \text{ seconds},\ 16 \text{ seconds}],
\end{equation}
which is
\begin{equation}
    [10\%,\ 38\%]
\end{equation}
of the average measurement of 42 seconds at 8001 nodes.

Alternatively, the true growth being quadratic superlinear is equivalent to the true quadratic coefficient being positive. If we for each experiment $i \in \{1,2,\ldots,40\}$ generate a decision variable $x_i$ that takes the value 1 if the quadratic coefficient is positive and 0 otherwise, then these 40 $x_i$ are i.i.d. $\Bern{\lambda}$ \acp{r.v.}, hence $s:=\sum_{i=1}^{40}$ is a $\Bin{40,\lambda}$ \ac{r.v.} With the unbiased prior $\lambda \sim \Unif{[0,1]}$, we get
\begin{equation}
    p(\lambda|s) \propto p(s|\lambda),
\end{equation}
hence $\lambda|s$ is a Beta($s+1,41-s$) \ac{r.v.} Since $s=40$, our posterior probability of the execution time growing quadratically and superlinearly in the node count for cycle graphs with less than 8000 nodes is
\begin{align}
    &p(x=1|s=40)\\
    &= \int_0^1{p(x=1|\lambda)p(\lambda|s=40)d\lambda} \\
    &= \int_0^1{\lambda p(\lambda|s=40)d\lambda}\\
    &= \int_0^1{\frac{\lambda^{41}}{\BetaConstant{41,1}}d\lambda}\\
    &= \int_0^1{\frac{\lambda^{41}}{\BetaConstant{42,1}}d\lambda} \frac{\BetaConstant{42,1}}{\BetaConstant{41,1}}\\
    &= \frac{\BetaConstant{42,1}}{\BetaConstant{41,1}}\\
    &\approx 0.98
\end{align}

We hence conclude that the growth rate in execution time for cycle graphs with less than 8000 nodes is superlinear. However, looking at Figure \ref{F:CycleBenchmarks}, we see indications that the quadratic coefficient shrinks as the span of node counts grows very large. Hence, the true growth rate in execution time may be subquadratic in $\abs{V}+\abs{E}$.

\subsubsection{Model assumptions}

In the above analysis, we implicitly assumed that \ac{OLS} gave us a best fit, hence that
\begin{equation*}
    y_i = p(x_i) + \varepsilon_i,
\end{equation*}
with $p$ some polynomial function and
\begin{equation*}
    \text{Cov}(\varepsilon) = \varepsilon\varepsilon^T = \sigma^2 I.
\end{equation*}
Both intuitively and empirically, multiplicative errors in the form of
\begin{equation*}
    y_i = p(x_i)\exp(\varepsilon_i)
\end{equation*}
seem closer to the truth. Hence, our quadratic fits were most likely not the truly best fits, as the model assumptions of \ac{OLS} were violated. However, as seen in \cref{F:Scatter}, the observed effect of variance on our observations was small compared to the observed effect of the node count. Moreover, in the above analysis our \ac{OLS} quadratic fits strongly indicated superlinearity. It is hence reasonable to assume that the truly best quadratic fits would be similar enough to the \ac{OLS} quadratic fits for us to still confidently arrive at the conclusion of superlinearity.

\chapter{Extending to arbitrary demands}\label{Ch:Extensions} We now wish to generalize beyond piecewise constant demand functions $\dfuns$. Assume $(g)$ is a \cref{P:ConstantOP} instance and $(f)$ equals $(g)$ in every way except its demands function $\dfuns^{(f)}$, which we permit to be any function such that
\begin{equation*}
    0 \leq \dfuns^{(f)} \leq \dfuns^{(g)}.
\end{equation*}

\section{Piecewise differentiable}\label{S:Differentiable}

We first restrict ourselves to every $\dfuns_\sink^{(f)}$ being differentiable with a bounded derivative:
\begin{equation*}
    \abs{\frac{d\dfuns_\sink^{(f)}(x)}{dx}} \leq \MaxDemandsSlope_\sink \in \R_{\geq 0}.
\end{equation*}

We want $\dfuns^{(g)}$ to be easily analyzable and a tight over-approximation of $\dfuns^{(f)}$. We thus split $[0,\dur]$ into $k$ equally large intervals $\TIntervals_1,\ldots,\TIntervals_k$, and change $\dfuns^{(g)}$ to
\begin{equation*}
    [0,\dur] \ni t \mapsto
    \begin{cases}
        \demands_{\sink 1} := \sup_{t \in \TIntervals_1}\dfuns_{\sink}^{(f)}(t), &\text{if } t \in \TIntervals_1 \subseteq [0,\dur], \\
        \vdots &\\
        \demands_{\sink k} := \sup_{t \in \TIntervals_k}\dfuns_{\sink}^{(f)}(t), &\text{if } t \in \TIntervals_k \subseteq [0,\dur],
    \end{cases}
    \qquad \forall{\sink \in \sinks}.
\end{equation*}
We moreover define the instance $(f')$ to be identical to $(f)$, except that
\begin{equation*}
    \dfuns^{(f')} := \dfuns^{(f)} + \MaxDemandsSlope_\sink \frac{\dur}{k}.
\end{equation*}
Then,
\begin{equation}
    \dfuns_\sink^{(f)} \leq \dfuns_\sink^{(g)} \leq \dfuns_\sink + \MaxDemandsSlope_\sink \frac{\dur}{k} = \dfuns_\sink^{(f')}, \ \forall{\sink \in \sinks}.
\end{equation}

Although the sensitivity analysis of \cref{S:Flows} was on instances of \cref{M:CP,M:QCQP}, the arguments it used applies equally well to $(f)$ and $(f')$. Hence, assume that $(f')$ is strictly feasible. Then,
\begin{equation}
    \OPT[(g)] - \OPT[(f)] \leq \OPT[(f')] - \OPT[(f)] \leq (\clin_{max} + 2\cquad_{max}\caps_{max}) \frac{\dur}{k} \left(\sum_{\sink \in \sinks}\MaxDemandsSlope_\sink\right) \prod_{a \in A}{\MinGrad_a^{-1}}.
\end{equation}
Thus, given some absolute error $\varepsilon > 0$, we can solve $(f)$ to within $\varepsilon$ by picking 
\begin{equation}\label{E:SufficientPieceCount}
    k := (\clin_{max} + 2\cquad_{max}\caps_{max}) \frac{\dur}{\varepsilon/2} \left(\sum_{\sink \in \sinks}\MaxDemandsSlope_\sink\right) \prod_{a \in A}{\MinGrad_a^{-1}}.
\end{equation}
and solving the corresponding $(g)$ to within $\varepsilon / 2$ absolute error. The latter is achieved by employing a slight modification of \cref{A:FPTAS}, where the modification is to call \cref{A:FeasibleEpsSolution} instead of \cref{A:ProgramFPTAS}. Note that the obtained $(g)$-feasible point will for $(f)$ contain waste in the sinks at all $t \in [0,\dur]$ where
\begin{equation*}
    \dfuns^{(f)}(t) \neq \dfuns^{(g)}(t).
\end{equation*}

To generalize our findings, we next loosen the restriction of differentiability to a requirement that all $\dfuns_\sink$ are piecewise differentiable. Note that we still require these derivatives to be upper-bounded by $\MaxDemandsSlope_\sink$ wherever they are defined. If $k$ and $\TIntervals$ are chosen as above, then splitting the $\TIntervals_1,\ldots,\TIntervals_k$ on the breakpoints of $\dfuns$ will cause the above analysis to hold within each differentiable piece of $\dfuns$. Hence, in this case we can get away with $k$ as determined by \eqref{E:SufficientPieceCount} plus the number of breakpoints of $\dfuns$.

\section{Piecewise convex}

We now instead restrict every $\dfuns_\sink$ to be convex.

Denote by $\PieceLengths_i$ the length of $\TIntervals_i$. Then, by convexity,
\begin{equation*}
    \sup_{t \in \TIntervals_i}\dfuns_{\sink}(t) = \max\left\{ \dfuns_{\sink}(\sum_{j \in [i-1]}\PieceLengths_j),\dfuns_{\sink}(\sum_{j \in [i]}\PieceLengths_j) \right\},
\end{equation*}
which is a convex function of $\PieceLengths=(\PieceLengths_1,\ldots,\PieceLengths_k)$. Hence, if we modify \cref{A:ToStatic,M:QCQP} accordingly, add to \cref{M:QCQP} the constraint
\begin{equation}\label{E:TotalPieceLengths}
    \sum_{i \in [k]} \PieceLengths_i = \dur,
\end{equation}
and optimize over $x,y,\PieceLengths$, then we still get a \ac{CP}. Given a fixed $k$, optimizing over this modified program means optimizing over all possible choices of $\PieceLengths$, hence giving a lower-cost point for $(f)$ than in \cref{S:Differentiable}, as there $\PieceLengths$ was fixed and optimization was only over $x,y$.

If each $\dfuns_\sink$ is instead piecewise convex, then we can manually assign the $\TIntervals_i$ to the convex pieces of $\dfuns$ and for each such piece add one equation like \eqref{E:TotalPieceLengths}. For each such equation the set we sum over in the left-hand side is made to correspond to the $\TIntervals$ that were manually assigned to this piece of $\dfuns$, while the right-hand side is set to the length of this piece of $\dfuns$. The resulting program is then once again a \ac{CP}.

\chapter{Discussion and future directions} \section{Model assumptions and comparison to DC-OPF}

In our AC-OPF simplification we made the assumptions that electric power travels instantaneously and is lost quadratically in each power line, and that we can control how it flows through the network. In practice, phase-shifting transformers allow us to control how power flows out of a node. However, as these are expensive, they are typically not installed on every node of the grid (see Figure \ref{F:PSTLocations}). Hence, assuming total control of how power flows through the network is indeed a simplification, thus likely leads our solutions to be infeasible for the actual AC-OPF problem. Moreover, our assumptions of zero transit times and quadratic losses can be seen as assuming that reactive power is negligible.
    \begin{figure}
        \centering
        \includegraphics[width=0.425\linewidth]{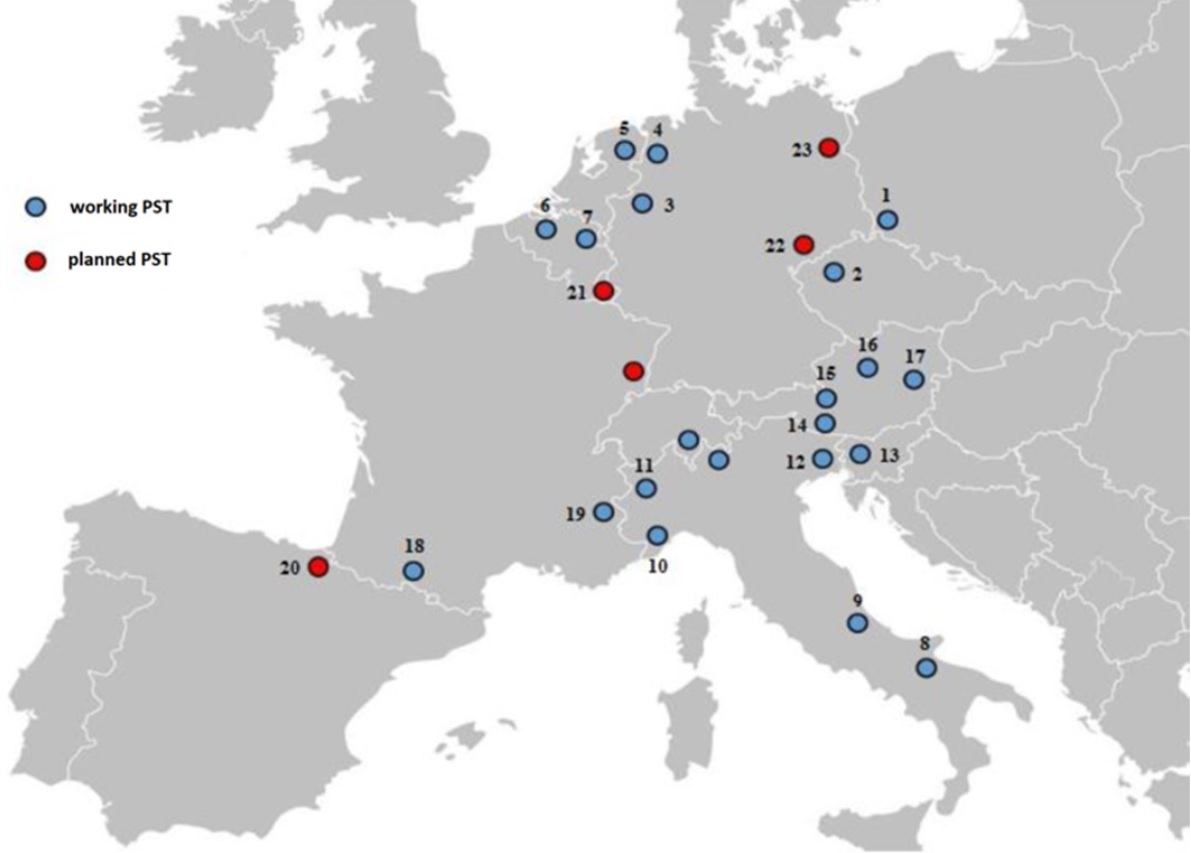}
        \caption{Locations of working and planned phase-shifting transformers in the Synchronous grid of Continental Europe. Source: \cite{zurek2023the}.}
        \label{F:PSTLocations}
    \end{figure}

A well-studied alternative simplification is the DC-OPF, where the power flow is instead determined by a set of physics equations, but under the assumptions that resistance and reactive power are zero, voltages are the same at every node, and the phase shifts between nodes are small enough to make the approximations $\sin(\theta_{im}) \approx \theta_{im}$ and $\cos(\theta_{im}) \approx 1$ \citep{baker2021solutions}. Due to its simplifications, DC-OPF is an \ac{LP} and can be solved much faster than the full AC-OPF. However, by these same assumptions and under light assumptions on the grid instance, DC-OPF will never give feasible points for the AC-OPF problem.

Our assumption of quadratic losses can be interpreted as reactive power being zero and power and current being linearly related. If we in our equations assume that every node has a transformer keeping its voltage at a constant level and that reactive power and reactance are negligible, then the power and current flowing out of that node will indeed be linearly related. Hence, the underlying assumptions of our simplification and the DC-OPF simplification are similar but different, with neither being a generalization of the other nor likely to provide feasible solutions to the corresponding AC-OPF problem. As reactance tends to be much larger than resistance in power transmission grids \citep{baker2021solutions}, the assumptions of the DC-OPF may be more realistic than our assumptions. However, by assuming zero resistance, the DC-OPF, unlike our model, may fail to punish long transmission distances, even when these are undesired in practice.

An interesting direction for future research would thus be to more rigorously compare the solution quality of our model to the solution quality of DC-OPF, by testing both on a range of instances and for each comparing their respective distances to the optimal AC-OPF solution. Hopefully this could then provide heuristics for deciding which model to use in a given instance.

\section{Algorithmic assumptions}

A caveat of \cref{A:FPTAS} is that we only know that it is an \ac{FPTAS} if $\log{\frac{1}{b^{(0)}}}$ is polynomial in the instance size. Intuitively it seems $\log{\frac{1}{b^{(0)}}}$ should always be polynomial in the instance size, but we did not prove this. Hence, one future research direction could be to either prove that $\log{\frac{1}{b^{(0)}}}$ is indeed always polynomial in the instance size, or alternatively obtain some explicit and testable assumptions on the instance that imply $\log{\frac{1}{b^{(0)}}}$ is polynomial in the instance size. Another possible research direction is to try to eliminate this dependence altogether by introducing some slack node $\slack$ with a very expensive direct arc to every node in the graph, as then $b^{(0)}$ can be made arbitrarily large by increasing the capacities on these arcs.

Finally, if the \cref{P:ConstantOP} instance is not strictly feasible, then \cref{A:FPTAS} will get stuck in Step \ref{A:FindEps} and run forever, as the \cref{A:FindEps} subroutine will never find a valid $\varepsilon'$. One interesting future research direction is thus to look into detecting infeasibility of our instance in polynomial time. An alternative direction is to instead look into exactly or approximately solving instances that are feasible but not strictly feasible.

\chapter{Summary} In this thesis, we looked at the min-cost concave dynamic generalized flow problem and its special case of the min-cost dynamic generalized flow with quadratic losses problem. To maximize our generality, we alternated between showing general results for the former (general) problem and special results for the latter (quadratic) problem. As a use case, we in \cref{Ch:Introduction} showed how this quadratic problem can model a simplified \ac{DOPF} problem.

The main contributions of this thesis were the derivation and implementation of an \ac{FPTAS} for a mild restriction of the quadratic problem. Secondary contributions were the derivation of a reduction from the general problem to a \ac{CP} and a corresponding sensitivity analysis of this \ac{CP}.

For the general problem, we in \cref{S:Reducing} first showed how it reduces to a \ac{CP} for which every optimal solution is a min-cost concave generalized flow. As a corollary of this, we showed how the quadratic problem, hence the simplified \ac{DOPF}, reduces to a convex \ac{QCQP} for which every optimal solution is a min-cost generalized flow with quadratic losses.

For our more general \ac{CP}, we then in \cref{S:Boyd} first reviewed classic \ac{CP} theory, outlining how the barrier method obtains a feasible $\varepsilon$-solution in polynomially many steps of Newton's method with backtracking line search. In \cref{S:Convex}, we further employed this \ac{CP} theory to show how local and global sensitivity analysis can be performed on the \ac{CP} given such a feasible $\varepsilon$-solution. Finally, in \cref{S:Flows}, we used the min-cost flow property of every optimal solution of our \ac{CP} to derive global sensitivity bounds explicit in the \ac{CP} instance.

For the special case of our convex \ac{QCQP}, we in \cref{S:Nesterov} paired this general sensitivity analysis with more specific convex \ac{QCQP} theory to obtain an \ac{FPTAS} under mild assumptions. Combined with the reductions of \cref{S:Reducing}, we hence arrived at an \ac{FPTAS} for a mild restriction of the min-cost dynamic generalized flow with quadratic losses problem.

A slightly modified version of this \ac{FPTAS} was then in \cref{Ch:Implementation} implemented in Python using NetworkX and gurobipy and benchmarked on cycle graphs, circular ladder graphs, and complete graphs. Through careful analysis we concluded that the execution time in practice grows slightly superlinearly and possibly subquadratically in the edge count of the input graph.

Finally, we generalized our algorithm to arbitrary demand functions, discussed how our \ac{DOPF} simplification compares to the classic direct current simplification, and pointed out some promising directions for future research.

\cleardoublepage
\phantomsection
\addcontentsline{toc}{chapter}{Bibliography}
\bibliography{bibliography}

\begin{thebibliography}{}

\bibitem[Aronson, 1989]{aronson1989survey}
Aronson, J. (1989).
\newblock A survey of dynamic network flows.
\newblock {\em Annals of Operations Research}, 20:1--66.

\bibitem[Baker, 2021]{baker2021solutions}
Baker, K. (2021).
\newblock Solutions of dc opf are never ac feasible.
\newblock In {\em Proceedings of the Twelfth ACM International Conference on Future Energy Systems}, e-Energy '21, page 264–268, New York, NY, USA. Association for Computing Machinery.

\bibitem[Bienstock and Verma, 2019]{bienstock2019strong}
Bienstock, D. and Verma, A. (2019).
\newblock Strong np-hardness of ac power flows feasibility.
\newblock {\em Operations Research Letters}, 47(6):494--501.

\bibitem[Boyd and Vandenberghe, 2004]{boyd2004convex}
Boyd, S. and Vandenberghe, L. (2004).
\newblock {\em Convex Optimization}.
\newblock {Cambridge University Press}.

\bibitem[Council, 2010]{nrc2010hidden}
Council, N.~R. (2010).
\newblock {\em Hidden Costs of Energy: Unpriced Consequences of Energy Production and Use}.
\newblock The National Academies Press, Washington, DC.

\bibitem[Gill et~al., 2014]{gill2014dynamic}
Gill, S., Kockar, I., and Ault, G.~W. (2014).
\newblock Dynamic optimal power flow for active distribution networks.
\newblock {\em IEEE Transactions on Power Systems}, 29(1):121--131.

\bibitem[Grötschel et~al., 1988]{grötschel1988geometric}
Grötschel, M., Lovász, L., and Schrijver, A. (1988).
\newblock {\em Geometric algorithms and combinatorial optimization}.
\newblock Springer.

\bibitem[Hern\'{a}ndez-Santib\'{a}\~{n}ez et~al., 2023]{hernandez2023pollution}
Hern\'{a}ndez-Santib\'{a}\~{n}ez, N., Jofr\'{e}, A., and Possama\"{\i}, D. (2023).
\newblock Pollution regulation for electricity generators in a transmission network.
\newblock {\em SIAM Journal on Control and Optimization}, 61(2):788--819.

\bibitem[IEA, 2023]{iea2023co2}
IEA (2023).
\newblock Co2 emissions in 2022 – analysis.

\bibitem[Nesterov and Nemirovskii, 1994]{nesterov1994interior}
Nesterov, Y. and Nemirovskii, A. (1994).
\newblock {\em Interior-Point Polynomial Algorithms in Convex Programming}.
\newblock Society for Industrial and Applied Mathematics.

\bibitem[Salman et~al., 2022]{salman2022optimal}
Salman, D., Kusaf, M., Elmi, Y.~K., and Almasri, A. (2022).
\newblock Optimal power systems planning for ieee-14 bus test system application.
\newblock In {\em 2022 10th International Conference on Smart Grid (icSmartGrid)}, pages 290--295.

\bibitem[von Meier, 2006]{meier2006electric}
von Meier, A. (2006).
\newblock {\em Electric Power Systems: A Conceptual Introduction}.
\newblock John Wiley \& Sons.

\bibitem[Wayne, 2002]{wayne2002polynomial}
Wayne, K.~D. (2002).
\newblock A polynomial combinatorial algorithm for generalized minimum cost flow.
\newblock {\em Mathematics of Operations Research}, 27(3):445--459.

\bibitem[Żurek and Przygrodzki, 2023]{zurek2023the}
Żurek, S. and Przygrodzki, M. (2023).
\newblock The use of a regulating transformer for shaping power flow in the power system.
\newblock {\em Energies}, 16(3):1--27.

\end{thebibliography}

\appendix
\setcounter{secnumdepth}{-1}
\chapter{Appendix}
\hypersetup{linkcolor=black}
\begin{acronym}
    \acro{QCQP}{quadratically constrained quadratic program}
\acro{CP}{convex program}
\acro{LP}{linear program}
\acro{FPTAS}{fully polynomial-time approximation scheme}
\acro{BFS}{breadth-first search}
\acro{DOPF}{dynamic optimal power flow}
\acro{UPGG}{undirected power grid graph}
\acro{SPGG}{simplified power grid graph}
\acro{GPGG}{generalized power grid graph}
\acro{r.v.}{random variable}
\acro{MCDGFWQL}{min-cost dynamic generalized flow with quadratic losses}
\acro{MCCDGF}{min-cost concave dynamic generalized flow}
\acro{OLS}{ordinary least squares}

\end{acronym}
\hypersetup{linkcolor=mycolor}

\definecolor{lg}{HTML}{eeeeee}

\begin{table}[htbp]
    \centering
    \caption{Standard notation.}
    \begin{tabular}{@{}ll@{}}
        \toprule
        Symbol & Description \\
        \midrule
        $\domain{f}$    &    The domain of the function $f$. \\
        $\image{f}$    &    The image of the function $f$. \\
        $\interior{Q}$    &    The interior of the set $Q$. \\
        
        \rowcolor{lg} $\Bern{p}$    &    A Bernoulli distribution with success probability $p$. \\
        \rowcolor{lg} $\Bin{n,p}$    &    A Binomial distribution with $n$ trials and $p$ success probability. \\
        \rowcolor{lg} $\Unif{Q}$    &    A Uniform distribution over the set $Q$. \\
        \rowcolor{lg} $\Beta{\alpha,\beta}$    &    A Beta distribution with shape parameters $\alpha,\beta$. \\
        \rowcolor{lg} $\BetaConstant{\alpha,\beta}$    &    A scaling factor used for the corresponding Beta distribution. \\

        $\N$    &    The set of all natural numbers: $\{0,1,2,\ldots\}$. \\
        $\R$    &    The set of all real numbers. \\
        $\extR$    &    $\R \cup \{-\infty,\infty\}$ \\
        $\Id$    &    The identity function $x \mapsto x$ \\

        \rowcolor{lg} $\abs{q}$    &    The absolute value of the scalar $q$. \\
        \rowcolor{lg} $\abs{I}$    &    The length of the interval $I$. \\
        \rowcolor{lg} $\abs{Q}$    &    The (finite) cardinality of the set Q. \\
        \rowcolor{lg} $\norm{w}$    &    The norm of the vector $w$. \\
        \rowcolor{lg} $w_{max}$    &    The largest entry of the vector $w$. \\
        \bottomrule
    \end{tabular}
\end{table}

\begin{table}[htbp]
    \centering
    \caption{Custom, recurring notation.}
    \begin{tabular}{@{}ll@{}}
        \toprule
        Symbol & Description \\
        \midrule
        $\sinks$    &    The nodes in our graph where flow can be absorbed (sinks). \\
        $\dur$    &    The width of the time span we optimize over: $[0,\dur]$. \\
        $\csupplies_s$    &    The maximum cumulative production in source $s$. \\
        $\dfuns_\sink$    &    A function mapping time $t \in [0,\dur]$ to the demand at sink $\sink$. \\
        $\demands$    &    The actual demand values found in the constant pieces of $\dfuns$. \\
        $\TIntervals$    &    The (maximal) pieces of $[0,T]$ on which $\dfuns$ is constant. \\
        $\prodcpus_s$    &    A function mapping production at source $s$ to marginal costs at $s$. \\
        $\supplies_s$    &    The maximum production (rate) at source $s$. \\
        $\caps_a$    &    The maximum flow capacity of arc $a$. \\
        $\outfuns_a$    &    The function mapping incoming to outgoing flow for arc $a$. \\
        
        \rowcolor{lg} $\InArcs{v}$    &    All the incoming arcs of node $v$. \\
        \rowcolor{lg} $\OutArcs{v}$    &    All the outgoing arcs of node $v$. \\
        \rowcolor{lg} $\vin{v}$    &    The flow coming into node $v$. \\
        \rowcolor{lg} $\vout{v}$    &    The flow going out of node $v$. \\
        
        $\bits{(f)}$    &    The number of bits required to encode the instance $(f)$. \\
        $\InsMap$    &    See \cref{D:Reduction}. \\
        $\SolMap$    &    See \cref{D:Reduction}. \\
        
        \rowcolor{lg} $\val[(f)]{z}$    &    The value of $z$ according to the objective function of $(f)$. \\
        \rowcolor{lg} $\OPT[(f)]$    &    The optimal value of the instance $(f)$. \\
        \rowcolor{lg} $\FunOPT(\omega,\rho)$    &    The optimal value of the $(\omega,\rho)$-perturbed $(f_{(\omega,\rho)})$. \\
        \rowcolor{lg} $\DeltaOPT{\omega,\rho}$    &    The increase in optimal value incurred by $(\omega,\rho)$-perturbing $(f)$. \\
        
        $\ssource$    &    The singular flow source in our \acs{CP}. See \cref{A:ToStatic}. \\
        $\costs_a$    &    The cost function of arc $a$ in our \acs{CP}. See \cref{A:ToStatic}. \\
        $\clin_a$    &    The linear coefficient of $\costs_a$ in our \acs{QCQP}. See \cref{M:QCQP}. \\
        $\cquad_a$    &    The quadratic coefficient of $\costs_a$ in our \acs{QCQP}. See \cref{M:QCQP}. \\
        $\outscaler_a$    &    linearly scales the flow on arc $a$ in our \acs{QCQP}. See \cref{M:QCQP}. \\
        \bottomrule
    \end{tabular}
\end{table}

\end{document}